\documentclass[final]{siamart0216}
\usepackage{amsfonts}
\usepackage{todonotes}
\usepackage{amsfonts,amsmath,amssymb}
\usepackage{mathrsfs,mathtools,stmaryrd,wasysym}
\usepackage{enumerate}
\usepackage{xspace,mydef} 
\usepackage{esint}
\usepackage{graphicx,psfrag}
\usepackage{hyperref}
\usepackage{pgf,tikz,pgfplots}
\usepackage{pstricks-add}
\usepackage{enumitem}
\usepackage{yfonts}

\usetikzlibrary{arrows}
\usepackage{rotating}
\DeclareMathAlphabet{\mathpzc}{OT1}{pzc}{m}{it}

\usepackage[notcite,notref]{showkeys} 
\newcommand{\EO}[1]{{\color{black}#1}}
\newcommand{\GC}[1]{{\color{black}#1}}

\newsiamremark{remark}{Remark}

\newcommand{\norm}[1]{{\left\vert\kern-0.25ex\left\vert\kern-0.25ex\left\vert #1 
\right\vert\kern-0.25ex\right\vert\kern-0.25ex\right\vert}}

\newcommand{\TheTitle}{
Numerical discretization of a Darcy--Forchheimer problem coupled with a singular heat equation}
\newcommand{\ShortTitle}{A Darcy--Forchheimer problem coupled with a heat equation}
\newcommand{\TheAuthors}{A.~Allendes, G.~Campa\~na and E.~Ot\'arola}

\headers{\ShortTitle}{\TheAuthors}

\title{{\TheTitle}\thanks{AA is partially supported by ANID through FONDECYT project 1210729. GC is partially supported by \GC{ANID--Subdirecci\'on de Capital Humano/Doctorado Nacional/2020--21200920}. EO is partially supported by ANID through FONDECYT project 1220156.}}

\author{Alejandro Allendes\thanks{Departamento de Matem\'atica, Universidad T\'ecnica Federico Santa Mar\'ia, Valpara\'iso, Chile.
(\email{alejandro.allendes@usm.cl}, \url{http://aallendes.mat.utfsm.cl/}).}
\and 
Gilberto Campa\~na\thanks{Departamento de Ciencias, Universidad T\'ecnica Federico Santa Mar\'ia, Valpara\'iso, Chile.
(\email{gilberto.campana@usm.cl}.}
\and
Enrique Ot\'arola\thanks{Departamento de Matem\'atica, Universidad T\'ecnica Federico Santa Mar\'ia, Valpara\'iso, Chile.
(\email{enrique.otarola@usm.cl}, \url{http://eotarola.mat.utfsm.cl/}).}   
}

\ifpdf
\hypersetup{
  pdftitle={\TheTitle},
  pdfauthor={\TheAuthors}
}
\fi

\date{Draft version of \today.}

\begin{document}

\maketitle

\begin{abstract}
In Lipschitz domains, we study a Darcy--Forchheimer problem coupled with a singular heat equation by a nonlinear forcing term depending on the temperature. By singular we mean that the heat source corresponds to a Dirac measure. We establish the existence of solutions for a model that allows a diffusion coefficient in the heat equation depending on the temperature. For such a model, we also propose a finite element discretization scheme and provide an a priori convergence analysis. In the case that the aforementioned diffusion coefficient is constant, we devise an a posteriori error estimator and investigate reliability and efficiency properties. We conclude by devising an adaptive loop based on the proposed error estimator and presenting numerical experiments.
\end{abstract}

\begin{keywords}
nonlinear equations, Darcy--Forchheimer problem, singular heat equation, Dirac measures, finite element approximations, convergence, a posteriori error estimates, adaptive loop.
\end{keywords}

\begin{AMS}
35R06,               
65N12,               
65N15, 		     
65N50,               
76S05.  	     
\end{AMS}


\section{Introduction}
\label{sec:intro}
Let $\Omega\subset\mathbb{R}^d$, with $d\in\{2,3\}$, be an open and bounded domain with Lipschitz boundary $\partial\Omega$. In this work\GC{,} we will be interested in the study of existence and approximation results for the temperature distribution of a fluid in a porous medium modeled by a \emph{singular} convection--diffusion equation coupled with a Darcy--Forchheimer problem. To be precise, we will study the following system of partial differential equations (PDEs):
\begin{equation}\label{eq:model}
\left\{
\begin{array}{rcll}
\nu\mathbf{u} + |\mathbf{u}|\mathbf{u}+ \nabla \mathsf{p} & = & \mathbf{f}(T) & \text{ in }\Omega,\\
\text{div}~\mathbf{u} & = & 0 & \text{ in }\Omega, \\
-\text{div}(\kappa(T)\nabla T)+\text{div}(\mathbf{u}~T) & = & \textnormal{g} & \text{ in }\Omega,
\end{array}
\right.
\end{equation}
supplemented with \GC{$\mathbf{u}\cdot \mathbf{n}=0$} on $\partial\Omega$ and $T=0$ on $\partial \Omega$. The unknowns of the system \eqref{eq:model} are the velocity field $\mathbf{u}$, the pressure $\mathsf{p}$, and the temperature $T$ of the fluid. The data of the model are the thermal diffusivity coefficient $\kappa$, the viscosity coefficient $\nu$, the external density force $\mathbf{f}$, and the external heat source $\textnormal{g}$. The coefficient $\kappa$ and the force $\mathbf{f}$ may depend nonlinearly on the temperature $T$. Finally, $\mathbf{n}$ denotes the unit outward normal vector on $\partial\Omega$ and $|\cdot|$ corresponds to the Euclidean norm. In our work we shall be particularly interested in the case that $\textnormal{g}=\delta_z$, where $\delta_z$ corresponds to the Dirac delta supported at the interior point $z\in\Omega$.

Darcy's law, $\mathbf{u} = - K \nabla p/\mu$, is a \emph{linear} relationship that describes the creeping flow of Newtonian fluids in porous media and is backed by years of experimental data \cite{MR2425154}. However, when the fluid velocity is sufficiently high, experimental evidence indicates that the relationship between the pressure gradient and Darcy velocity may be \emph{nonlinear}. On the basis of flow experimentation in \GC{sand packs}, in reference \cite{forchheimer1901wasserbewegung} Forchheimer suggested modifying Darcy's equation by incorporating a quadratic term depending on the velocity. This modification is known as the \emph{Darcy--Forchheimer} equation and finds several applications in engineering: it is used for predicting high velocity flow in porous media, especially in the vicinity of gas wells \cite{MR1103090}. It is thus no surprise that its analysis and approximation have been investigated by several authors; \EO{see \cite{MR2425154,MR2597810,MR2948707,articlepark,MR3068725,convergenceanalysisdarcysayah} for a priori error estimates for suitable mixed finite element approximations, \cite{https://doi.org/10.48550/arxiv.2202.11643} for the analysis of a posteriori error estimates, and \cite{MR4337455,10.48550/arXiv.2202.11643} for the numerical treatment of a coupled problem.}
On the other hand, the study of different models of incompressible \emph{nonisothermal fluid flow models} has become increasingly important for a variety of research areas in the fields of the natural sciences and engineering branches. This fact has been motivated by \EO{their} diverse applications in industry such as \GC{the} design of heat exchangers and chemical reactors, cooling processes, and polymer processing, to name a few. \EO{For advances regarding the numerical approximation of the so-called Boussinesq problem we refer the interested reader to the nonextensive list \cite{MR2111747,MR3845205,MR4053090,MR4265062,MR4080228,MR3454217,MR1681112,MR3232446}. Regarding the approximation of coupled problems involving the Navier--Stokes equations and a suitable temperature equation we refer the interested reader to \cite{MR1364404,MR1051838,MR2802085,MR3267163}. Within this context, we also mention \cite{ACFO:22,MR3802674,MR3523581,MR4036533,MR4041519,GMRB} for similar results when the Navier--Stokes equations are replaced by Darcy's equations.}

In this work\GC{,} we are interested in the analysis and discretization of the temperature distribution of a fluid in a porous medium modeled by a convection--diffusion equation coupled with a Darcy--Forchheimer problem. To the best of our knowledge, the only two articles that have numerically explored such a coupled problem are \cite{MR4448201} and \cite{MR4337455}. In contrast to these advances, we are particularly interested in the case that the forcing term of the stationary heat equation is singular: a Dirac measure. We begin our analysis by introducing a concept of weak solution within the spaces $\mathbf{L}^3(\Omega)$, $W^{1,\frac{3}{2}}(\Omega)\cap L_0^2(\Omega)$, and $W_0^{1,p}(\Omega)$ for the velocity, pressure, and temperature of the fluid, respectively. Here, $p$ is such that $2d/(d+1) -\epsilon< p < d/(d-1)$, for some $\epsilon >0$. Next, we show the existence of at least one solution for the coupled problem on the basis of the Leray--Schauder's fixed point theorem. We propose a finite element discretization scheme for the coupled problem: we use continuous and piecewise linear functions to the discretize the temperature and pressure and piecewise constant functions to discretize the velocity. We prove that such a numerical scheme \GC{admits at least one solution, which remains} bounded with respect to discretization, and the existence \GC{of a} subsequence that converges to a solution of the continuous problem. \GC{Under the assumption that $\kappa$ is constant,} we devise an a posteriori error estimator for the coupled problem that can be decomposed as the sum of two contributions: a contribution that accounts for the discretization of the Darcy--Forchheimer problem and another contribution that accounts for the discretization of the heat equation. \EO{We explore reliability estimates in two dimensions and \GC{local efficiency} bounds in two and three dimensions.}
We conclude our work by designing an adaptive finite element method on the basis of the devised error estimator and providing numerical experiments in convex and non-convex domains.

\GC{Our manuscript will be organized as follows}. We set notation and introduce preliminary material in section \ref{sec:notation}. In section \ref{sec:coupled_problem}, we present a weak formulation for \eqref{eq:model} and show existence of solutions. In section \ref{sec:fem}, we propose a finite element scheme and investigate convergence properties. In section \ref{sec:a_posteriori_anal}, we devise an a posteriori error estimator and study reliability and \GC{efficiency} properties. \GC{We conclude in section \ref{sec:numericalexperiments} by providing a series of numerical experiments that illustrate and go beyond the presented theory.}


\section{Notation and preliminaries} 
\label{sec:notation}
We begin by fixing notation and the setting in which we will operate.


\subsection{Notation}
We shall use standard notation for Lebesgue and Sobolev spaces. Spaces of vector valued functions and \GC{their} elements will be indicated with boldface. Since we will deal with \GC{an incompressible fluid}, we must indicate a way to make the pressure unique. To do so, we denote by $L^2_0(\Omega)$ the space of functions in $L^2(\Omega)$ that have \GC{zero averages.}

If $\mathcal{W}$ and $\mathcal{Z}$ are Banach function spaces, we write $\mathcal{W} \hookrightarrow\mathcal{Z}$ to denote that $\mathcal{W}$ is continuously embedded in $\mathcal{Z}$. We denote by $\mathcal{W}'$ and $\|\cdot\|_{\mathcal{W}}$ the dual and the norm of $\mathcal{W}$, respectively. Given $\mathfrak{p} \in(1, \infty)$, we denote by $\mathfrak{p}'$ its H\"older conjugate, i.e., the real number such that $1/\mathfrak{p} + 1/\mathfrak{p}' = 1$. The relation $a\lesssim b$ indicates that $a\leq Cb$, with a constant $C$ that depends neither on $a$, $b$ nor on the discretization parameters. The value of $C$ might change at each occurrence.  


\subsection{A Darcy--Forchheimer model}
\label{sec:a_Darcy_Forchheimer}

In this section, \GC{we briefly present some results regarding the well-posedness of the Darcy--Forchheimer problem
\begin{equation}\label{eq:DF_problem}
\nu\mathbf{u} +|\mathbf{u}|\mathbf{u}+ \nabla \mathsf{p}  =  \mathbf{f}\text{ in }\Omega, \quad\text{div}~\mathbf{u}  =  0  \text{ in }\Omega, \quad\mathbf{u}\cdot\mathbf{n} =  \mathbf{0}  \text{ on }\partial\Omega,
\end{equation}
where $\mathbf{f}$ denotes the external density force and $\nu$ is the viscosity coefficient. We say that $(\mathbf{u},\mathsf{p})\in \mathbf{X}\times M$ is a weak solution to problem \eqref{eq:DF_problem} if}
\begin{equation}\label{eq:forchheimer_darcy_problem}
\begin{array}{rcll}
\displaystyle\int_\Omega(\GC{\nu}\mathbf{u}+|\mathbf{u}|\mathbf{u}+\nabla\mathsf{p})\cdot\mathbf{v}\mathrm{d}\mathbf{x}& = & 	\displaystyle\int_\Omega\mathbf{f}\cdot \mathbf{v}\mathrm{d}\mathbf{x} \quad &\forall \mathbf{v}\in \mathbf{X},\\
\displaystyle\int_\Omega \nabla\mathsf{q}\cdot\mathbf{u}\mathrm{d}\mathbf{x} & = & 0 \quad &\forall \mathsf{q}\in M.
\end{array}
\end{equation}
Here, 
$
\mathbf{X}:=\mathbf{L}^3(\Omega),
$
$
M:=W^{1,\frac{3}{2}}(\Omega)\cap L_0^2(\Omega),
$
$\mathbf{f}$ belongs to $\mathbf{L}^{\frac{3}{2}}(\Omega)$\GC{,} and $\GC{\nu} \in C^{0,1}(\bar \Omega)$ is a strictly positive and bounded function which satisfies $\GC{\nu}_{-}\leq \GC{\nu}(\mathbf{x})\leq \GC{\nu}_{+}$ for every $\mathbf{x}\in \Omega$, \GC{where} $0 < \GC{\nu}_{-} \leq \GC{\nu}_{+}$. We immediately comment that, owing to our assumptions on data and definition of weak solution, all terms in \eqref{eq:forchheimer_darcy_problem} are meaningful. 

\GC{In what follows}, we will repeatedly make use of the fact that, on the spaces $\mathbf{X}$ and $M$, the following inf--sup condition holds \cite[Lemma 1]{MR2425154}:
\begin{equation}\label{eq:infsup}
\inf_{\mathsf{q}\in M} \sup_{\mathbf{v}\in \mathbf{X}}\dfrac{\int_\Omega \mathbf{v}\cdot \nabla \mathsf{q}\mathrm{d}\mathbf{x}}{\|\mathbf{v}\|_{\mathbf{L}^3(\Omega)}\|\nabla\mathsf{q}\|_{\mathbf{L}^{\frac{3}{2}}(\Omega)}}=1.
\end{equation}
This follows immediately from the dual representation of the norm $\| \cdot \|_{\mathbf{L}^{\frac{3}{2}}(\Omega)}$ \cite{MR2425154}.

To study problem \eqref{eq:forchheimer_darcy_problem}, it is convenient to introduce the map 
\begin{eqnarray}\label{eq:A_map_forchheimer}
\mathcal{A}:\mathbf{L}^3(\Omega)\to\mathbf{L}^{\frac{3}{2}}(\Omega),
\qquad
\mathbf{v} \mapsto  \mathcal{A}(\mathbf{v}):=\GC{\nu}\mathbf{v}+|\mathbf{v}|\mathbf{v}.
\end{eqnarray}
It is immediate that $\mathcal{A}$ maps $\mathbf{L}^3(\Omega)$ into $\mathbf{L}^{\frac{3}{2}}(\Omega)$. In addition, $\mathcal{A}$ is bounded on all bounded \GC{subsets} of $\mathbf{L}^3(\Omega)$ and $\mathcal{A}$ is monotone, coercive, and hemicontinuous in $\mathbf{L}^3(\Omega)$.

\begin{proposition}[Properties of $\mathcal{A}$]
$\mathcal{A}$ satisfies the following properties.
\begin{itemize} \item [(i)]\label{i} \textit{$\mathcal{A}$ is bounded on all bounded \GC{subsets} of $\mathbf{L}^3(\Omega)$:} If $\mathbf{v}\in \mathbf{L}^3(\Omega)$, then
 \begin{equation}\label{eq:prop1_A}
\|\mathcal{A}(\mathbf{v})\|_{\mathbf{L}^{\frac{3}{2}}(\Omega)}\leq \GC{\nu}_{+}\|\mathbf{v}\|_{\mathbf{L}^{\frac{3}{2}}(\Omega)}+\|\mathbf{v}\|_{\mathbf{L}^{3}(\Omega)}^2.
\end{equation}
\item [(ii)]\label{ii} \textit{$\mathcal{A}$ is monotone from $\mathbf{L}^3(\Omega)$ into $\mathbf{L}^{\frac{3}{2}}(\Omega)$}: If $\mathbf{v},\mathbf{w} \in \mathbf{L}^3(\Omega)$, then
\begin{equation}
\label{eq:prop2_A}
	\int_\Omega (\mathcal{A}(\mathbf{v})-\mathcal{A}(\mathbf{w}))\cdot (\mathbf{v}-\mathbf{w}) \mathrm{d}\mathbf{x}
	\geq \gamma_d \|\mathbf{v}-\mathbf{w}\|^3_{\mathbf{L}^3(\Omega)}
	+
	\GC{\nu}_{-}\|\mathbf{v}-\mathbf{w}\|^2_{\mathbf{L}^2(\Omega)},
\end{equation}
where $\gamma_d$ is a strictly positive constant that depends on the dimension $d$.
\item [(iii)]\label{iii} \textit{$\mathcal{A}$ is coercive in $\mathbf{L}^3(\Omega)$:}
\begin{equation}\label{eq:prop3_A}
\lim_{\|\mathbf{v}\|_{\mathbf{L}^3(\Omega)}\to+\infty}\frac{1}{\|\mathbf{v}\|_{\mathbf{L}^3(\Omega)}}\int_\Omega \mathcal{A}(\mathbf{v})\cdot \mathbf{v}\mathrm{d}\mathbf{x}=+\infty.
\end{equation}
\item [(iv)]\label{iv} \textit{$\mathcal{A}$ is hemicontinuous in $\mathbf{L}^3(\Omega)$:} Let $\mathbf{v},\mathbf{w}\in \mathbf{L}^3(\Omega)$. Then, the mapping
\begin{eqnarray}\label{eq:prop4_A}
t\to \int_\Omega \mathcal{A}(\mathbf{v}+t\mathbf{w})\cdot\mathbf{w}\mathrm{d}\mathbf{x},
\end{eqnarray}
is continuous from $\mathbb{R}$ into $\mathbb{R}$.
\item[(v)] Almost everywhere in $\Omega$, we have
\begin{eqnarray}\label{eq:prop5_A}
|\mathcal{A}(\mathbf{v})-\mathcal{A}(\mathbf{w})|\leq |\mathbf{v}-\mathbf{w}|(\GC{\nu}_{+} + |\mathbf{v}|+|\mathbf{w}|).
\end{eqnarray}
\end{itemize}
\label{eq:proposition_A}
\end{proposition}
\begin{proof}
The proofs of \eqref{eq:prop1_A} and \eqref{eq:prop5_A} are trivial. To obtain \eqref{eq:prop2_A}, we utilize the definition of $\mathcal{A}$ to write $(\mathcal{A}(\mathbf{v}) - \mathcal{A}(\mathbf{w}))\cdot(\mathbf{v}-\mathbf{w})$ and invoke the inequality of \cite[Lemma 4.4, Chapter I]{MR1230384} with $p=3$. Property \eqref{eq:prop3_A} follows from a direct calculation; see also the proof \cite[Theorem 2]{MR2425154}. Finally, the proof of the continuity of the map described in \eqref{eq:prop4_A} can be found in \cite[Proposition 3]{MR2425154}.
\end{proof}

Let us now introduce the functional spaces
\[
\mathbf{H}:= 
\left \{\mathbf{v}\in \mathbf{L}^3(\Omega):\textnormal{div}(\mathbf{v})\in L^{\frac{3d}{d+3}}(\Omega) \right\},
\]
endowed with the norm $\| \mathbf{v} \|_{\mathbf{H}}:= \| \mathbf{v} \|_{\mathbf{L}^3(\Omega)} + \| \mathrm{div}(\mathbf{v}) \|_\GC{{L^{\frac{3d}{d+3}}(\Omega)}}$, 
$\mathbf{H}_0:= \overline{\mathbf{C}_0^{\infty}(\Omega)}^{\mathbf{H}}$, and
\[
\mathbf{V}:= \left\{\mathbf{v} \in \mathbf{H}_0: \textrm{div}(\mathbf{v}) = 0 \textrm{ in } \Omega 
\right\}.
\]
The definition of $\mathbf{H}$ is motivated by the Sobolev embedding $W^{1,\frac{3}{2}}(\Omega) \hookrightarrow L^s(\Omega)$, which holds for every $s \leq s_{\star}:=3d/(2d-3)$. Observe that $s_{\star}'=3d/(d+3)$. The space $\mathbf{H}$ is a reflexive Banach space when equipped with 
$\| \cdot \|_{\mathbf{H}}$. In addition, 
$\mathbf{C}^{\infty}(\bar \Omega)$
is dense in $\mathbf{H}$.
With such a density result at hand, the following Green's formula can be derived \cite[formula (2.5)]{MR2425154}:
$
\int_{\Omega} \mathbf{v} \cdot \nabla q \mathrm{d}\mathbf{x} 
= 
-\int_{\Omega} q \mathrm{div} \mathbf{v} \mathrm{d}\mathbf{x}
+
\langle q, \mathbf{v}\cdot \mathbf{n}\rangle_{\partial \Omega}
$
for all $q \in M$ and $\mathbf{v} \in \mathbf{H}$. Finally, we mention the characterization $\mathbf{H}_0:= \left\{\mathbf{v} \in \mathbf{H}: \mathbf{v}\cdot\mathbf{n} = 0 \textrm{ on } \partial \Omega \right\}.$

Having introduced the space $\mathbf{V}$, the following result follows from \cite[Proposition 2]{MR2425154}: Problem \eqref{eq:forchheimer_darcy_problem} is equivalent to the following formulation: Find $\mathbf{u} \in \mathbf{V}$ such that
\begin{equation}\label{eq:Darcy_reduced}
\begin{array}{rcl}
\displaystyle\int_\Omega\mathcal{A}(\mathbf{u})\cdot\mathbf{v}\mathrm{d}\mathbf{x}& = & 	\displaystyle\int_\Omega\mathbf{f}\cdot \mathbf{v}\mathrm{d}\mathbf{x}\quad\forall\mathbf{v}\in\mathbf{V}.
\end{array}
\end{equation}

The well-posedness of problem \eqref{eq:Darcy_reduced} is as follows.

\begin{theorem}[well-posedness of a Darcy--Forchheimer problem]\label{thm:wp_darcy_forchheimer} If $\mathbf{f}\in \mathbf{L}^{\frac{3}{2}}(\Omega)$, then problem \eqref{eq:forchheimer_darcy_problem} admits a unique solution $(\mathbf{u},\mathsf{p})\in \mathbf{X}\times M$. In addition, we have
\begin{align}
\label{eq:bounds_forchheimer_velocity}
\|\mathbf{u}\|_{\mathbf{L}^3(\Omega)}^2& \leq \|\mathbf{f}\|_{\mathbf{L}^{\frac{3}{2}}(\Omega)},\\
\label{eq:bounds_forchheimer_pressure}
\|\nabla \mathsf{p}\|_{\mathbf{L}^{\frac{3}{2}}(\Omega)}&\leq \|\mathbf{f}\|_{\mathbf{L}^{\frac{3}{2}}(\Omega)}+ \GC{\nu}_{+}\|\mathbf{u}\|_{\mathbf{L}^{\frac{3}{2}}(\Omega)}+\|\mathbf{u}\|_{\mathbf{L}^{3}(\Omega)}^2,
\end{align}
with hidden constants that are independent of the data $\mathbf{f}$ and $\GC{\nu}$ and the solution $(\mathbf{u},\mathsf{p})$.
\end{theorem}
\begin{proof}
In view of the equivalence of problems \eqref{eq:forchheimer_darcy_problem} and \eqref{eq:Darcy_reduced}, we analyze problem \eqref{eq:Darcy_reduced}. To accomplish this task, we notice that Proposition \ref{eq:proposition_A} guarantees that $\mathcal{A}$ is bounded on all bounded \GC{subsets} of $\mathbf{L}^3(\Omega)$ and monotone, coercive, and hemicontinuous in $\mathbf{L}^3(\Omega)$. We can thus invoke the standard theory for monotone operators \cite[Theorem 2.18]{MR3014456}, \cite[Theorem 9.14-1]{MR3136903} to conclude the existence of a solution $\mathbf{u} \in \mathbf{L}^{3}(\Omega)$. Uniqueness follows from \GC{the} item (ii) in Proposition \ref{eq:proposition_A}. To obtain \eqref{eq:bounds_forchheimer_velocity}, it suffices to set $\mathbf{v} = \mathbf{u}$ in the first equation of problem \eqref{eq:forchheimer_darcy_problem}. The remaining estimate \eqref{eq:bounds_forchheimer_pressure} follows from the first equation of problem \eqref{eq:forchheimer_darcy_problem} and the inf--sup condition \eqref{eq:infsup}.
\end{proof}


\section{The coupled problem}
\label{sec:coupled_problem} 
In this section, we analyze the existence of weak solutions for problem \eqref{eq:model}.


\subsection{Main assumptions}\label{sec:main_assump}

We begin our studies by introducing the set of assumptions under which we will operate.

\begin{itemize}
\item \emph{External force:} The external density force is as follows \cite[Assumption 2.1]{MR4337455}:
\begin{equation}
 \mathbf{f}(\mathbf{x},s):=\mathbf{f}_0(\mathbf{x})+\mathbf{f}_1(s),
 \quad \GC{\mathbf{x}} \in \Omega,~ s \in \mathbb{R},
 \label{eq:external_density_force}
\end{equation}
where $\mathbf{f}_0\in \mathbf{L}^{\frac{3}{2}}(\Omega)$ and $\mathbf{f}_1$ is a Lipschitz--continuous function with constant $C_{\mathbf{f}}$ \GC{which satisfies $\mathbf{f}_1(0)=\GC{\mathbf{0}}$. In particular,} $|\mathbf{f}_1(s)|\leq C_{\mathbf{f}} |s|$ for every $s \in\mathbb{R}$.

\item \emph{Viscosity:} $\nu$ is a Lipschitz--continuous function with constant $C_{\mathcal{L}}$ and is such that 
$
0<\nu_{-}\leq \nu(\mathbf{x})\leq \nu_{+} 
$
for every $\mathbf{x} \in \Omega$. 

\item \emph{Diffusivity:} $\kappa \in C^{0,1}(\mathbb{R})$ is such that
$
0<\kappa_{-}\leq \kappa(s)\leq \kappa_{+} 
$
for every $s \in \mathbb{R}$.
\end{itemize}


\subsection{Weak formulation}
\label{sec:weak_solutions}

\GC{We introduce the following notion of weak solution.}


\begin{definition}[weak solution]
Let $z\in\Omega$ and $p<d/(d-1)$. We say that $(\mathbf{u},\mathsf{p},T)\in \mathbf{X}\times M\times Y$ is a weak solution to \eqref{eq:model} if
\begin{equation}\label{eq:modelweak}
\left\{
\begin{array}{rcll}
\displaystyle\int_\Omega\left(\nu\mathbf{u}\!+\!|\mathbf{u}|\mathbf{u}\!+\!\nabla\mathsf{p}\right)\cdot \mathbf{v}\mathrm{d}\mathbf{x}& = & 	\displaystyle\int_\Omega\mathbf{f}(T)\cdot \mathbf{v}\mathrm{d}\mathbf{x}  &\forall \mathbf{v}\in \mathbf{X},\\
\displaystyle\int_\Omega \nabla\mathsf{q}\cdot\mathbf{u}\mathrm{d}\mathbf{x} & = & 0  &\forall \mathsf{q}\in M,\\
\displaystyle\int_\Omega(\kappa(T)\nabla T\cdot \nabla S\!-\!T \mathbf{u}\cdot\nabla S)\mathrm{d}\mathbf{x} & = & \langle \delta_{z}, S\rangle  &\forall S\in W_0^{1,p'}(\Omega).
\end{array}
\right.
\end{equation}
Here, $\langle\cdot,\cdot\rangle$ denotes the duality pairing between $W_0^{1,p'}(\Omega)$ and $W^{-1,p}(\Omega)=(W_0^{1,p'}(\Omega))'$ and
$
 Y := W_0^{1,p}(\Omega).
$
\end{definition}

\EO{The following comments are now in order. The asymptotic behavior of solutions $\chi$ to second order elliptic problems with homogeneous Dirichlet boundary conditions and $\delta_z$ as a forcing term is dictated by $|\nabla \chi(\mathbf{x})| \approx |\mathbf{x} - z|^{-1}$ \EO{\cite[Theorem 3.3]{MR0223740}}. On the basis of a simple computation, this asymptotic behavior motivates us to seek for a temperature distribution within the space $W_0^{1,p}(\Omega)$ for $p<d/(d-1)$. On the other hand, we notice that,}
in view of the assumptions imposed on the problem data and the definition of weak solution, all terms in problem \eqref{eq:modelweak} are well-defined. In particular, the convective term can be controlled as follows:
\begin{equation}
\begin{aligned}
\label{eq:termconv}
\left|\int_{\Omega}T\mathbf{u}\cdot\nabla S d\textbf{x}\right|
&\leq \|\mathbf{u}\|_{\mathbf{L}^3(\Omega)}\|T\|_{L^{\frac{3p}{3-p}}(\Omega)}\|\nabla S\|_{\mathbf{L}^{p'}(\Omega)}\\
&\leq C_{e}
\|\mathbf{u}\|_{\mathbf{L}^3(\Omega)}\| \nabla T\|_{\mathbf{L}^{p}(\Omega)}\| \nabla S\|_{\mathbf{L}^{p'}(\Omega)},
\end{aligned}
\end{equation}
upon utilizing that $\mathbf{L}^{\frac{2p}{2-p}}(\Omega)\subset \mathbf{L}^{\frac{3p}{3-p}}(\Omega)$ and that $W_0^{1,p}(\Omega)\hookrightarrow L^{\frac{dp}{d-p}}(\Omega)$ \cite[Theorem 4.12, Case \textbf{C}]{MR2424078}, with $C_{e}$ being the best constant in the second embedding.

\subsection{A problem for the single variable $T$}
\EO{We follow \cite[Section 2.2]{MR3802674} and write}
\eqref{eq:modelweak} as a problem for the single variable $T$: For a given $T$, the first two equations in \eqref{eq:modelweak} correspond to a Darcy--Forchheimer system that, in view of Theorem \ref{thm:wp_darcy_forchheimer}, admits a unique solution $(\mathbf{u},\mathsf{p})\in \mathbf{X}\times M$. Observe that $\mathbf{f}(T) \in \mathbf{L}^{3/2}(\Omega)$. The variables $\mathbf{u}$ and $\mathsf{p}$ can thus be seen as functions depending on $T$, $(\mathbf{u},\mathsf{p})=(\mathbf{u}(T),\mathsf{p}(T))$, and  \eqref{eq:modelweak} is equivalent to the reduced formulation \cite[Section 2.2]{MR4337455}: Find $T\in W_0^{1,p}(\Omega)$ \EO{such that}
\begin{equation}\label{eq:heat_decoupled}
\int_\Omega(\kappa(T)\nabla T\cdot \nabla S - T \mathbf{u}(T)\cdot\nabla S)\mathrm{d}\mathbf{x}  =  \langle \delta_{z}, S\rangle \qquad \forall S\in W_0^{1,p'}(\Omega),
\end{equation}
where $p<d/(d-1)$ and $\mathbf{u}(T) \in \mathbf{X}$ denotes the velocity component of the pair $(\mathbf{u}(T),\mathsf{p}(T))$ that solves the problem: Find $(\mathbf{u}(T),\mathsf{p}(T))\in\mathbf{X}\times M$ \EO{such that}
\begin{equation}\label{eq:Darcy_decoupled}
\displaystyle\int_\Omega\left(\nu\mathbf{u}(T) +|\mathbf{u}(T)|\mathbf{u}(T)+\nabla\mathsf{p}(T)\right)\cdot \mathbf{v}\mathrm{d}\mathbf{x} =  \displaystyle\int_\Omega\mathbf{f}(T)\cdot \mathbf{v}\mathrm{d}\mathbf{x},
\quad
\displaystyle\int_\Omega \nabla\mathsf{q}\cdot\mathbf{u}(T)\mathrm{d}\mathbf{x}  =  0,
\end{equation}
for all $(\mathbf{v},\mathsf{q}) \in \mathbf{X} \times M$. Since $\mathbf{f}$ and $\nu$ satisfy the assumptions stated in section \ref{sec:main_assump}, \eqref{eq:Darcy_decoupled} admits a unique solution $(\mathbf{u}(T),\mathsf{p}(T))\in\mathbf{X}\times M$ which satisfies the bounds
\begin{align}
 \label{eq:bounds_forchheimercoupled_velocity}
 \|\mathbf{u}(T)\|_{\mathbf{L}^3(\Omega)}^2& \leq \|\mathbf{f}_0\|_{\mathbf{L}^{\frac{3}{2}}(\Omega)}+C_{\mathbf{f}}\mathfrak{C}_e
 \|\nabla T\|_{\mathbf{L}^{p}(\Omega)},\\
 \label{eq:bounds_forchheimercoupled_pressure}
 \|\nabla \mathsf{p} (T)\|_{\mathbf{L}^{\frac{3}{2}}(\Omega)} & \leq \|\mathbf{f}_0\|_{\mathbf{L}^{\frac{3}{2}}(\Omega)}\!+\!C_{\mathbf{f}}\mathfrak{C}_e\|\nabla T\|_{\mathbf{L}^{p}(\Omega)}\!+\!\nu_{+}\|\mathbf{u}(T)\|_{\mathbf{L}^{\frac{3}{2}}(\Omega)}\!+\!\|\mathbf{u}(T)\|_{\mathbf{L}^{3}(\Omega)}^2,
 \end{align}
where $\mathfrak{C}_e$ is the best constant in the Sobolev embedding $W_0^{1,p}(\Omega) \hookrightarrow L^{\frac{3}{2}}(\Omega)$ ($p>1$).

\subsection{A singular and stationary heat equation with convection} 
\label{subsection:heatequation} 
\EO{In this section, we review and extend to three dimensions some of the results obtained in \cite{ACFO:22}.} 

Let $\xi$ be a bounded and uniformly continuous function such that $0<\xi_{-}\leq \xi(\mathbf{x})\leq \xi_{+}$ for a.e. $\mathbf{x}\in \Omega$. With $\xi$ at hand, we introduce the following weak formulation for a singular and stationary heat equation with convection: Find $T\in W_0^{1,p}(\Omega)$ such that
\begin{equation}\label{eq:heat}
\displaystyle\int_\Omega\left(\xi\nabla T\cdot \nabla S - T\mathbf{v}\cdot\nabla S \right)\mathrm{d}\mathbf{x} =  \langle \delta_{z}, S\rangle \quad \forall S\in W_0^{1,p'}(\Omega).
\end{equation}
Here, $p$ is such that $\frac{2d}{d+1} -\epsilon< p < \frac{d}{d-1}$, for some $\epsilon >0$, $\mathbf{v} \in \mathbf{L}^3(\Omega)$, and $\frac{1}{p} + \frac{1}{p'}= 1$.

We begin our studies by providing a well-posedness result for the case $\mathbf{v}=\mathbf{0}$.

\begin{proposition}[case $\mathbf{v}=\mathbf{0}$]\label{prop:u=0}
Problem \eqref{eq:heat} with $\mathbf{v}=\mathbf{0}$ is well-posed. This, in particular, implies that
\begin{equation}\label{eq:inf_sup_Poisson}
\| \nabla R \|_{\mathbf{L}^{p}(\Omega)}\leq C_{\xi}\sup_{S\in W_0^{1,p'}(\Omega)}\dfrac{\int_\Omega\xi\nabla R\cdot \nabla S \mathrm{d}\mathbf{x}}{\|\nabla S\|_{\mathbf{L}^{p'}(\Omega)}}\qquad \forall R\in W_0^{1,p}(\Omega),
\end{equation}
with a constant $C_{\xi}$ that depends on $\xi$, $p$, and $\Omega$.
\label{pro:v=0}
\end{proposition}
\begin{proof}
The fact that problem \eqref{eq:heat} is well-posed with $\mathbf{v}=\mathbf{0}$ on Lipschitz domains follows from \cite[Theorem 0.5, item (a)]{MR1331981} and \cite{MR2141694}. \EO{The desired inf--sup condition \eqref{eq:inf_sup_Poisson} thus follows from a result due to Ne\v{c}as \cite[Theorem 2.6]{Guermond-Ern}.}
\end{proof}

We now present a well--posedness result for the case of nonzero convection.

\begin{proposition}[case $\mathbf{v}\neq\mathbf{0}$]\label{prop:wp_temp}
\EO{If
$
C_\xi C_e\|\mathbf{v}\|_{\mathbf{L}^3(\Omega)} \leq \alpha< 1,	
$
then problem \eqref{eq:heat} is well-posed. In particular, we have the stability bound
\begin{equation}\label{eq:estimatesheat}
\| \nabla T\|_{\mathbf{L}^{p}(\Omega)}\leq C_{\alpha} \|\delta_z\|_{W^{-1,p}(\Omega)}, \quad C_{\alpha}=\tfrac{C_\xi}{1-\alpha}, 
\qquad
\tfrac{2d}{d+1} -\epsilon< p < \tfrac{d}{d-1},
\end{equation}
for some $\epsilon > 0$. Here, $C_e$ denote the best constant in $W_0^{1,p}(\Omega)\hookrightarrow L^{\frac{dp}{d-p}}(\Omega)$.}
\end{proposition}
\begin{proof}
The proof follows from a simple adaption of the arguments elaborated in the proof of \cite[Proposition 3.3]{ACFO:22}.
\end{proof}


\subsection{Existence of solutions for the coupled problem} 
\GC{We now study the existence of solutions for \eqref{eq:heat_decoupled}--\eqref{eq:Darcy_decoupled} on the basis \EO{of} a fixed point argument.} To accomplish this task, we introduce the map $\mathcal{F}: W_0^{1,p}(\Omega)\to W_0^{1,p}(\Omega)$ defined by $\mathcal{F}(\theta)=\zeta$, where $\zeta$ denotes the solution to the following problem: Find $\zeta \in W_0^{1,p}(\Omega)$ such that
\begin{equation}\label{eq:operator_F}
\int_\Omega(\kappa(\theta)\nabla \zeta\cdot \nabla S - \zeta \mathbf{u}(\theta)\cdot\nabla S)\mathrm{d}\mathbf{x}  =  \langle \delta_{z}, S\rangle \quad \forall S\in W_0^{1,p'}(\Omega).
\end{equation}
The definition of $\zeta$ compromises solving, for a prescribed temperature $\theta \in W_0^{1,p}(\Omega)$, the Darcy--Forchheimer problem \eqref{eq:Darcy_decoupled} with $T$ being replaced by $\theta$. 

To present the following result, we define
$
\mathfrak{B}_{T}:= \{ \theta\in W_0^{1,p}(\Omega): \|\nabla\theta\|_{\mathbf{L}^{p}(\Omega)}\leq C_{\frac{1}{2}}\|\delta_{z}\|_{W^{-1,p}(\Omega)}  \}
$,
where $C_{\frac{1}{2}}$ is as in \eqref{eq:estimatesheat}, and the constant
\begin{equation}\label{def:constant_C}
\mathfrak{C}:=(2C_{e}C_{\kappa})^{-1}.
\end{equation}
Here, $C_e$ is the best constant in the embedding $W_0^{1,p}(\Omega)\hookrightarrow L^{\frac{dp}{d-p}}(\Omega)$ and $C_{\kappa}$ is \GC{the} constant involved in the inf-sup condition \eqref{eq:inf_sup_Poisson} with $\xi$ being replaced by $\kappa$.

\begin{lemma}[$\mathcal{F}(\mathfrak{B}_{T}) \subset \mathfrak{B}_{T}$] 
\label{lemma:welldefined}
\EO{If $p \in \big(\frac{2d}{d+1}-\epsilon, \frac{d}{d-1} \big)$, for some $\epsilon >0$, and
$
 \|\mathbf{f}_0\|_{\mathbf{L}^{\frac{3}{2}}(\Omega)}+C_{\mathbf{f}} \mathfrak{C_e} C_{\frac{1}{2}}\|\delta_z\|_{W^{-1,p}(\Omega)}\leq \mathfrak{C}^2,
$
then $\mathcal{F}$ is well-defined on $\mathfrak{B}_{T}$ and $\mathcal{F}(\mathfrak{B}_{T})\subset \mathfrak{B}_{T}$.}
\end{lemma}

\GC{\begin{proof}
The proof follows the same arguments as those developed in the proof of \cite[Lemma 5]{ACFO:22}. For brevity, we skip the details.
\end{proof}}


The following result is instrumental to show the compactness of the operator $\mathcal{F}$.

\begin{lemma}[convergence of sequences]
\label{thm:convergence_seq_u_theta}
Let $\{\theta_n\}_{n\geq0}$ be a sequence in $L^2(\Omega)$ such that $\theta_n \rightarrow \theta$ in $L^2(\Omega)$ as $n \uparrow \infty$. Then, as $n \uparrow \infty$,
\[
 \mathbf{u}(\theta_n) \rightarrow \mathbf{u}(\theta) \textrm{ in } \mathbf{L}^3(\Omega),
 \qquad
 \mathsf{p}(\theta_n) \rightharpoonup \mathsf{p}(\theta) \textrm{ in } W^{1,\frac{3}{2}}(\Omega)\cap L_0^2(\Omega).
\]
\end{lemma}

\begin{proof}
See \cite[Lemma 2.5]{MR4337455}.
\end{proof}

We are now ready to show existence of solutions. 

\begin{theorem}[existence of solutions]\label{thm:existence}
In the framework of Lemma \ref{lemma:welldefined}, there exists a solution $(\mathbf{u},\mathsf{p},T)\in \mathbf{X}\times M\times Y$ of problem \eqref{eq:modelweak}. In addition, we have that $T\in \mathfrak{B}_{T}$.
\end{theorem}

\begin{proof}
\EO{Notice that} $\mathfrak{B}_{T}$ is nonempty, closed, bounded, and convex. \EO{In addition,} in view of Lemma \ref{lemma:welldefined} we have that $\mathcal{F}(\mathfrak{B}_T)\subset \mathfrak{B}_T$. It thus suffices to prove the compactness of $\mathcal{F}$ to conclude the desired result on the basis of the Leray--Schauder fixed point theorem 
applied to \EO{$\mathcal{F}$; see} \cite[Theorem 8.8]{MR787404} and \cite[Theorem 9.12-1]{MR3136903}.

We recall that $p$ is such that $2d/(d+1) -\epsilon< p < d/(d-1)$ for some $\epsilon >0$. Let $\{\theta_{n}\}_{n\geq 0}\subset \mathfrak{B}_{T}$ be a sequence such that $\theta_{n}\rightharpoonup \theta$ \GC{in} $W_0^{1,p}(\Omega)$ as $n \uparrow \infty$. \EO{Since $\mathfrak{B}_{T}$ is weakly closed, $\theta \in \mathfrak{B}_T$.} Define $\zeta:=\mathcal{F}(\theta)$ and, for $n \in \mathbb{N}_0$, $\zeta_{n}:=\mathcal{F}(\theta_n)$. \EO{Let us} prove that
$\zeta_{n}\to \zeta$ in $W_0^{1,p}(\Omega)$ as $n\uparrow\infty$. Observe that $e_{\zeta,n}:=\zeta-\zeta_n$ verifies the relation
\begin{equation*}
\int_\Omega \left(\kappa(\theta_n)\nabla e_{\zeta,n}-e_{\zeta,n}\mathbf{u}(\theta)\right)\cdot \nabla S \mathrm{d}\mathbf{x}=\langle g_n, S\rangle 
\quad
\forall S \in W_0^{1,p'}(\Omega),
\end{equation*} 
where, for $S \in W_0^{1,p'}(\Omega)$,
$
\langle g_n, S\rangle:=\int_\Omega\left[ \zeta_n (\mathbf{u}(\theta)-\mathbf{u}(\theta_{n}))+(\kappa(\theta_n)-\kappa(\theta))\nabla \zeta \right]\cdot \nabla S\mathrm{d}\mathbf{x},
$
i.e., for $n \in \mathbb{N}_0$, $e_{\zeta,n}$ solves \eqref{eq:heat} with $\delta_z$, $\mathbf{v}$, and $\xi$ being replaced by $g_n$, $\mathbf{u}(\theta)$, and $\kappa(\theta_n)$, respectively. Proposition \ref{prop:wp_temp} with $\alpha = 1/2$ thus \GC{guarantees} that
\begin{equation}\label{eq:bound_e_zeta_n}
\|\nabla e_{\zeta,n}\|_{\mathbf{L}^{p}(\Omega)}\leq C_{\frac{1}{2}} \|g_n\|_{W^{-1,p}(\Omega)},
\end{equation}
because, as a consequence of \GC{\eqref{eq:bounds_forchheimercoupled_velocity},} $\|\mathbf{u}(\theta)\|_{\mathbf{L}^3(\Omega)}\leq \mathfrak{C}$. We now prove that $g_{n} \rightarrow 0$ \GC{in $W^{-1,p}(\Omega)$} as $n\uparrow \infty$. To accomplish this task, we analyze each of the terms \GC{involved} in the definition of $g_{n}$ separately. First, notice that the estimates in \eqref{eq:termconv} reveal that
\begin{align*}
\int_{\Omega}\zeta_n(\mathbf{u}(\theta)-\mathbf{u}(\theta_{n}))\cdot \nabla S\mathrm{d}\mathbf{x}\leq C_e \|\nabla \zeta_n\|_{\mathbf{L}^{p}(\Omega)}\|\mathbf{u}(\theta)-\mathbf{u}(\theta_{n})\|_{\mathbf{L}^3(\Omega)}\|\nabla S\|_{\mathbf{L}^{p'}(\Omega)}.
\end{align*}
Now, since $\theta_{n}\rightharpoonup \theta\in W_0^{1,p}(\Omega)$ as $n \uparrow \infty$, an application of the Rellich--Kondrachov theorem \cite[Theorem 6.3, Part I]{MR2424078} reveals that $\theta_{n}\rightarrow \theta$ in $L^{r}(\Omega)$, as $n \uparrow \infty$, for every $r$ such that $1 \leq r<dp/(d-p)$. Invoke Lemma \ref{thm:convergence_seq_u_theta}, upon observing that $dp/(d-p)>2$, to conclude that $\mathbf{u}(\theta_n)\to\mathbf{u}(\theta)$ in $\mathbf{L}^3(\Omega)$ as $n \uparrow \infty$. To control the remaining term in $g_{n}$, we observe that, since $\kappa$ is continuous and uniformly bounded and $\theta_{n}\rightarrow \theta$ in $L^{r}(\Omega)$, we have that $\kappa(\theta_n) \rightarrow \kappa(\theta)$ in $L^{r}(\Omega)$ as $n \uparrow \infty$. Consequently, $(\kappa(\theta_n)-\kappa(\theta))\nabla \zeta \to \mathbf{0}$ in $\mathbf{L}^p(\Omega)$. Therefore, in view of \eqref{eq:bound_e_zeta_n}, $\zeta_n\to \zeta$ in $W_0^{1,p}(\Omega)$ as $n\uparrow \infty$. We have thus proved that the weak convergence $\theta_n\rightharpoonup \theta$ in $W_0^{1,p}(\Omega)$ implies the strong one $\zeta_n\to \zeta$ in $W_0^{1,p}(\Omega)$ as $n\uparrow \infty$. This shows that $\mathcal{F}$ is compact and concludes the proof.
\end{proof}


\section{Finite element approximation}
\label{sec:fem}

In this section, we devise a finite element discretization scheme for problem \eqref{eq:modelweak} and analyze convergence properties.

\subsection{Basic ingredients and assumptions}
We denote by $\mathscr{T}_h = \{ K\}$ a conforming partition of $\bar{\Omega}$ into closed simplices $K$ with size $h_K = \text{diam}(K)$. Define $h:=\max \{ h_K: K \in \mathscr{T}_h \}$. We denote by $\mathbb{T} = \{\mathscr{T}_h \}_{h>0}$ a collection of conforming and shape regular meshes $\mathscr{T}_h$.  We define $\mathscr{S}$ as the set of interelement boundaries $\gamma$ of $\T_h$. For $K \in \T_h$, let $\mathscr{S}^{}_K$ denote the subset of $\mathscr{S}$ that contains the sides in $\mathscr{S}$ which are sides of $K$. We denote by $\mathcal{N}_{\gamma}$, for $\gamma \in \mathscr{S}$, the subset of $\T_h$ that contains the two elements that have $\gamma$ as a side. In addition, we define \GC{the} \emph{stars} or \emph{patches}
\begin{equation}
\label{eq:patch}
\mathcal{N}_K= \cup \{K' \in \T_h: \mathscr{S}_K \cap \mathscr{S}_{K'} \neq \emptyset \},
\quad 
\mathcal{N}_K^*= \cup \{K' \in \T_h: K \cap {K'} \neq \emptyset \}.
\end{equation}
In an abuse of notation, below we denote by $\mathcal{N}_K$, $\mathcal{N}_K^*$, \GC{and $\mathcal{N}_{\gamma}$} either the sets themselves or the union of its elements.

Given a mesh $\mathscr{T}_{h} \in \mathbb{T}$, we define 
$
Y_{h}:=\{S_{h}\in C(\bar{\Omega}): S_{h}|_K\in \mathbb{P}_{1}(K) \ \forall K\in \T_{h} \ \textrm{and} \ v|_{\partial \Omega} = 0\}.
$
Notice that, for each $h>0$, $Y_{h} \subset W_0^{1,p\prime}(\Omega)\subset W_0^{1,p}(\Omega)$. 

We denote by $I_h$ the Lagrange interpolation operator and immediately observe that, since $W_0^{1,p\prime}(\Omega)\hookrightarrow C(\bar{\Omega})$, $I_h$ is well-defined as map from $W_0^{1,p\prime}(\Omega)$ into $Y_h$. The following error estimate is classical \cite[Theorem 1.103]{Guermond-Ern}: for each $K\in {\T}_h$,
\begin{equation}\label{eq:interp_Lagrange}
\|S-I_h S\|_{L^{p\prime}(K)}
\lesssim
h_K\|\nabla S\|_{\mathbf{L}^{p\prime}(K)} \quad \forall S\in W_0^{1,p\prime}(K).
\end{equation}
With this estimate at hand, a trace identity yields, for $\gamma \in \mathscr{S}$, the estimate
\begin{equation}\label{eq:interp_Lagrange_side}
 \|S-I_h S\|_{L^{p\prime}(\gamma)}
\lesssim
h_{\gamma}^\frac{1}{p}\|\nabla S\|_{\mathbf{L}^{p\prime}(\mathcal{N}_{\gamma})} \quad \forall S\in W_0^{1,p\prime}(\mathcal{N}_{\gamma}).
\end{equation}

Regarding the approximation of the Darcy--Forchheimer model \eqref{eq:modelweak}, we introduce
\begin{align*}
\mathbf{X}_{h} &:= \left\{ \mathbf{v}_{h}\in \mathbf{L}^2(\Omega): \mathbf{v}_{h}|_{K}\in \mathbb{P}_{0}(K)^d \ \forall K\in \T_{h} \right \},\\
M_{h}&:=\{\mathsf{q}_{h}\in C(\bar{\Omega}): q_{h}|_K\in \mathbb{P}_{1}(K) \ \forall K\in \T_{h} \}\cap L_0^2(\Omega),\\
\mathbf{V}_{h} & := \left\{\mathbf{v}_{h}\in \mathbf{X}_{h}: (\nabla\mathsf{q}_{h}, \mathbf{v}_h)_{\mathbf{L}^2(\Omega)}
=0 \,\, \forall \mathsf{q}_{h}\in M_h \right\}.
\end{align*}
The spaces $\mathbf{X}_{h}$ and $M_{h}$ satisfy the following discrete inf--sup condition: there exists $\beta>0$, independent of the discretization parameter $h$, such that
\begin{equation}\label{eq:discrete_inf_sup}
\inf_{\mathsf{q}_h\in M_{h}} \sup_{\mathbf{v}_h\in\mathbf{X}_{h}}\dfrac{\int_\Omega  \mathbf{v}_h \cdot \nabla\mathsf{q}_h \mathrm{d}\mathbf{x}}{\|\mathbf{v}_h\|_{\mathbf{L}^3(\Omega)}\| \nabla \mathsf{q}_h\|_{\mathbf{L}^{\frac{3}{2}}(\Omega)}}\geq \beta;
\end{equation}
\GC{see} \cite[Proposition 4.1]{MR3068725} and \GC{\cite[estimate (3.4)]{https://doi.org/10.48550/arxiv.2202.11643}}.

Since it will be useful later, we also introduce the orthogonal projection $\mathbf{\Pi}_h: \mathbf{L}^{1}(\Omega) \to \mathbf{X}_h$. The operator $\mathbf{\Pi}_h$ satisfies, for any real number $r$ such that $r\geq 1$,
\begin{equation}
\label{eq:interp_Pih}
\lim_{h\to 0}\mathbf{\Pi}_h\mathbf{v}= \mathbf{v} \text{ strongly in }\mathbf{L}^r(\Omega)
\quad
\forall 
\mathbf{v}\in \mathbf{L}^r(\Omega);
\end{equation}
\GC{see \cite[Section 3.3]{MR2425154}. With this convergence result at hand, we immediately conclude the following result \cite[Corollary 1.109]{Guermond-Ern}: If $\mathbf{v}\in \mathbf{L}^3(\Omega)$, then}
\begin{equation}
\label{eq:density_result}
\lim_{h \rightarrow 0} \left( \inf_{\mathbf{v}_h \in \mathbf{X}_h} \| \mathbf{v} - \mathbf{v}_h\|_{\mathbf{L}^3(\Omega)} \right)  = 0.
\end{equation}
\subsection{The discrete coupled problem}

We introduce the following finite element approximation of \eqref{eq:modelweak}: Find $(\mathbf{u}_h, \mathsf{p}_h, T_h)\in \mathbf{X}_{h}\times M_h\times Y_h$ such that
\begin{equation}\label{eq:model_discrete}
\left\{
\begin{array}{rcll}
\displaystyle\int_\Omega\left(\nu\mathbf{u}_h+|\mathbf{u}_h|\mathbf{u}_h+\nabla\mathsf{p}_h\right)\cdot \mathbf{v}_h\mathrm{d}\mathbf{x}& = & 	\displaystyle\int_\Omega\mathbf{f}(T_h)\cdot \mathbf{v}_h\mathrm{d}\mathbf{x}  &\forall \mathbf{v}_h\in \mathbf{X}_h,\\
\displaystyle\int_\Omega \nabla\mathsf{q}_h\cdot\mathbf{u}_h\mathrm{d}\mathbf{x} & = & 0  &\forall \mathsf{q}_h\in M_h,\\
\displaystyle\int_\Omega(\kappa(T_h)\nabla T_h\cdot \nabla S_h\!-\!T_h \mathbf{u}_h\cdot\nabla S_h)\mathrm{d}\mathbf{x} & = & \langle \delta_{z}, S_h\rangle  &\forall S_h\in Y_h.
\end{array}
\right.
\end{equation}
We recall that we are operating under the assumptions stated in section \ref{sec:main_assump}.

In what follows, we prove that, for every $h>0$, problem \eqref{eq:model_discrete} always has a solution and that, as $h \rightarrow 0$, the sequence $\{ (\mathbf{u}_h, \mathsf{p}_h, T_h) \}_{h>0}$ converges weakly, up to subsequences, to a solution of the coupled problem \eqref{eq:modelweak}.

\subsection{A discrete stationary heat equation}
In this section, we review a well-posedness result for a suitable discretization of the singular heat equation with convection \eqref{eq:heat}. To present such a result, we will make use of \GC{the} following assumption: If $\xi \in C^{0,1}(\bar \Omega)$ is a strictly positive and bounded function which satisfies $0<\xi_{-} \leq \xi(\mathbf{x}) \leq \xi_{+}$ for every $\mathbf{x} \in \Omega$, then there exist $h_{\star}>0$ such that for all $0< h \leq h_{\star}$ and $R_h \in Y_h$, the following discrete inf-sup conditions holds:
\begin{equation}\label{eq:discrete_inf_sup_Poisson}
\|\nabla R_{h} \|_{\mathbf{L}^{p}(\Omega)}\leq \tilde{C}_{\xi}\sup_{S_{h}\in  Y_h}\dfrac{\int_\Omega \xi \nabla R_{h} \cdot \nabla S_{h} \mathrm{d}\mathbf{x}}{\|\nabla S_{h}\|_{\mathbf{L}^{p'}(\Omega)}}.
\end{equation}
Here, $p$ is such that $\frac{2d}{d+1} -\epsilon< p < \frac{d}{d-1}$, for some $\epsilon >0$\GC{,} and $\tilde{C}_{\xi}$ denotes a positive constant that is independent of $h$.

Given $\xi \in C^{0,1}(\bar \Omega)$ as above and $\mathbf{v}\in \mathbf{L}^3(\Omega)$, we introduce the following discrete version of problem \eqref{eq:heat}: Find $T_{h}\in Y_{h}$ such that
\begin{equation}\label{eq:discrete_heat}
\displaystyle\int_\Omega \left(\xi\nabla T_{h}\cdot \nabla S_{h} - T_{h}\mathbf{v}\cdot\nabla S_{h}\right)\mathrm{d}\mathbf{x} =  \langle \delta_{z}, S_{h}\rangle \quad \forall S_{h}\in Y_h.
\end{equation}

Under a suitable smallness assumption on the convective term, problem \eqref{eq:discrete_heat} always has a discrete solution. In addition, discrete solutions are uniformly bounded with respect to the discretization parameter $h$.

\begin{proposition}[well-posedness]\label{prop:disc_wp_temp}
There exists $h_{\star}>0$ such that, if
\begin{equation}\label{eq:disc_condicion}
\tilde{C}_\xi C_e\|\mathbf{v}\|_{\mathbf{L}^3(\Omega)} \leq \alpha< 1,
\end{equation}
then the discrete problem \eqref{eq:discrete_heat} is well-posed for all $0< h \leq h_{\star}$ whenever $2d/(d+1) -\epsilon< p < d/(d-1)$, for some $\epsilon >0$. In particular, we have the stability estimate
\begin{equation}\label{eq:disc_estimatesheat}
\| \nabla T_{h}\|_{\mathbf{L}^{p}(\Omega)}\leq \tilde{C}_{\alpha} \|\delta_z\|_{W^{-1,p}(\Omega)}, \quad \tilde{C}_{\alpha}=\tfrac{\tilde{C}_\xi}{1-\alpha},
\quad
\tfrac{2d}{d+1} -\epsilon< p < \tfrac{d}{d-1}.
\end{equation}
\end{proposition}
\begin{proof}
With the discrete inf--sup condition \eqref{eq:discrete_inf_sup_Poisson} at hand, the proof follows the same arguments as the ones developed in the proof of \GC{\cite[Proposition 4]{ACFO:22}}.
\end{proof}

\subsection{Existence of solutions for the discrete coupled problem} 
We now show existence of solutions for \eqref{eq:model_discrete} via a fixed point argument. As in the continuous case, we introduce, for each $h > 0$, the map $\mathcal{F}_{h}: Y_h \rightarrow Y_h$ by $\theta_{h} \mapsto \mathcal{F}(\theta_{h})=\zeta_{h}$, where 
\begin{equation}
\zeta_{h} \in Y_h:
\quad
\displaystyle\int_\Omega(\kappa(\theta_h)\nabla \zeta_h\cdot \nabla S_h - \zeta_h \mathbf{u}_h(\theta_h)\cdot\nabla S_h)\mathrm{d}\mathbf{x}  =  \langle \delta_{z}, S_h\rangle \quad \forall S_h\in Y_h.
\label{eq:discrete_zeta_h}
\end{equation}
The definition of $\zeta_h$ compromises solving the following discretization of a Darcy--Forchheimer model: Find $(\mathbf{u}_{h}(\theta_h),\mathsf{p}_{h}(\theta_h))\in \mathbf{X}_{h} \times M_{h}$ such that
\begin{equation}\label{eq:discrete_darcy_system}
\begin{array}{rcll}
\displaystyle\int_\Omega \left( \nu\mathbf{u}_h(\theta_h)+|\mathbf{u}_h(\theta_h)|\mathbf{u}_h(\theta_h)+\nabla\mathsf{p}_h(\theta_h) \right)\cdot \mathbf{v}_h\mathrm{d}\mathbf{x}&=& 	\displaystyle\int_\Omega\mathbf{f}(\theta_h)\cdot \mathbf{v}_h\mathrm{d}\mathbf{x}, 
\\
\displaystyle\int_\Omega \nabla\mathsf{q}_h\cdot\mathbf{u}_h(\theta_h)\mathrm{d}\mathbf{x} &=& 0, \quad 
\end{array}
\end{equation}
for all $(\mathbf{v}_h,q_h) \in \mathbf{X}_{h} \times M_{h}$. In view of the assumptions on the problem data stated in section \ref{sec:main_assump}, the discrete problem \eqref{eq:discrete_darcy_system} admits a unique solution. In addition,
\begin{align}
\label{eq:discrete_bound_Darcy_F_velocity}
\|\mathbf{u}_h(\theta_h)\|_{\mathbf{L}^3(\Omega)}^2&\leq \|\mathbf{f}_0\|_{\mathbf{L}^{\frac{3}{2}}(\Omega)}+C_{\mathbf{f}}\mathfrak{C}_e\|\nabla T_h\|_{\mathbf{L}^{p}(\Omega)},\\
\label{eq:discrete_bound_Darcy_F_preassure}
\|\nabla\mathsf{p}_h(\theta_h)\|_{\mathbf{L}^{\frac{3}{2}}(\Omega)}&\leq \|\mathbf{f}_0\|_{\mathbf{L}^{\frac{3}{2}}(\Omega)}\!+\!C_{\mathbf{f}}\mathfrak{C}_e\|\nabla T_h\|_{\mathbf{L}^{p}(\Omega)}+\nu_{+}\|\mathbf{u}_h\|_{\mathbf{L}^{\frac{3}{2}}(\Omega)}+\|\mathbf{u}_h\|_{\mathbf{L}^{3}(\Omega)}^2,
\end{align}
where $\mathfrak{C}_e$ is the best constant in the Sobolev embedding $W_0^{1,p}(\Omega) \hookrightarrow L^{\frac{3}{2}}(\Omega)$ ($p>1$).

To present the following result, we introduce the ball
\begin{equation*}
\mathfrak{B}_{T}^{h}:=\left\{\theta_{h}\in Y_{h}:~ \|\nabla\theta_{h}\|_{\mathbf{L}^{p}(\Omega)}\leq \tilde{C}_{\frac{1}{2}}\|\delta_{z}\|_{W^{-1,p}(\Omega)}\right\},
\end{equation*}
where $\tilde{C}_{\frac{1}{2}}$ is defined as in \eqref{eq:disc_estimatesheat} with $\alpha=\frac{1}{2}$. As a final ingredient, we also introduce 
\begin{equation}\label{def:constant_C_tilde}
\tilde{\mathfrak{C}}:=(2C_{e}\tilde{C}_{\kappa})^{-1},
\end{equation}
where $C_e$ is defined as in \eqref{eq:termconv} and $\tilde{C}_{\kappa}$ corresponds to the constant involved in the discrete inf--sup condition \eqref{eq:discrete_inf_sup_Poisson} with $\xi$ being replaced by $\kappa$.

\begin{lemma}[$\mathcal{F}_{h}: \mathfrak{B}_{T}^{h} \rightarrow \mathfrak{B}_{T}^{h}$] 
\label{lemma:welldefined_Fh}
If $\frac{2d}{d+1} -\epsilon< p < \frac{d}{d-1}$, for some $\epsilon >0$, and 
 \[
  \|\mathbf{f}_0\|_{\mathbf{L}^{\frac{3}{2}}(\Omega) }
  +
  C_{\mathbf{f}}\mathfrak{C}_e \GC{\tilde{C}_{\frac{1}{2}}}\|\delta_z\|_{W^{-1,p}(\Omega)}\leq \tilde{\mathfrak{C}}^2,
 \]
then, there exists $h_{\star}>0$ such that the map $\mathcal{F}_{h}$ is well-defined on $\mathfrak{B}_{T}^{h}$ for all $0<h \leq h_{\star}$. In addition, $\mathcal{F}_h(\mathfrak{B}_{T}^{h})\subset \mathfrak{B}_{T}^{h}$.
\end{lemma}
\begin{proof}
The proof follows similar arguments as the ones \GC{utilized} in the proof of Lemma \ref{lemma:welldefined}. Let $\theta_{h}\in \mathfrak{B}_{T}^{h}$. It is immediate that there exists a unique solution $(\mathbf{u}_{h}(\theta_h),\mathsf{p}_h(\theta_h))\in \mathbf{X}_{h}\times M_h$ to problem \eqref{eq:discrete_darcy_system}. In view of \eqref{eq:discrete_bound_Darcy_F_velocity}, \eqref{eq:disc_estimatesheat},
and \eqref{def:constant_C_tilde}, we conclude the following bound for the discrete velocity field $\mathbf{u}_{h}(\theta_h)$:
\begin{equation*}
\|\mathbf{u}_h(\theta_h)\|^2_{\mathbf{L}^3(\Omega)}
\leq\|\mathbf{f}_0\|_{\mathbf{L}^{\frac{3}{2}}(\Omega)}+C_{\mathbf{f}}\mathfrak{C}_e \GC{\tilde{C}_{\frac{1}{2}}}\|\delta_z\|_{W^{-1,p}(\Omega)}\leq \tilde{\mathfrak{C}}^2= (2C_{e}\tilde{C}_{\kappa})^{-2}.
\end{equation*}
Consequently, $\tilde{C}_{\kappa} C_{e}\|\mathbf{u}_{h}(\theta_h)\|_{\mathbf{L}^3(\Omega)} \leq 1/2$, i.e., $\mathbf{u}_{h}(\theta_h)$ satisfies \eqref{eq:disc_condicion} with $\alpha = 1/2$. We are thus in position to utilize the results of Proposition \ref{prop:disc_wp_temp} to obtain the existence of a unique $\zeta_h \in Y_{h}$ solving \eqref{eq:discrete_zeta_h}. In addition, $\zeta_{h} \in\mathfrak{B}^{h}_{T}$. This concludes the proof.
\end{proof}

We now study the existence of discrete solutions.

\begin{theorem}[existence]\label{thm:existence_discrete}
In the framework of Lemma \ref{lemma:welldefined_Fh}, there exists $h_{\star}>0$ such that \eqref{eq:model_discrete} admits a discrete solution $(\mathbf{u}_h,\mathsf{p}_h,T_h)\in \mathbf{X}_{h} \times M_{h}\times Y_{h}$ for all $0<h \leq h_{\star}$ whenever $2d/(d+1) -\epsilon< p < d/(d-1)$, for some $\epsilon >0$. In addition, $T_{h} \in \mathfrak{B}_{T}^{h}$.
\end{theorem}
\begin{proof}
Since $\emptyset \neq \mathfrak{B}_{T}^{h} \subset Y_h$ is compact and convex and $\mathcal{F}_h$ is continuous, which follows from similar arguments to the ones used in the proof of Theorem \ref{thm:existence}, we apply Brouwer's fixed point theorem \cite[Theorem 3.2]{MR787404} to obtain the existence of a solution.
\end{proof}


\subsection{Convergence}

We present the following convergence result.

\begin{theorem}[convergence]\label{thm:convergence} 
Let $\mathfrak{C}$ and $\tilde{\mathfrak{C}}$ be the constants defined in \eqref{def:constant_C} and \eqref{def:constant_C_tilde}, respectively. Let $p$ be such that $2d/(d+1) -\epsilon< p < d/(d-1)$ for some $\epsilon >0$. If 
\begin{equation} 
\|\mathbf{f}_0\|_{\mathbf{L}^{\frac{3}{2}}(\Omega)}+C_{\mathbf{f}}\mathfrak{C}_e C_{\frac{1}{2}}\|\delta_z\|_{W^{-1,p}(\Omega)}\leq \GC{\mathfrak{C}^2},
\,
\GC{\|\mathbf{f}_0\|_{\mathbf{L}^{\frac{3}{2}}(\Omega)}+C_{\mathbf{f}}\mathfrak{C}_e \tilde{C_{\frac{1}{2}}}\|\delta_z\|_{W^{-1,p}(\Omega)} \leq \tilde{\mathfrak{C}}^2},
 \label{eq:assump_convergence}
\end{equation}
\GC{and, for $d = 3$, $\mathbf{u} \in \mathbf{L}^{3+\varepsilon}(\Omega)$}, where $\varepsilon>0$ is \GC{arbitrarily} small, then there exists $h_{\star} >0$ and $\{(\mathbf{u}_h, \mathsf{p}_h, T_h)\}_{0<h \leq h_{\star}}$, a nonrelabelared subsequence, such that $\mathbf{u}_h \to \mathbf{u}$ in $\mathbf{X}$, $\mathsf{p}_h \rightharpoonup \mathsf{p}$ in $M$, and $T_h \rightharpoonup T$ in $Y$ as $h\downarrow 0$. In addition, $T \in Y$ solves \eqref{eq:heat_decoupled} and $(\mathbf{u},\mathsf{p})\in \mathbf{X}\times M$ solves \eqref{eq:Darcy_decoupled}.
\end{theorem}

\begin{proof}
The existence of a discrete solution $(\mathbf{u}_h,\mathsf{p}_h,T_h)$, for $h$ such that $0 < h \leq h_{\star}$, follows from Theorem \ref{thm:existence_discrete}. We now invoke the results of Lemma \ref{lemma:welldefined_Fh} and the discrete stability estimates
\eqref{eq:discrete_bound_Darcy_F_velocity} and \eqref{eq:discrete_bound_Darcy_F_preassure} to deduce that $\{(\mathbf{u}_h,\mathsf{p}_h,T_h)\}_{0<h \leq h_{\star}}$ is uniformly bounded in $\mathbf{X} \times M \times Y$ for $h$ such that $0 < h \leq h_{\star}$. Consequently, we conclude that, up to a subsequence if necessary, $(\mathbf{u}_h,\mathsf{p}_h,T_h) \rightharpoonup (\mathbf{u},\mathsf{p},T)$ in $\mathbf{X}\times M\times Y$, as $h\downarrow 0$, whenever $2d/(d+1) -\epsilon< p < d/(d-1)$, for some $\epsilon >0$.

We now follow \cite{MR2425154,MR4337455} and show, in several steps, that $(\mathbf{u},\mathsf{p},T) \in \mathbf{X}\times M \times Y$ solves system \eqref{eq:modelweak} or, equivalently, problems \eqref{eq:heat_decoupled} and \eqref{eq:Darcy_decoupled}. We begin the analysis by utilizing the monotonicity property \eqref{eq:prop2_A} of $\mathcal{A}$ to obtain
\begin{align}\label{eq:monotonicity_discrete}
	\int_\Omega \left(\mathcal{A}(\mathbf{u}_h)-\mathcal{A}(\mathbf{v}_h)\right)\cdot (\mathbf{u}_h-\mathbf{v}_h)\mathrm{d}\mathbf{x}\geq 0 \quad \forall \mathbf{v}_h \in \mathbf{V}_h.
\end{align}
Set $\mathbf{u}_h-\mathbf{v}_h \in \mathbf{V}_h$ as a test function in the first equation of problem \eqref{eq:model_discrete} and utilize \eqref{eq:monotonicity_discrete} to arrive at
\begin{align}\label{eq:convergence_monotonicity}
	\int_\Omega \mathcal{A}(\mathbf{v}_h)\cdot (\mathbf{u}_h-\mathbf{v}_h)\mathrm{d}\mathbf{x}\leq \int_\Omega \mathbf{f}(T_h)\cdot (\mathbf{u}_h-\mathbf{v}_h)\mathrm{d}\mathbf{x} \quad \forall \mathbf{v}_h \in \mathbf{V}_h.
\end{align}
Let us now take the limit as $h \downarrow 0$ in the previous inequality. Set $\mathbf{v}_h = \boldsymbol{\Pi}_{h}(\mathbf{v}) \in \mathbf{V}_h$, for an arbitrary $\mathbf{v} \in \mathbf{V}$. Invoke \eqref{eq:prop5_A} and the convergence property \eqref{eq:interp_Pih} to immediately deduce that $\|\mathcal{A}(\mathbf{v}_h)-\mathcal{A}(\mathbf{v})\|_{\mathbf{L}^{3/2}(\Omega)}\to 0$ as $h\downarrow 0$. Since $\mathbf{u}_h-\mathbf{v}_h\ \rightharpoonup  \mathbf{u}-\mathbf{v}$ in $\mathbf{L}^3(\Omega)$ as $h\downarrow 0$, we take the limit as $h\downarrow 0$ in \eqref{eq:convergence_monotonicity} to obtain, for an \GC{arbitrary} $\mathbf{v}\in \mathbf{V}$,
\begin{align}\label{eq:convergence_monotonicity_limit}
	\int_\Omega \mathcal{A}(\mathbf{v})\cdot (\mathbf{u}-\mathbf{v})\mathrm{d}\mathbf{x}\leq \int_\Omega \mathbf{f}(T)\cdot (\mathbf{u}-\mathbf{v})\mathrm{d}\mathbf{x},
\end{align}
where we have also used the compactness of the embedding $W_0^{1,p}(\Omega)\hookrightarrow L^{q}(\Omega)$ for $q<q_{\star}=dp/(d-p)$ \cite[Theorem 6.3, Part \textbf{I}]{MR2424078}, combined with the fact that $q_{\star}>3/2$, to deduce that $\|\mathbf{f}(T)-\mathbf{f}(T_h)\|_{\mathbf{L}^{3/2}(\Omega)}\leq C_{\mathbf{f}}\|T-T_h\|_{L^{3/2}(\Omega)}\to 0$ as $h\downarrow 0$.
With the inequality \eqref{eq:convergence_monotonicity_limit} at hand, we invoke \cite[Lemma 2.13]{MR3014456} to conclude that
\begin{align}
	\int_\Omega \mathcal{A}(\mathbf{u})\cdot \mathbf{v}\mathrm{d}\mathbf{x}= \int_\Omega \mathbf{f}(T)\cdot \mathbf{v}\mathrm{d}\mathbf{x}\quad\forall \mathbf{v}\in \mathbf{V}.
\end{align}
The inf--sup condition \eqref{eq:infsup} imply the existence of a unique $\varrho \in M$ such that $(\mathbf{u},\varrho) \in \mathbf{X} \times M$ solves problem \eqref{eq:Darcy_decoupled} \cite[Proposition 2]{MR2425154}.

We now show that $\mathbf{u}_h$ converges strongly to $\mathbf{u}$ in $\mathbf{L}^3(\Omega)$ as $h \downarrow 0$. To accomplish this task, we first consider the monotonicity property \eqref{eq:prop2_A} with $\mathbf{v}=\mathbf{u}_h$ and $\mathbf{w}=\mathbf{\Pi}_h(\mathbf{u})$: 
\begin{align*}
\|\mathbf{u}_h-\mathbf{\Pi}_h(\mathbf{u})\|_{\mathbf{L}^3(\Omega)}^3\lesssim\int_{\Omega}(\mathcal{A}(\mathbf{u}_h)-\mathcal{A}(\mathbf{\Pi}_h(\mathbf{u})))\cdot \left(\mathbf{u}_h-\mathbf{\Pi}_h(\mathbf{u})\right)  \mathrm{d}\mathbf{x}.
\end{align*}
Set $\mathbf{v}_h = \mathbf{u}_h-\mathbf{\Pi}_h(\mathbf{u})$ in the first equation of problem \eqref{eq:model_discrete} to thus obtain
\begin{align}\label{eq:eqauxconvergence}
\|\mathbf{u}_h\!-\!\mathbf{\Pi}_h(\mathbf{u})\|_{\mathbf{L}^3(\Omega)}^3\!\lesssim\!\int_{\Omega}\mathbf{f}(T_h)\GC{\cdot}(\mathbf{u}_h\!-\!\mathbf{\Pi}_h(\mathbf{u})) \mathrm{d}\mathbf{x}\!-\!\int_{\Omega}\mathcal{A}(\mathbf{\Pi}_h(\mathbf{u}))\!\cdot\!\left(\mathbf{u}_h\!-\!\mathbf{\Pi}_h(\mathbf{u})\right)  \mathrm{d}\mathbf{x},
\end{align}
upon utilizing that $\int_{\Omega} \nabla \mathsf{p}_h \cdot (\mathbf{u}_h-\mathbf{\Pi}_h(\mathbf{u}))\mathrm{d}\mathbf{x} = 0$. We now utilize the fact that $(\mathbf{u},\varrho) \in \mathbf{X} \times M$ solves problem \eqref{eq:Darcy_decoupled}
to arrive at
\begin{multline}\label{eq:eqaux02}
\|\mathbf{u}_h-\mathbf{\Pi}_h(\mathbf{u})\|_{\mathbf{L}^3(\Omega)}^3\!\lesssim
\left| 
\int_{\Omega} 
\left( \mathcal{A}(\mathbf{u}) - \mathcal{A}(\mathbf{\Pi}_h(\mathbf{u})) \right) \cdot \left(\mathbf{u}_h\!-\!\mathbf{\Pi}_h(\mathbf{u})\right) \mathrm{d} \mathbf{x}
\right|
\\
\left| \int_{\Omega} (\GC{\mathbf{f}}(T_h) - \GC{\mathbf{f}}(T)) \cdot ( \mathbf{u}_h\!-\!\mathbf{\Pi}_h(\mathbf{u}) ) \mathrm{d}\mathbf{x} \right|
+ \left| \int \nabla \varrho \cdot ( \mathbf{u}_h\!-\!\mathbf{\Pi}_h(\mathbf{u}) ) \mathrm{d}\mathbf{x} \right| \rightarrow 0,
\quad
h \downarrow 0,
\end{multline}
upon utilizing that $\GC{\mathbf{f}}(T_h) \rightarrow \GC{\mathbf{f}}(T)$ and $\mathcal{A}(\mathbf{\Pi}_h(\mathbf{u})) \rightarrow \mathcal{A}(\mathbf{u})$ in $\mathbf{L}^{\frac{3}{2}}(\Omega)$, as $h \downarrow 0$, and that
$\mathbf{u}_h-\mathbf{\Pi}_h(\mathbf{u}) \rightharpoonup 0$ in 
$\mathbf{L}^{3}(\Omega)$ as $h \downarrow 0$. A basic application of a triangle inequality thus implies that $\mathbf{u}_h\to\mathbf{u}$ in $\mathbf{L}^3(\Omega)$ as $h\downarrow 0$.

We now show that $\varrho = \mathsf{p}$. To accomplish this task, we subtract the first equation of the problem that $(\mathbf{u},\varrho)$ solves (problem \eqref{eq:Darcy_decoupled} with $\mathbf{u}(T)$ and $\GC{\mathsf{p}}(T)$ being replaced by $\mathbf{u}$ and $\varrho$, respectively) with \GC{$\boldsymbol{\Pi}_{h}(\mathbf{v})$, for an arbitrary function $\mathbf{v}$, as a test function}, from the first equation of \eqref{eq:model_discrete} to obtain
\begin{align}
\int_{\Omega} \left(\mathcal{A}(\mathbf{u}_h)-\mathcal{A}(\mathbf{u})\right)\cdot\mathbf{v}_h \mathrm{d}\mathbf{x} =\int_{\Omega}(\mathbf{f}(T_h)-\mathbf{f}(T))\cdot \mathbf{v}_h \mathrm{d}\mathbf{x}+\int_{\Omega}\nabla (\varrho-\mathsf{p}_h)\cdot \mathbf{v}_h\mathrm{d}\mathbf{x}.
\end{align}
Since $\mathbf{u}_h \to \mathbf{u}$ in $\mathbf{L}^3(\Omega)$, we deduce, in view of \eqref{eq:prop5_A}, that $\mathcal{A}(\mathbf{u}_h) \rightarrow \mathcal{A}(\mathbf{u})$ in $\mathbf{L}^{\frac{3}{2}}(\Omega)$ as $h\downarrow 0$. This property combined with the fact that $\mathbf{f}(T_h)\to \mathbf{f}(T)$ in $\mathbf{L}^{\frac{3}{2}}(\Omega)$ as $h\downarrow 0$ allow us to deduce that 
$
 (\nabla(\varrho - \mathsf{p}_h), \boldsymbol{\Pi}_h(\mathbf{v}))_{\mathbf{L}^2(\Omega)} \rightarrow (\nabla(\varrho - \mathsf{p}), \mathbf{v})_{\mathbf{L}^2(\Omega)} = 0,
$
as $h \downarrow 0$, for every $\mathbf{v} \in \mathbf{L}^3(\Omega)$. We have thus proved that $\mathsf{p} = \varrho$. Consequently, the limit point $(\mathbf{u},\mathsf{p})\in\mathbf{X}\times M$ solves problem \eqref{eq:Darcy_decoupled}.

It remains to prove that $T\in W_0^{1,p}(\Omega)$ solves \eqref{eq:heat_decoupled} with $\mathbf{u}(T) = \mathbf{u}$. To accomplish this task, we let $S\in C_0^{\infty}(\Omega)$ \GC{and set $S_h=I_h S \in Y_h$.} Utilize H\"older's inequality, the assumptions on $\kappa$, the Lebesgue dominated convergence \GC{theorem}, and standard properties of the interpolation operator $I_h$ to obtain
\begin{multline*}
\left|\int_\Omega (\kappa(T)\nabla T\cdot \nabla S - \kappa(T_h)\nabla T_h\cdot \nabla S_h)\mathrm{d}\mathbf{x}\right| 
\leq
\left|\int_\Omega (\kappa(T) - \kappa(T_h))\nabla T\cdot \nabla S\mathrm{d}\mathbf{x}\right| \\
+\left|\int_\Omega \kappa(T_h)\nabla(T - T_h)\cdot \nabla S\mathrm{d}\mathbf{x}\right|+\left|\int_\Omega \kappa(T_h)\nabla T_h\cdot \nabla (S-S_h)\mathrm{d}\mathbf{x}\right| \rightarrow 0, \text{ as } h\downarrow 0.
\end{multline*}
Finally, we prove that $\int_\Omega T_h \mathbf{u}_h \cdot \nabla S_h\mathrm{d}\mathbf{x} \to  \int_\Omega T \mathbf{u} \cdot \nabla S\mathrm{d}\mathbf{x}$ as $h\downarrow 0$. Observe that
\begin{multline*}
\left|\int_\Omega (T_h \mathbf{u}_h \cdot \nabla S_h- T \mathbf{u} \cdot \nabla S) \mathrm{d}\mathbf{x}\right|\leq \left|\int_\Omega (T-T_h) \mathbf{u}\cdot \nabla S\mathrm{d}\mathbf{x}\right|\\
+
\left|\int_\Omega T_h (\mathbf{u}-\mathbf{u}_h)\cdot \nabla S\mathrm{d}\mathbf{x}\right|
+
\left|\int_\Omega T_h \mathbf{u}_h\cdot \nabla( S- S_h)\mathrm{d}\mathbf{x}\right|=:\mathrm{I}_h+\mathrm{II}_h+\mathrm{III}_h.
\end{multline*}
The fact that $\mathrm{II}_h \rightarrow 0$ as $h \downarrow 0$ follows from \eqref{eq:termconv} and $\mathbf{u}_h \rightarrow \mathbf{u}$ in $\mathbf{L}^{3}(\Omega)$ as $h \downarrow 0$. To prove that $\mathrm{III}_h \rightarrow 0$ as $h \downarrow 0$, we utilize \eqref{eq:termconv} again and standard properties for $I_h$. Finally, we control $\mathrm{I}_h$ in view of $T_h \rightharpoonup T$ in $W_0^{1,p}(\Omega)$, the compact embedding $W_0^{1,p}(\Omega)\hookrightarrow L^{q}(\Omega)$ for $q<q_{\star}=dp/(d-p)$, H\"older inequality, and the extra assumption on $\mathbf{u}$ \GC{in three dimensions}. We have thus proved that $T\in W_0^{1,p}(\Omega)$ solves \eqref{eq:heat_decoupled} with $\mathbf{u}=\mathbf{u}(T)$. This concludes the proof.
\end{proof}

\section{A posteriori error analysis}
\label{sec:a_posteriori_anal}

In this section, we design and analyze an a posteriori error estimator for the finite element \GC{approximation}  \eqref{eq:model_discrete} of the coupled system \eqref{eq:modelweak}. In particular, we derive a global \GC{reliability} estimate for the devised error estimator and investigate local efficiency bounds. \GC{To achieve} these goals, we shall assume that, in addition to the assumptions stated in \S\ref{sec:main_assump}, $\kappa$ is a positive constant.

In what follows, $(\mathbf{u},\mathsf{p},T) \in \mathbf{X}\times M\times Y$ denotes a solution to \eqref{eq:modelweak}. The existence of such a solution follows from Theorem \ref{thm:existence}; $p$ is such that $2d/(d+1)-\epsilon < p <d/(d-1)$ for some $\epsilon >0$. In addition, $(\mathbf{u}_h,\mathsf{p}_h,T_h)\in\mathbf{X}_{h}\times M_{h}\times Y_{h}$ denotes a solution to \eqref{eq:model_discrete}, whose existence, for $0 < h \leq h_{\star}$ and $p$ as above, is guaranteed by Theorem \ref{thm:existence_discrete}. 

Within the a posteriori error analysis that follows, since we will not be dealing with uniform refinement, the parameter $h$ does not bear the meaning of a mesh size. It can thus be thought as $h = 1/k$, where $k\in\mathbb{N}$ is the index set in a sequence of refinements of an initial mesh or partition $\mathscr{T}_{0}$.

\subsection{A posteriori error estimator} 
The proposed error estimator is decomposed as the sum of two contributions: a contribution related to the discretization of the Darcy--Forchheimer model and another one associated to the discretization of the stationary heat equation. To present such an error estimator, let us first introduce some notation. Let $\mathbf{w}_{h}$ be a discrete tensor valued function and let $\gamma \in \mathscr{S}$. We define the \emph{jump} or \emph{interelement residual} of $\mathbf{w}_{h}$ on $\gamma$ by
$
\llbracket \mathbf{w}_{h} \cdot \mathbf{n} \rrbracket:= \mathbf{w}_{h} \cdot \mathbf{n}^{+} |^{}_{K^{+}} + \mathbf{w}_{h} \cdot \mathbf{n}^{-}  |^{}_{K^{-}},
$
where $\mathbf{n}^{+}$ and $\mathbf{n}^{-}$ denote the unit normals on $\gamma$ pointing towards $K^{+}$ and $K^{-}$, respectively; $K^{+}$, $K^{-} \in \T_h$ are such that $K^{+} \neq K^{-}$ and $\partial K^{+} \cap \partial K^{-} = \gamma$.

\subsubsection{Heat equation with convection: local indicators and estimator}
We define, for an element $K \in \mathscr{T}_h$ and an internal side $\GC{\gamma} \in \mathscr{S}$, the \emph{element residual} $\mathcal{R}_K$ and the \emph{interelement residual} $\mathcal{J}_{\gamma}$ as
\begin{equation}
\mathcal{R}_K:= -\nabla T_h \cdot \mathbf{u}_h |_{K},
\qquad
\mathcal{J}_{\gamma}:= \llbracket(\kappa\nabla T_h -T_h \mathbf{u}_h)\cdot \mathbf{n} \rrbracket.
\label{eq:residual_heat}
\end{equation}
We define an \emph{element indicator} $\mathcal{E}_{p,K}$ associated to the discretization of the heat equation on the basis of three scenarios. First, if $z \in K$ and $z$ is not a vertex, then
\begin{equation}
\label{eq:local_indicator_temp_I}
\mathcal{E}_{p,K}:=
\left[
h_K^{d+p(1-d)} 
+
h_K^p \| \mathcal{R}_K\|_{L^p(K)}^p 
+
h_K \| \mathcal{J}_{\gamma} \|_{L^p(\partial K \setminus \partial \Omega)}^p
\right]^{\frac{1}{p}}.
\end{equation}
Second, if $z \in K$ and $z$ is a vertex of $K$, then
\begin{equation}\label{eq:local_indicator_temp_II}
\mathcal{E}_{p,K}:=
\left[
h_K^p \| \mathcal{R}_K  \|_{L^p(K)}^p 
+
h_K \|
\mathcal{J}_{\gamma}
\|_{L^p(\partial K \setminus \partial \Omega)}^p
\right]^{\frac{1}{p}}.
\end{equation}
Third, if $z \notin K$, then the indicator $\mathcal{E}_{p,K}$ is defined as in \eqref{eq:local_indicator_temp_II}. 

The following comments point in the direction of creating some insight into the definition of the local indicators \eqref{eq:local_indicator_temp_I} and \eqref{eq:local_indicator_temp_II}. First, we recall that we consider our elements $K$ to be closed sets. Second, the Lagrange interpolation operator $I_h$ is well-defined over the space $W_0^{1,p\prime}(\Omega)$ with $p\prime>d$. Third, since $I_h$ is constructed by matching the point values at the Lagrange nodes, we have the basic property
\[
(S -I_{h}S)(\GC{\texttt{v}}) = 0
\quad
\forall S \in W_0^{1,p\prime}(\Omega), \, p\prime>d, \forall K \in \T_h.
\]
Here, $\GC{\texttt{v}}$ denotes a vertex of $K$. In particular, the third observation explains the discrepancy between definitions \eqref{eq:local_indicator_temp_I} and \eqref{eq:local_indicator_temp_II}.

With the previous local indicators at hand, we define the error estimator
\begin{equation}\label{eq:error_estimator_heat}
\mathcal{E}_{p,\T}:=\left( \sum_{K\in\mathscr{T}_h} \mathcal{E}_{p,K}^p\right)^{\frac{1}{p}},\qquad p<\tfrac{d}{d-1}.
\end{equation}

\subsubsection{A Darcy--Forchheimer model: local indicators and estimator}

Let $K \in \T_h$ be an element and let $\gamma \in \mathscr{S}$ be an internal side. We define the \emph{element residual} $\mathfrak{R}_{K}$ and the \emph{interelement residual} $\mathfrak{J}_{\gamma}$ as
\begin{align}
&\mathfrak{R}_{K} := (\mathbf{f}_0+\mathbf{f}_1(T_h)-\nu\mathbf{u}_h-|\mathbf{u}_h|\mathbf{u}_h-\nabla \mathsf{p}_h)|_{K}, 
&& \mathfrak{J}_{\gamma}:= \llbracket \mathbf{u}_h\cdot\mathbf{n}\rrbracket.
\label{eq:residuals_Darcy_1}
\end{align}
With $\mathfrak{R}_{K}$ and $\mathfrak{J}_{\gamma}$ at hand, we define the \emph{element indicator} and \emph{error estimator}
\begin{equation}\label{eq:local_indicator_darcy_I}
\mathfrak{E}_{K}: =\left[\|\mathfrak{R}_{K}\|_{\mathbf{L}^2(K)}^2 + 
h_{K}^{\frac{2}{3}}\|\mathfrak{J}_{\gamma}\|_{L^3(\partial K\setminus \partial \Omega)}^2
\right]^{\frac{1}{2}},
\qquad
\mathfrak{E}_{\T}:=\left( \sum_{K\in\mathscr{T}_h} \mathfrak{E}_{K}^2\right)^{\frac{1}{2}}.
\end{equation}

With all these ingredients at hand, we define the total a posteriori error estimator
\begin{equation}\label{def:total_error_estimator}
\mathsf{E}_{\T}:=\mathcal{E}_{p,\T}+\mathfrak{E}_{\T}.
\end{equation}

\subsection{Reliability analysis}
\label{sec:reliability_estimates}
We begin our analysis by defining the temperature error $e_{T} := T-T_h$, the velocity error $\mathbf{e}_{\mathbf{u}} := \mathbf{u}-\mathbf{u}_h$, and the pressure error $e_{\mathsf{p}} := \mathsf{p}-\mathsf{p}_h$.

Let $\mathsf{q}\in M$ and $\mathsf{q}_h\in M_h$. Since $(\mathbf{u},\mathsf{p},T)$ and $(\mathbf{u}_h,\mathsf{p}_h,T_h)$ solve \eqref{eq:modelweak} and \eqref{eq:model_discrete}, respectively, and $\int_{\Omega} \nabla \mathsf{q}_h \cdot \mathbf{u}_h \mathrm{d}\mathbf{x} = 0$, standard computations based on a integration by parts argument reveal that
\begin{multline}
\int_\Omega\nabla \mathsf{q}\cdot \mathbf{e}_\mathbf{u}\mathrm{d}\mathbf{x}=\int_{\Omega}\nabla \mathsf{q}\cdot \mathbf{u}\mathrm{d}\mathbf{x}-\int_{\Omega}\nabla \mathsf{q}_h\cdot \mathbf{u}_h\mathrm{d}\mathbf{x}-\int_\Omega\nabla (\mathsf{q}-\mathsf{q}_h)\cdot \mathbf{u}_h\mathrm{d}\mathbf{x}\\
=\sum_{K\in\T_h}\int_{K}(\mathsf{q}-\mathsf{q}_h)\text{div}\mathbf{u}_h\mathrm{d}\mathbf{x}+\int_{\partial K}(\mathsf{q}-\mathsf{q}_h)\mathbf{u}_h\cdot\mathbf{n}\mathrm{d}\mathbf{s}
=
\sum_{\gamma\in \mathscr{S}}
\int_\gamma \llbracket \mathbf{u}_h\cdot\mathbf{n}\rrbracket  (\mathsf{q}-\mathsf{q}_h)\mathrm{d}\mathbf{s}.
\label{eq:lemma_00}
\end{multline}

Let us now define the following problem: Find
$\mathbf{w}\in \mathbf{X}$ such that 
\begin{equation}\label{eq:aux_problem_apost}
\int_{\Omega}\nabla \mathsf{q}\cdot \mathbf{w} \mathrm{d}\mathbf{x} = F(\mathsf{q})
~\forall q \in M,
~~~
F: M \rightarrow \mathbb{R},
~~~
F(\mathsf{q}):=\sum_{\gamma\in \mathscr{S}}\int_{\gamma}\llbracket \mathbf{u}_h\cdot\mathbf{n}\rrbracket (\mathsf{q}-R_h \mathsf{q})\mathrm{d}\mathbf{s}.
\end{equation}
Here, $R_h$ denotes the Cl\'ement interpolation operator. We recall that, for each $K\in\T_h$ and $\gamma \in \mathscr{S}$, the operator $R_h$ satisfies the following error estimates \cite[Lemma 1.127]{Guermond-Ern}
\begin{align}
  \label{eq:interp_Rh_i}	
   \|\mathsf{q}-R_h \mathsf{q}\|_{\mathbf{L}^{\frac{3}{2}}(K)}
 &\lesssim
 h_{K}\|\mathsf{q}\|_{W^{1,\frac{3}{2}}(\mathcal{N}^{*}_{K})} \quad \forall \mathsf{q}\in W^{1,\frac{3}{2}}(\mathcal{N}^{*}_{K}),\\
  \label{eq:interp_Rh_ii}
   \|\mathsf{q}-R_h \mathsf{q}\|_{\mathbf{L}^{\frac{3}{2}}(\gamma)}
 &\lesssim
 h_{\gamma}^{1/3}\|\mathsf{q}\|_{W^{1,\frac{3}{2}}(\mathcal{N}^*_{\gamma})} \quad \forall \mathsf{q}\in W^{1,\frac{3}{2}}(\mathcal{N}^*_{\gamma}).
 \end{align}
The set $\mathcal{N}^{*}_{K}$ is defined as in \eqref{eq:patch} and $\mathcal{N}^*_{\gamma}$ corresponds to the set of elements in $\mathscr{T}_h$ sharing at least one vertex with $\gamma$.

The following \GC{result} is instrumental to perform a reliability analysis.
\begin{lemma}[auxiliary result]\label{lemma:01_aposteriori} There exists $\mathbf{v}_r \in \mathbf{X}$ which satisfies  
\eqref{eq:aux_problem_apost} and 
\begin{equation}
\|\mathbf{v}_r\|_{\mathbf{L}^3(\Omega)}\lesssim \sum_{\gamma\in\mathscr{S}}
h_\gamma^{1/3}\|\llbracket \mathbf{u}_h\cdot\mathbf{n}\rrbracket \|_{L^3(\gamma)}.
\label{eq:bound_vr}
\end{equation}
\end{lemma}

\begin{proof}
We first notice that $\|F\|_{M'} \lesssim \sum_{\gamma \in \mathscr{S}} h_\gamma^{1/3}\|\llbracket \mathbf{u}_h\cdot\mathbf{n}\rrbracket \|_{L^3(\gamma)}$.
This bound follows from  H\"older's inequality and the error estimate \eqref{eq:interp_Rh_ii}.
We can thus apply the Banach--Ne\v{c}as--Babu\v{s}ka Theorem \cite[Theorem 2.6]{Guermond-Ern}, to obtain, on the basis of the inf--sup condition \eqref{eq:infsup}, the existence of $\mathbf{v}_r$ solving \eqref{eq:aux_problem_apost} and satisfying \eqref{eq:bound_vr}.
\end{proof}

We are now ready to enunciate and prove the main result of this section.

\begin{theorem}[global reliability]\label{thm:global_reliability}
Let $d=2$, let $\mathfrak{C}$ and $\tilde{\mathfrak{C}}$ be the constants defined in \eqref{def:constant_C} and \eqref{def:constant_C_tilde}, respectively, and let $p$ be such that $4/3 -\epsilon < p < 2$ for some $\epsilon >0$. Let $(\mathbf{u}_h,\mathsf{p}_h,T_h)\in\mathbf{X}_{h}\times M_{h}\times Y_{h}$, for $0<h \leq h_{\star}$, be a solution to the discrete system \eqref{eq:model_discrete}; \GC{passing to a nonrelabelared subsequence if necessary $\mathbf{u}_h \to \mathbf{u}$ in $\mathbf{X}$, $\mathsf{p}_h \rightharpoonup \mathsf{p}$ in $M$, and $T_h \rightharpoonup T$ in $Y$ as $h\downarrow 0$, where $(\mathbf{u},\mathsf{p},T) \in \mathbf{X}\times M\times Y$ is a solution to \eqref{eq:modelweak}.} \GC{Assume that \eqref{eq:assump_convergence} holds and that, for $d = 3$, $\mathbf{u} \in \mathbf{L}^{3+\varepsilon}(\Omega)$, where $\varepsilon>0$ is arbitrarily small. Assume, in addition,} that the aforementioned solutions and problem data are such that the following inequalities hold:
\begin{align}
  \label{eq:global_rel_assump}
  C_{\mathbf{f}}\|\delta_z\|_{W^{-1,p}(\Omega)}  &\leq (1-\rho)\nu^{}_{-}\mathfrak{D},
  \\
  \label{eq:discrete_assump}
\|\mathbf{e}_{\mathbf{u}} \|_{\mathbf{L}^{2}(\Omega)}+\|\mathbf{u} \|_{\mathbf{L}^{2}(\Omega)}& \leq 2\rho \mathfrak{C},
\\
\label{eq:discrete_assump02}
2\|\mathbf{u}\|_{\mathbf{L}^6(\Omega)}+\|\mathbf{e}_{\mathbf{u}}\|_{\mathbf{L}^6(\Omega)} &\leq C, 
\end{align}
where $\mathfrak{D}=(2\sqrt{10} C_{\kappa}^2C_{e}\mathscr{C}_e)^{-1}$, $C_e$ is the best constant in $W_0^{1,p}(\Omega)\hookrightarrow L^{\frac{2p}{2-p}}(\Omega)$, $\mathscr{C}_e$ is the best constant in $W_0^{1,p}(\Omega) \hookrightarrow L^{2}(\Omega)$, $\mathfrak{C}_e$ is the best constant in $W_0^{1,p}(\Omega)\hookrightarrow L^{\frac{3}{2}}(\Omega)$, $C_{\kappa}$ is as in \GC{\eqref{eq:inf_sup_Poisson}}, $C_{\mathbf{f}}$ denotes the Lipschitz constant of $\mathbf{f}$, $\rho \in (0,1)$, and  $C>0$. Then
\begin{equation}\label{eq:upper_bound}
\|\nabla e_{T}\|_{\mathbf{L}^p(\Omega)} + \|\mathbf{e}_{\mathbf{u}}\|_{\mathbf{L}^2(\Omega)}+\|\mathbf{e}_{\mathbf{u}}\|_{\mathbf{L}^{3}(\Omega)}^{\frac{3}{2}}+\|\nabla e_{\mathsf{p}}\|_{\mathbf{L}^{\frac{3}{2}}(\Omega)} 
\lesssim
\mathsf{E}_{\T}.
\end{equation}
The hidden constant is independent of the size of the elements in $\mathscr{T}_h$ and $\#\mathscr{T}_h$.
\end{theorem}
\begin{proof}
We divide the proof in four steps.

\emph{Step 1.} We first bound $\|\nabla e_{T}\|_{\mathbf{L}^p(\Omega)}$. Owing to Proposition \ref{prop:u=0}, we have that there is a positive constant $C_{\kappa}$ such that the following inf-sup condition holds:
\begin{equation}
 \|\nabla e_{T}\|_{\mathbf{L}^{p}(\Omega)}\leq C_{\kappa}\sup_{S\in W_0^{1,p'}(\Omega)}\frac{1}{\|\nabla S\|_{\mathbf{L}^{p'}(\Omega)}}\int_\Omega\kappa\nabla e_{T}\cdot \nabla S\mathrm{d}\mathbf{x}.
 \label{eq:a_posteriori_infsup}
\end{equation}
To estimate the right-hand side of \eqref{eq:a_posteriori_infsup}, we invoke the third equation in \eqref{eq:modelweak}, a standard integration by parts formula, and Galerkin orthogonality to obtain
\begin{multline}
\label{eq:error_equation_temp}
\int_\Omega \left( \kappa\nabla e_{T} + e_{T} \mathbf{e}_\mathbf{u}  -e_T \mathbf{u}  - T\mathbf{e}_{\mathbf{u}}\right)\cdot\nabla S\mathrm{d}\mathbf{x} = \langle \delta_{z},S-I_{h}S\rangle \\
+\sum_{K\in\mathscr{T}_h}\int_{K}\mathcal{R}_K(S-I_{h}S)\mathrm{d}{\mathbf{x}} 
+ 
\sum_{\gamma\in\mathscr{S}}\int_\gamma \mathcal{J}_{\gamma}(S-I_{h}S)\mathrm{d}\mathbf{s},
\end{multline}
for an arbitrary $S\in W_0^{1,p\prime}(\Omega)$; $\mathcal{R}_K$ and $\mathcal{J}_{\gamma}$ are defined in \eqref{eq:residual_heat}. We can thus invoke \eqref{eq:a_posteriori_infsup}, H\"older's inequality, the local error bounds \eqref{eq:interp_Lagrange} and \eqref{eq:interp_Lagrange_side}, \EO{and the estimate $\| S - I_h S \|_{L^{\infty}(K)} \lesssim h_K^{\sigma} \| \nabla S\|_{L^ {p'}(K)}$ \cite[Corollary 4.4.7]{MR2373954}, where $\sigma = 1-d/p'$,} to obtain
\begin{multline*}
\|\nabla e_{T}\|_{\mathbf{L}^{p}(\Omega)}
\!\leq \!C_{\kappa} [ \|e_{T}\|_{L^{\mathfrak{p}}(\Omega)} 
\!\left( \|\mathbf{e}_{\mathbf{u}}\|_{\mathbf{L}^2(\Omega)} \!+\!\! \|\mathbf{u}\|_{\mathbf{L}^2(\Omega)} \right)
+  \|T\|_{L^{\mathfrak{p}}(\Omega)} \|\mathbf{e}_{\mathbf{u}}\|_{\mathbf{L}^2(\Omega)} + C\mathcal{E}_{p,\T}],
\end{multline*} 
where $\mathfrak{p} = \frac{2p}{2-p}$. This bound, $W_0^{1,p}(\Omega)\hookrightarrow L^{\frac{2p}{2-p}}(\Omega)$, and assumption \eqref{eq:discrete_assump}, reveal that
\begin{equation}\label{eq:temp_estimate}
(1-\rho)\|\nabla e_{T}\|_{\mathbf{L}^{p}(\Omega)} 
\leq 
2C_{\kappa}^{2}C_{e} \|\delta_z\|_{W^{-1,p}(\Omega)}\|\mathbf{e}_{\mathbf{u}}\|_{\mathbf{L}^2(\Omega)} + CC_{\kappa}\mathcal{E}_{p,\T},
\end{equation}
where $C_e$ denotes the best constant in the embedding $W_0^{1,p}(\Omega)\hookrightarrow L^{\frac{2p}{2-p}}(\Omega)$ and \GC{$C_{\kappa}$ is the constant involved in the inf-sup condition \eqref{eq:inf_sup_Poisson} with $\xi$ being replaced by $\kappa$.}

\emph{Step 2}. We now follow \cite{https://doi.org/10.48550/arxiv.2202.11643} and control $\mathbf{e}_{\mathbf{u}}$. To do this, we first observe that
\begin{align}\label{eq:error_equation_darcy_1}
\int_\Omega \big(\mathcal{A}(\mathbf{u})-\mathcal{A}(\mathbf{u}_h)+\nabla e_{\mathsf{p}}\big)\cdot\mathbf{v}\mathrm{d}\mathbf{x}\!-\!\int_\Omega (\mathbf{f}(T)-\mathbf{f}(T_h))\cdot\mathbf{v}\mathrm{d}\mathbf{x}\!=\! \sum_{K\in\mathscr{T}_h}\int_K \mathfrak{R}_{K}\cdot\mathbf{v}\mathrm{d}\mathbf{x},
\end{align}
for $\mathbf{v} \in \mathbf{L}^3(\Omega)$; $\mathfrak{R}_{K}$ is defined in \eqref{eq:residuals_Darcy_1}. Let us now define $\mathbf{z}:=\mathbf{u}-\mathbf{u}_h-\mathbf{v}_r = \mathbf{e}_{\mathbf{u}} -\mathbf{v}_r $, where $\mathbf{v}_r$ is as in the statement of Lemma \ref{lemma:01_aposteriori}, and set $\mathbf{v}=\mathbf{z}$ in \eqref{eq:error_equation_darcy_1} to obtain
\begin{align}\label{eq:error_equation_darcy_id}
\int_\Omega \left(\nu\mathbf{e}_{\mathbf{u}}+|\mathbf{u}|\mathbf{u}-|\mathbf{u}_h|\mathbf{u}_h\right)\cdot\mathbf{z}\mathrm{d}\mathbf{x} - \int_\Omega \left(\mathbf{f}(T)-\mathbf{f}(T_h)\right) \cdot\mathbf{z}\mathrm{d}\mathbf{x}  = \sum_{K\in\mathscr{T}_h}\int_K \mathfrak{R}_{K}\cdot\mathbf{z}\mathrm{d}\mathbf{x},
\end{align}
where we have utilized that $\int_\Omega \nabla \mathsf{q}\cdot \mathbf{z} \mathrm{d}\mathbf{x}=0$, for every $\mathsf{q}\in M$, which follows from \eqref{eq:lemma_00} and \eqref{eq:aux_problem_apost}. Consequently, we can arrive at the identity
\begin{multline*}
\int_\Omega \left[\nu\mathbf{e}_{\mathbf{u}} \cdot \mathbf{z}+(|\mathbf{u}|\mathbf{u}-|\mathbf{u}_h|\mathbf{u}_h) \cdot \mathbf{e}_{\mathbf{u}} \right]\mathrm{d}\mathbf{x} = \sum_{K\in\mathscr{T}_h}\int_K \mathfrak{R}_{K}\cdot\mathbf{z}\mathrm{d}\mathbf{x}
\\
+
\int_\Omega \left(\mathbf{f}(T)-\mathbf{f}(T_h)\right) \cdot\mathbf{z}\mathrm{d}\mathbf{x}
+
\int_\Omega|\mathbf{u}|(\mathbf{u}-\mathbf{u}_h)\cdot \mathbf{v}_r\mathrm{d}\mathbf{x}+\int_\Omega(|\mathbf{u}|-|\mathbf{u}_h|)\mathbf{u}_h\cdot \mathbf{v}_r\mathrm{d}\mathbf{x}.
\end{multline*}
Let now us invoke \cite[Lemma 4.4, Chapter I]{MR1230384} with $p=3$ and utilize the fact that $\nu$ is such that $0<\nu_{-}\leq \nu(\mathbf{x})\leq \nu_{+}$ for every $\mathbf{x} \in \Omega$ to conclude that
\begin{multline*}
\nu_{-}\|\mathbf{z}\|_{\mathbf{L}^2(\Omega)}^2+ \gamma_d \|\mathbf{e}_{\mathbf{u}}\|_{\mathbf{L}^3(\Omega)}^3\leq C_{\mathbf{f}}\mathscr{C}_e\|\nabla e_T\|_{\mathbf{L}^p(\Omega)} \|\mathbf{z}\|_{\mathbf{L}^2(\Omega)}+\sum_{K\in\mathscr{T}_h}\|\mathfrak{R}_{K}\|_{\mathbf{L}^2(\Omega)}\|\mathbf{z}\|_{\mathbf{L}^2(\Omega)}\\
+(2\|\mathbf{u}\|_{\mathbf{L}^6(\Omega)}+\|\mathbf{e}_{\mathbf{u}}\|_{\mathbf{L}^6(\Omega)})(\|\mathbf{z}\|_{\mathbf{L}^2(\Omega)}+\|\mathbf{v}_r\|_{\mathbf{L}^2(\Omega)})\|\mathbf{v}_r\|_{\mathbf{L}^3(\Omega)}+\nu_{+}\|\mathbf{v}_r\|_{\mathbf{L}^2(\Omega)}\|\mathbf{z}\|_{\mathbf{L}^2(\Omega)},
\end{multline*}
where $\mathscr{C}_e$ is the best constant in $W_{0}^{1,p}(\Omega) \hookrightarrow L^2(\Omega)$. We now utilize the basic bound $\|\mathbf{v}\|_{\mathbf{L}^2(\Omega)}\lesssim \|\mathbf{v}\|_{\mathbf{L}^3(\Omega)}$, the results of Lemma \ref{lemma:01_aposteriori}, assumptions \eqref{eq:global_rel_assump} and \eqref{eq:discrete_assump02}, the a posteriori bound \eqref{eq:temp_estimate}, and suitable Young's inequalities to deal with the terms involving $\|\mathbf{z}\|_{\mathbf{L}^2(\Omega)}$ appearing in the right-hand side of the previous inequality, to obtain
\begin{align*}
\|\mathbf{e}_{\mathbf{u}}\|_{\mathbf{L}^2(\Omega)}^2+\|\mathbf{e}_{\mathbf{u}}\|_{\mathbf{L}^3(\Omega)}^3 
\lesssim 
\mathfrak{E}_{\T}^2 
+\mathcal{E}_{p,\T}^2.
\end{align*} 
Replacing this bound into estimate \eqref{eq:temp_estimate} immediately yields
$\|\nabla e_{T}\|_{\mathbf{L}^{p}(\Omega)} \lesssim \mathfrak{E}_{\T} + \mathcal{E}_{p,\T}$, upon utilizing assumption \eqref{eq:global_rel_assump}. 

\emph{Step 3}. Finally, we control the pressure error $\| \nabla e_{\mathsf{p}} \|_{\mathbf{L}^{3/2}(\Omega)}$. Since $\mathcal{A}(\mathbf{u}) - \mathcal{A}(\mathbf{u}_h) = \nu \mathbf{e}_{\mathbf{u}} + |\mathbf{u}|\mathbf{u}-|\mathbf{u}_h|\mathbf{u}_h$, we utilize \eqref{eq:error_equation_darcy_1} combined with the Lipschitz property of $\mathbf{f}_1$ and the embedding $W_0^{1,p}(\Omega)\hookrightarrow L^{2}(\Omega)$ to obtain
\begin{multline}
\left|
\int_\Omega \nabla e_{\mathsf{p}}\cdot\mathbf{v}\mathrm{d}\mathbf{x}
\right|
\leq \sum_{K\in\mathscr{T}_h}\|\mathfrak{R}_{K}\|_{\mathbf{L}^2(K)}\|\mathbf{v}\|_{\mathbf{L}^2(K)}+\nu_{+}\|\mathbf{e}_{\mathbf{u}}\|_{\mathbf{L}^2(\Omega)}\|\mathbf{v}\|_{\mathbf{L}^2(\Omega)}
\\
+
C
\|\mathbf{e}_{\mathbf{u}}\|_{\mathbf{L}^{2}(\Omega)}\|\mathbf{v}\|_{\mathbf{L}^3(\Omega)}
+
C_{\mathbf{f}}\mathscr{C}_e\|\nabla e_T\|_{\mathbf{L}^{p}(\Omega)}\|\mathbf{v}\|_{\mathbf{L}^2(\Omega)},
\end{multline}
upon utilizing that $2\|\mathbf{u}\|_{\mathbf{L}^{6}(\Omega)} + \|\mathbf{e}_\mathbf{u}\|_{\mathbf{L}^{6}(\Omega)} \leq C$. It thus suffices to invoke the inf-sup condition \eqref{eq:infsup} to arrive at the a posteriori error bound
\begin{equation*}
\|\nabla e_{\mathsf{p}}\|_{\mathbf{L}^{\frac{3}{2}}(\Omega)}
\lesssim 
\|\mathbf{e}_{\mathbf{u}}\|_{\mathbf{L}^2(\Omega)}+ \|\nabla e_{T}\|_{\mathbf{L}^p(\Omega)}  + \mathfrak{E}_{\T}.
\end{equation*}

\emph{Step 4.} The desired a posteriori error bound \eqref{eq:upper_bound} follows from collecting the a posteriori error estimates derived in Steps 1, 2, and 3. This concludes the proof.
\end{proof}


\subsection{Efficiency estimates}\label{sec:efficiency_estimates}

In this section, we analyze efficiency properties for the local error indicators $\mathcal{E}_{p,K}$ and $\mathfrak{E}_K$. We begin our analysis by introducing the following notation: for an edge, triangle or tetrahedron $G$, let $\mathcal{V}(G)$ be the set of vertices of $G$. With this notation at hand, for $K\in\mathscr{T}_h$ and $\gamma\in\mathscr{S}$, we introduce the following standard element and edge bubble functions:
\begin{equation}\label{def:standard_bubbles}
\varphi^{}_{K}=
(d+1)^{d+1}\prod_{\textsc{v} \in \mathcal{V}(K)} \lambda^{}_{\textsc{v}},
\quad
\varphi^{}_{\gamma}=
d^d \prod_{\textsc{v} \in \mathcal{V}(\gamma)}\lambda^{}_{\textsc{v}}|^{}_{K'},
\quad 
K' \subset \mathcal{N}_{\gamma}\GC{.}
\end{equation}
Here, $\lambda_{\textsc{v}}$ denote the barycentric coordinate function associated to $\textsc{v}$. $\mathcal{N}_{\gamma}$ corresponds to the patch composed of the two elements of $\mathscr{T}_h$ sharing $\gamma$.

We will also make use of the following bubble functions, whose construction we owe to \cite{MR3237857,MR2262756}.
Given $K\in\mathscr{T}_h$, we define the element bubble function $\phi_K$ as
\begin{equation}\label{def:bubble_element}
\phi_K(\mathbf{x}):=
h_K^{-2}|\mathbf{x}-z|^{2}\varphi_K(\mathbf{x})  \text{  if  } z\in K,
\qquad
\phi_K(\mathbf{x}):=
\varphi_K(\mathbf{x}) \text{  if  } z\not\in K.
\end{equation}
Given $\gamma\in \mathscr{S}$, we define the edge bubble function $\phi_\gamma$ as
\begin{equation}\label{def:bubble_side}
\phi_\gamma(\mathbf{x}):=
h_\gamma^{-2}|\mathbf{x}-z|^{2} \varphi_\gamma(\mathbf{x})  \text{ if } z\in \mathring{\mathcal{N}}_\gamma,
\qquad
\phi_\gamma(\mathbf{x}):=
\varphi_\gamma(\mathbf{x}) \text{ if } z\not\in \mathring{\mathcal{N}}_\gamma,
\end{equation}
where $\mathring{\mathcal{N}}_\gamma$ denotes the interior of $\mathcal{N}_\gamma$. We recall that the Dirac measure $\delta_{z}$ is supported at $z \in \Omega$: it can thus be supported on the interior, an edge, or a vertex of an element $K$ of the triangulation $\mathscr{T}_h$.

Given $\gamma\in\mathscr{S}$, we introduce the continuation operator $\Pi_{\gamma}:L^\infty(\gamma)\rightarrow L^\infty(\mathcal{N}_\gamma)$ \cite[Section 3]{MR1650051} and notice that $\Pi_{\gamma}$ maps polynomials onto piecewise polynomials of the same degree. $\Pi_{\gamma}$ will be useful for controlling \emph{jump} terms.

We now provide the following result \cite[Lemmas 3.1 and 3.2]{MR3237857}. 

\begin{lemma}[bubble function properties]\label{lemma:bubble_properties}
Let $K\in\mathscr{T}_h$, $\gamma\in\mathscr{S}$, and $r\in(1,\infty)$. If $S_{h}|^{}_K\in\mathbb{P}_1(K)$ and $R_{h}|^{}_\gamma\in \mathbb{P}_1(\gamma)$, then
\begin{align*}
\|S_{h}\|_{L^r(K)} &\lesssim \|S_{h} \phi_K^{\frac{1}{r}}\|_{L^r(K)} \lesssim \|S_{h}\|_{L^r(K)}, 
\\
\|R_{h}\|_{L^r(\gamma)} &\lesssim \|R_{h} \phi_\gamma^{\frac{1}{r}}\|_{L^r(\gamma)} \lesssim \|R_{h}\|_{L^r(\gamma)},
\\
\|\phi_\gamma \Pi_{\gamma} (R_{h})\|_{L^r(\mathcal{N}_{\gamma})} & \lesssim h_{\gamma}^{\frac{1}{r}}\|R_{h}\|_{L^r(\gamma)}.
\end{align*}
\end{lemma}

\subsubsection{Local estimates for $\mathcal{E}_{p,K}$}

In the following result\GC{,} we derive local error bounds for $\mathcal{E}_{p,K}$, which is defined in \eqref{eq:local_indicator_temp_I} and \eqref{eq:local_indicator_temp_II}.

\begin{theorem}[local estimate for $\mathcal{E}_{p,K}$]\label{thm:efficiency} 
Let $d\in\{2,3\}$. 
\EO{In the framework of Theorem \ref{thm:convergence},
we have}
\begin{equation}\label{eq:local_efficiency_temp}
\mathcal{E}_{p,K}
\lesssim
\|\nabla e_{T}\|_{\mathbf{L}^p(\mathcal{N}_K^{*})} + \|e_{T}\|_{L^{\frac{dp}{d-p}}(\mathcal{N}_K^{*})} + \|\mathbf{e}_{\mathbf{u}}\|_{\mathbf{L}^d(\mathcal{N}_K^{*})},
\end{equation}
where $\mathcal{N}_{K}^{*}$ is defined in \eqref{eq:patch}. The hidden constant is independent of continuous and discrete solutions $(\mathbf{u},\mathsf{p},T)$ and $(\mathbf{u}_h,\mathsf{p}_h,T_h)$, respectively, the size of the elements in the mesh $\mathscr{T}_h$, and $\#\mathscr{T}_h$.
\end{theorem}
\begin{proof}
The proof of \eqref{eq:local_efficiency_temp} for $d=2$ is available in \cite[Theorem 5.2]{ACFO:22}. An extension of these arguments to the three dimensional scenario yield \eqref{eq:local_efficiency_temp} for $d=3$. For brevity, we skip details.
\end{proof}

\subsubsection{Local estimates for $\mathfrak{E}_{K}$}
We now investigate local efficiency bounds for the error indicator $\mathfrak{E}_{K}$ defined in \eqref{eq:local_indicator_darcy_I}.
\begin{theorem}[local estimate for $\mathfrak{E}_{K}$]\label{thm:efficiencydarcy}
Let $d\in\{2,3\}$. Let us assume that $\mathbf{f}_0 \in \mathbf{L}^2(\Omega)$ and that $\mathbf{u} \in \mathbf{L}^{4}(\Omega)$. Let $K\in\mathscr{T}_h$. In the framework of Theorem \ref{thm:convergence}, we have 
\begin{multline}\label{eq:local_efficiency_darcya}
\mathfrak{E}_{K}
\lesssim 
\||\mathbf{u}|\mathbf{u}-|\mathbf{u}_h|\mathbf{u}_h+\nabla e_{\mathsf{p}}\|_{\mathbf{L}^2(\mathcal{N}_K)}
+ 
h_K^{1+\frac{d}{2}-\frac{d}{p}}\| e_{T}\|_{L^{\frac{dp}{d-p}}(\mathcal{N}_K)}
+ \|\mathbf{e}_{\mathbf{u}}\|_{\mathbf{L}^2(\mathcal{N}_K)}
\\
+\sum_{K'\in\mathcal{N}_K}
\|\mathbf{v}_r\|_{\mathbf{L}^3(K')}+\|\mathbf{f}_1(T)\!-\!\mathbf{\Pi}_{K'}\mathbf{f}_1(T)\|_{\mathbf{L}^2(K')}\!+\!\|\mathbf{f}_0\!-\!\mathbf{\Pi}_{K'}\mathbf{f}_0\|_{\mathbf{L}^2(K')},
\end{multline}
where $\mathcal{N}_K$ is defined in \eqref{eq:patch} and $\mathbf{\Pi}_{K}$ is the orthogonal projection operator onto $[\mathbb{P}_{0}(K)]^{d}$. The hidden constant is independent of $(\mathbf{u},\mathsf{p},T)$ and $(\mathbf{u}_h,\mathsf{p}_h,T_h)$, the size of the elements in the mesh $\mathscr{T}_h$, and $\#\mathscr{T}_h$.
\end{theorem}
\begin{proof}
We begin the proof by rewriting \eqref{eq:error_equation_darcy_1} as follows:
\begin{multline}
\label{eq:error_equation_darcy_3}
\int_\Omega 
\left(
\nu \mathbf{e}_{\mathbf{u}} +  |\mathbf{u}|\mathbf{u}-|\mathbf{u}_h|\mathbf{u}_h 
+
\nabla e_{\mathsf{p}}
\right)
\cdot\mathbf{v}\mathrm{d}\mathbf{x} 
- \int_{\Omega} (\mathbf{f}(T)-\mathbf{f}(T_h)) \cdot \mathbf{v} \mathrm{d}\mathbf{x}
\\
= \sum_{K\in\mathscr{T}_h}
\left(
\int_K 
\left[ 
(\mathbf{\Pi}_{K}\mathbf{f}(T_h)-\nu\mathbf{u}_h-|\mathbf{u}_h|\mathbf{u}_h-\nabla \mathsf{p}_h)  + (\mathbf{f}(T_h)-\mathbf{\Pi}_{K}\mathbf{f}(T_h)
\right]
\cdot\mathbf{v}\mathrm{d}\mathbf{x}
\right),
\end{multline}
which holds for every $\mathbf{v} \in \mathbf{L}^3(\Omega)$.

Let us now proceed on the basis of three steps.

\emph{Step 1.} Let $K\in\mathscr{T}_h$. We bound the term  $\|\mathfrak{R}_{K}\|_{\mathbf{L}^2(K)}$ in \eqref{eq:local_indicator_darcy_I}. To accomplish this task, we begin with a basic application of the triangle inequality and write
\begin{equation}\label{eq:residual_triangle_darcy}
\|\mathfrak{R}_{K}\|_{\mathbf{L}^2(K)} 
\leq
\|\mathbf{\Pi}_{K}\mathbf{f}(T_h)-\nu\mathbf{u}_h-|\mathbf{u}_h|\mathbf{u}_h-\nabla \mathsf{p}_h\|_{\mathbf{L}^2(K)}
+ 
\|\mathbf{f}(T_h)-\mathbf{\Pi}_{K}\mathbf{f}(T_h)\|_{\mathbf{L}^2(K)}.
\end{equation}
It thus suffices to control the first term on the right-hand side of \eqref{eq:residual_triangle_darcy}. To do this, we set $\mathbf{v} = \varphi_K \tilde{\mathfrak{R}}_{K} \in \mathbf{L}^3(\Omega)$ in \eqref{eq:error_equation_darcy_3}, where $\tilde{\mathfrak{R}}_{K}=(\mathbf{\Pi}_{K}\mathbf{f}(T_h)-\nu\mathbf{u}_h-|\mathbf{u}_h|\mathbf{u}_h-\nabla \mathsf{p}_h)|^{}_K$. Basic properties of the bubble function $\varphi_K$ combined with H\"older's inequality yield
\begin{multline}
\|\tilde{\mathfrak{R}}_K\|_{\mathbf{L}^2(K)}
\! \lesssim 
\|\mathbf{e}_{\mathbf{u}}\|_{\mathbf{L}^2(K)} +\||\mathbf{u}|\mathbf{u}-|\mathbf{u}_h|\mathbf{u}_h+\nabla e_{\mathsf{p}}\|_{\mathbf{L}^2(K)}
\\
+ 
h_K^{1+\frac{d}{2}-\frac{d}{p}}\| e_{T}\|_{L^{\frac{dp}{d-p}}(K)}
+
\|\mathbf{f}(T_h)-\mathbf{\Pi}_{K}\mathbf{f}(T_h)\|_{\mathbf{L}^2(K)},
 \label{eq:aux_tildeRk}
\end{multline}
upon utilizing the Lipschitz property that $\mathbf{f}_1$ satisfies. To control the term $\|\mathbf{f}(T_h)-\mathbf{\Pi}_{K}\mathbf{f}(T_h)\|_{\mathbf{L}^2(K)}$, we invoke \eqref{eq:external_density_force}, the Lipschitz property that $\mathbf{f}_1$ satisfies, the triangle inequality, and H\"older's inequality to arrive at
\begin{multline}
\|\mathbf{f}(T_h)-\mathbf{\Pi}_{K}\mathbf{f}(T_h)\|_{\mathbf{L}^2(K)}
\leq  
\|\mathbf{f}_0-\mathbf{\Pi}_{K}\mathbf{f}_0\|_{\mathbf{L}^2(K)}
+
\|\mathbf{f}_1(T_h)-\mathbf{\Pi}_{K} \mathbf{f}_1(T_h)\|_{\mathbf{L}^2(K)}
\\
\lesssim
\|\mathbf{f}_0-\mathbf{\Pi}_{K}\mathbf{f}_0\|_{\mathbf{L}^2(K)}
+
 \|\mathbf{f}_1(T)-\mathbf{\Pi}_{K}\mathbf{f}_1(T)\|_{\mathbf{L}^2(K)}
 +
h_K^{1+\frac{d}{2}-\frac{d}{p}}\| e_{T}\|_{L^{\frac{dp}{d-p}}(K)}.
  \label{eq:aux_tildeRk_2}
\end{multline}
To obtain the last bound, we have also utilized the estimate $\|\mathbf{\Pi}_K \mathbf{S}-\mathbf{\Pi}_K \mathbf{R} \|_{\mathbf{L}^{r}(K)} \leq \|\mathbf{S}-\mathbf{R}\|_{\mathbf{L}^{r}(K)}$, which holds for every $r \in (1,\infty)$ and $\mathbf{S}, \mathbf{R} \in \mathbf{L}^r(\Omega)$. Replacing \eqref{eq:aux_tildeRk_2} into \eqref{eq:aux_tildeRk}, we obtain
\begin{multline}\label{eq:estimate_residual_darcy}
\|\tilde{\mathfrak{R}}_K\|_{\mathbf{L}^2(K)}
\lesssim
\|\mathbf{e}_{\mathbf{u}}\|_{\mathbf{L}^2(K)} 
+\||\mathbf{u}|\mathbf{u}-|\mathbf{u}_h|\mathbf{u}_h+\nabla e_{\mathsf{p}}\|_{\mathbf{L}^2(K)}
\\
h_K^{1+\frac{d}{2}-\frac{d}{p}}\| e_{T}\|_{L^{\frac{dp}{d-p}}(K)}
+
\|\mathbf{f}_0-\mathbf{\Pi}_K\mathbf{f}_0\|_{\mathbf{L}^2(K)} +\|\mathbf{f}_1(T)-\mathbf{\Pi}_K\mathbf{f}_1(T)\|_{\mathbf{L}^2(K)}.
\end{multline}
The desired estimate for the term $\| \mathfrak{R}_K\|_{\mathbf{L}^2(K)}$ thus follows from \eqref{eq:residual_triangle_darcy} and \eqref{eq:estimate_residual_darcy}.

\emph{Step 2.} Let $K\in\mathscr{T}_h$ and $\gamma\in\mathscr{S}_K$. We now bound the term $h_{K}^{1/3}\|\mathfrak{J}_{\gamma}\|_{L^3(\gamma)}$ in \eqref{eq:local_indicator_darcy_I}. To accomplish this task, we define $\mathfrak{q}:= \varphi_\gamma\Pi_{\gamma}(\mathfrak{J}_{\gamma})$. Here, $\varphi_\gamma$ denotes the bubble function defined in \eqref{def:standard_bubbles} and $\Pi_{\gamma}$ denotes the continuation operator introduced in \S \ref{sec:efficiency_estimates}. We would like to set $q = \mathfrak{q}$ in \eqref{eq:lemma_00}. Unfortunately, $\mathfrak{q}$  does not necessarily belongs to $L_0^2(\Omega)$. Therefore, we set $q = \mathfrak{q} - \mathfrak{q}_{\Omega} \in M$, where $\mathfrak{q}_{\Omega}$ corresponds to the mean of $\mathfrak{q}$ in $\Omega$. On the other hand, we consider a discrete function $q_h \in M_h$ such that, over the particular internal side $\gamma$, $(q_h,1)_{L^2(\gamma)} = - (\mathfrak{q}_{\Omega},1)_{L^2(\gamma)}$. Setting $q_h$ in \eqref{eq:lemma_00} yields
\begin{equation*}
\|\mathfrak{J}_{\gamma}\|_{\mathbf{L}^2(\gamma)}^2
\lesssim
\sum_{K'\in\mathcal{N}_\gamma}\int_{K'} \nabla q\cdot \mathbf{e}_{\mathbf{u}}\mathrm{d}\mathbf{x}
=
\sum_{K'\in\mathcal{N}_\gamma}\int_{K'}\nabla q\cdot\mathbf{v}_r  \mathrm{d}\mathbf{x},
\end{equation*}
where we have utilized that $\mathbf{e}_{\mathbf{u}} = \mathbf{z} + \mathbf{v}_{r}$ and that $\int_{\Omega} \nabla q\cdot \mathbf{z}\mathrm{d}\mathbf{x} = 0$; the latter being a consequence of the fact that $q \in M$. Let us now utilize standard properties of the bubble function $\varphi_\gamma$ and inverse estimates to obtain
\begin{equation*}
h_{\gamma}^{1/3}\|\mathfrak{J}_{\gamma}\|_{\mathbf{L}^3(\gamma)}\lesssim
\sum_{K'\in\mathcal{N}_\gamma}\|\mathbf{v}_r\|_{\mathbf{L}^3(K')},
\end{equation*}
upon utilizing the bound $\|\mathbf{v}_r\|_{\mathbf{L}^2(K')} \leq|K'|^{1/6} \|\mathbf{v}_r\|_{\mathbf{L}^3(K')}$. 

\emph{Step 3}. A collection of the bounds derived in Steps 1 and 2 yield the desired estimate. This concludes the proof.
\end{proof}

\begin{remark}[\GC{higher} integrability on $\mathbf{u}$]\label{eq:remark}
In Theorem \ref{thm:efficiencydarcy}, we have assumed that $\mathbf{u} \in \mathbf{L}^4(\Omega)$ and that $\mathbf{f}_0 \in \mathbf{L}^2(\Omega)$. Under this assumption, it can thus be deduced that $\mathbf{f}(T) - \nu \mathbf{u} - |\mathbf{u}|\mathbf{u} = \nabla \mathsf{p} \in \mathbf{L}^2(\Omega)$. Observe that $T \in L^2(\Omega)$. Consequently, the construction of the finite element spaces allows us to \GC{guarantee that}, for every $K \in \T_h$, the term $\| |\mathbf{u}|\mathbf{u}-|\mathbf{u}_h|\mathbf{u}_h+\nabla e_{\mathsf{p}} \|_{\mathbf{L}^2(K)}$ is well-defined.
\end{remark}

\begin{remark}[auxiliary estimates]
\label{eq:lemma_press_nonlinear}
We notice that the right-hand side of the \GC{efficiency} estimate \eqref{eq:local_efficiency_darcya} contains the local terms $\||\mathbf{u}|\mathbf{u}-|\mathbf{u}_h|\mathbf{u}_h+\nabla e_{\mathsf{p}}\|_{\mathbf{L}^2(\mathcal{N}_K)}$ and $\| \mathbf{v}_r \|_{\mathbf{L}^3(\mathcal{N}_K)}$. The global version of these terms \emph{are not contained} on the left-hand side of the global reliability bound \eqref{eq:upper_bound}. The control of $\| \mathbf{v}_r \|_{\mathbf{L}^3(\Omega)}$ follows immediately from \eqref{eq:bound_vr}: 
$\| \mathbf{v}_r \|_{\mathbf{L}^3(\Omega)} \lesssim \sum_{K \in \T_h} \mathfrak{E}_{K}$. In what follows, we briefly elaborate an argument that allows us to control $\||\mathbf{u}|\mathbf{u}-|\mathbf{u}_h|\mathbf{u}_h+\nabla e_{\mathsf{p}}\|_{\mathbf{L}^2(\Omega)}$. Set $\mathbf{v}=|\mathbf{u}|\mathbf{u}-|\mathbf{u}_h|\mathbf{u}_h+\nabla e_{\mathsf{p}} \in \mathbf{L}^2(\Omega)$ in \eqref{eq:error_equation_darcy_1}. We notice that this is possible because of the arguments in Remark \ref{eq:remark}. We thus invoke H\"older's inequality and the Lipschitz property of $\mathbf{f}_1$ to arrive at
\begin{equation*}
\label{eq:local_efficiency_press_nolinear}
\||\mathbf{u}|\mathbf{u} -|\mathbf{u}_h|\mathbf{u}_h\!+\!\nabla e_{\mathsf{p}}\|_{\mathbf{L}^2(\Omega)}\lesssim
\|\mathbf{e}_{\mathbf{u}}\|_{\mathbf{L}^2(\Omega)}
+ \mathfrak{E}_{\T} + \| \nabla e_T \|_{\mathbf{L}^p(\Omega)} \lesssim \mathsf{E}_{\T},
\end{equation*}
upon utilizing the global reliability bound \eqref{eq:upper_bound}.
\end{remark}


\section{Numerical experiments}
\label{sec:numericalexperiments}

We conduct a series of two-dimensional numerical examples that illustrate the performance of the error estimator $\mathsf{E}_{\T}$ defined in \eqref{def:total_error_estimator}. These examples have been carried out with the help of a code that we implemented using \texttt{C++}. All matrices have been assembled exactly and global linear systems were solved using the multifrontal massively parallel sparse direct solver (MUMPS) \cite{MUMPS1,MUMPS2}. The right-hand sides, local indicators, and the error estimator were computed by a quadrature formula which is exact for polynomials of degree 19. To visualize finite element approximations, we have used the open source application ParaView \cite{Ahrens2005ParaViewAE,Ayachit2015ThePG}.

For a given partition $\mathscr{T}_h$, we solve the discrete system \eqref{eq:model_discrete} in $\mathbf{X}_{h}\times M_h\times Y_h$ using the iterative strategy described in \textbf{Algorithm 1}. Once a discrete solution is obtained, we compute, for each $K \in \mathscr{T}_h$, the local error indicator $\mathsf{E}_{K}$, defined by
\begin{equation}\label{def:total_indicator}
\mathsf{E}_K:=\mathcal{E}_{p,K}+\mathfrak{E}_{K},
\end{equation}
to drive the adaptive procedure described in \textbf{Algorithm 2}. A sequence of adaptively refined meshes is thus generated from the initial meshes shown in Figure \ref{fig:mesh}.

\begin{figure}[!ht]
\centering
\hspace{-1.0cm}
\begin{minipage}[b]{0.35\textwidth}\centering
\includegraphics[width=1.8cm,height=1.8cm,scale=0.66]{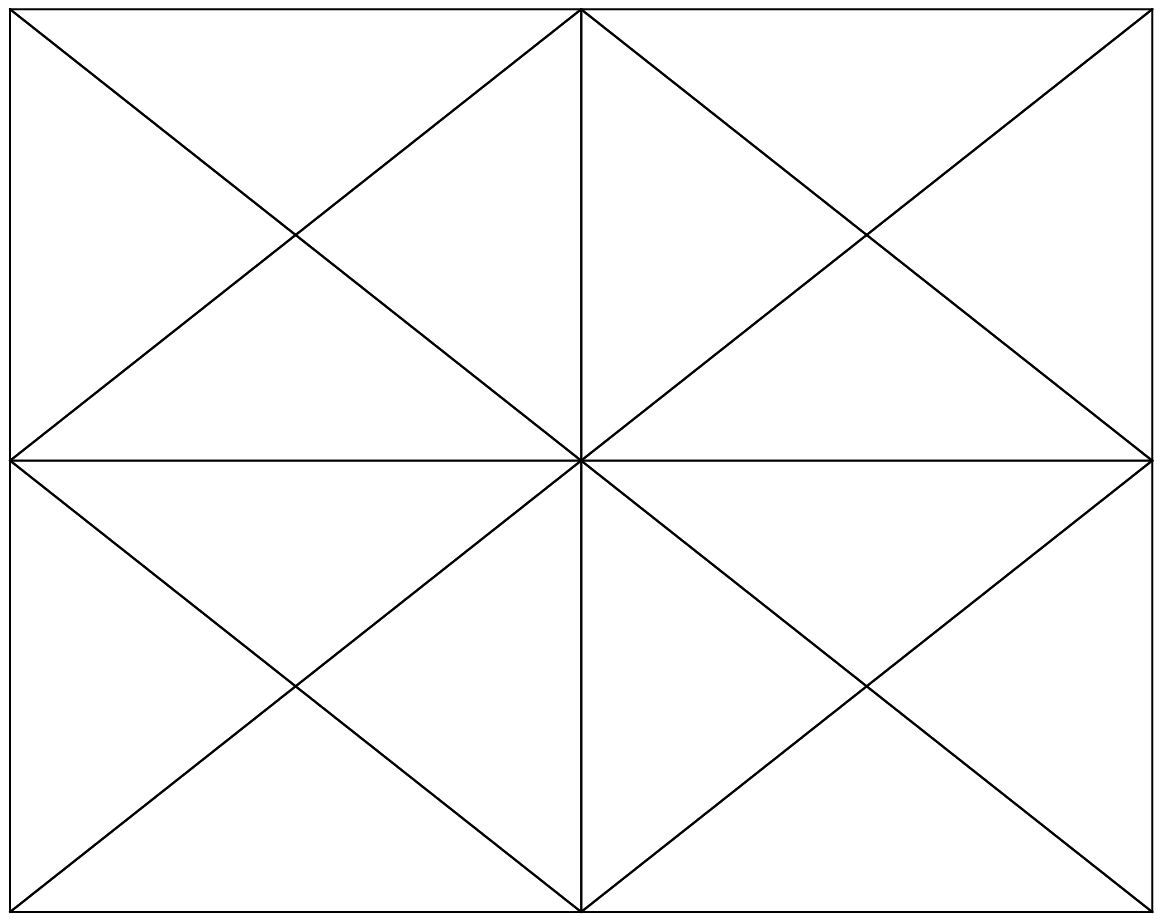} \\
\tiny{(A.1)}
\end{minipage}
\qquad
\begin{minipage}[b]{0.35\textwidth}\centering
\includegraphics[width=1.8cm,height=1.8cm,scale=0.66]{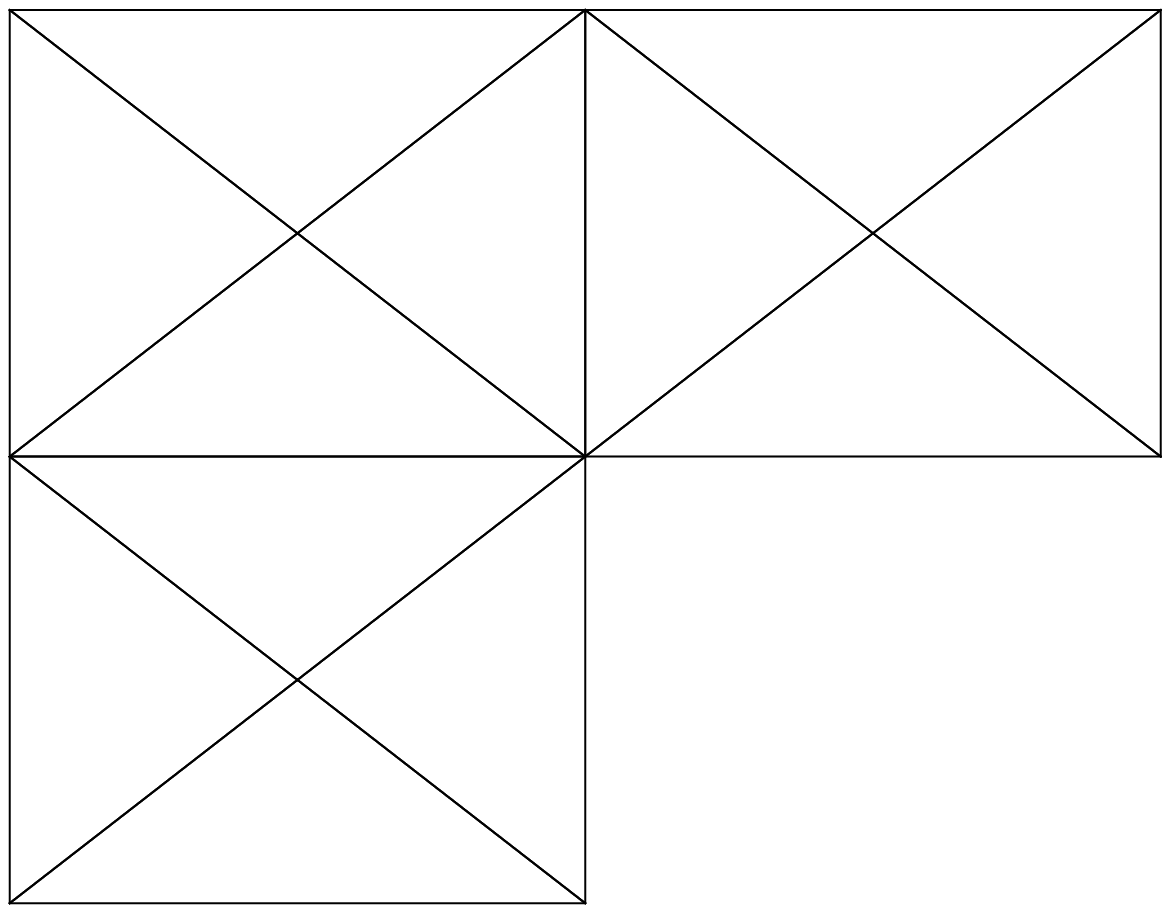} \\
\qquad \tiny{(A.2)}
\end{minipage}
\caption{The initial meshes used in the adaptive algorithm, \emph{Algorithm} 1, when \textrm{(A.1)} $\Omega = (0,1)^2$ and \textrm{(A.2)} $\Omega=(-1,1)^2 \setminus[0,1)\times (-1,0]$.}
\label{fig:mesh}
\end{figure}

In some of the numerical examples that we perform we go beyond the presented theory and include a series of Dirac delta sources on the right-hand side of the temperature equation. To be precise, we consider $\textnormal{g} = \sum_{z\in\mathcal{D}}\delta_z$, where $\mathcal{D}$ denotes a finite subset of $\Omega$ with cardinality $\#\mathcal{D}$. Within this setting, we \GC{suitably} modify the error estimator $\mathcal{E}_{p,\T}$, associated to the discretization of the heat equation, as follows:
\begin{equation}\label{eq:new_total_estimator_temp}
\mathcal{E}_{p,\T}:=\left( \sum_{K\in\GC{\mathscr{T}_h}} \mathcal{E}_{p,K}^p\right)^{\frac{1}{p}},\quad \tfrac{4}{3}-\epsilon<p<2,
\end{equation}
for some $\epsilon >0$. For each $K\in\mathscr{T}_h$, the local error indicators $\mathcal{E}_{p,K}$ are given now as follows: if $z\in\mathcal{D}\cap K$ and $z$ is not a vertex of $K$, then
\begin{equation}
\label{eq:local_indicator_tempI_seriesdirac}
\mathcal{E}_{p,K}:=
\left[
\sum_{z\in\mathcal{D}\cap K}h_K^{2-p} 
+
h_K^p \| \mathcal{R}_K
\|_{L^p(K)}^p 
+
h_K \| \mathcal{J}_{\gamma}
\|_{L^p(\partial K \setminus \partial \Omega)}^p
\right]^{\frac{1}{p}}.
\end{equation}
If $z\in\mathcal{D}\cap K$ and $z$ is a vertex of $K$, then
\begin{equation}\label{eq:local_indicatorII_seriesdirac}
\mathcal{E}_{p,K}:=
\left[ 
h_K^p \| \mathcal{R}_K  \|_{L^p(K)}^p 
+
h_K \|
\mathcal{J}_{\gamma}
\|_{L^p(\partial K \setminus \partial \Omega)}^p
\right]^{\frac{1}{p}}.
\end{equation}
If $\mathcal{D}\cap K=\emptyset$, then the indicator $\mathcal{E}_{p,K}$ is defined as in \eqref{eq:local_indicatorII_seriesdirac}. We notice that the previous modification is not needed if $\#\mathcal{D}=1$; \eqref{eq:new_total_estimator_temp} and \eqref{eq:error_estimator_heat} coincide.

\begin{algorithm}[ht]
\caption{\textbf{Iterative Scheme}}
\label{Algorithm1}
\textbf{Input:} Initial guess $(\mathbf{u}_{h}^{0},\mathsf{p}_{h}^{0},T_{h}^{0}) \in \mathbf{X}_h \times M_h \times Y_h$ and $\textrm{tol} = 10^{-8}$;
\\
$\boldsymbol{1}$: For $i\geq 0$, find $(\mathbf{u}_{h}^{i+1}, \mathsf{p}_{h}^{i+1}) \in \mathbf{X}_{h}\times M_h$ such that
\begin{equation*}
\begin{array}{rcll}
\displaystyle\int_\Omega (\nu\mathbf{u}_{h}^{i+1}\cdot \mathbf{v}_h+|\mathbf{u}_{h}^{i}|\mathbf{u}_{h}^{i+1}\cdot \mathbf{v}_h +  \nabla\mathsf{p}_{h}^{i+1}\cdot \mathbf{v}_h)\mathrm{d}\mathbf{x}& = & 	\displaystyle\int_\Omega\mathbf{f}(T_{h}^{i})\cdot \mathbf{v}_h\mathrm{d}\mathbf{x} \quad &\forall \mathbf{v}_h\in \mathbf{X}_h,\\
\displaystyle\int_\Omega \nabla\mathsf{q}_h\cdot\mathbf{u}_{h}^{i+1}\mathrm{d}\mathbf{x} & = & 0 \quad &\forall \mathsf{q}_h\in M_h.
\end{array}
\end{equation*}
\\
Then, $T_{h}^{i+1} \in Y_h$ is found as the solution to
\[
\displaystyle\int_\Omega(\kappa\nabla T_{h}^{i+1}\cdot \nabla S_h - T_{h}^{i+1} \mathbf{u}_{h}^{i+1}\cdot\nabla S_h)\mathrm{d}\mathbf{x}  =  
\langle \delta_{z}, S_h\rangle \qquad \forall S_h\in Y_h.
\]
$\boldsymbol{2}$: If $|(\mathbf{u}_{h}^{i+1}, \mathsf{p}_{h}^{i+1}, T_{h}^{i+1})-(\mathbf{u}_{h}^{i}, \mathsf{p}_{h}^{i}, T_{h}^{i} )| > \textrm{tol}$,  set $i \leftarrow i + 1$ and go to step $\boldsymbol{1}$. Otherwise, return $(\mathbf{u}_{h}, \mathsf{p}_{h}, T_{h}) = (\mathbf{u}_{h}^{i+1}, \mathsf{p}_{h}^{i+1}, T_{h}^{i+1} )$. Here, $|\cdot|$ denotes the Euclidean norm.
\end{algorithm}

\begin{algorithm}[ht]
\caption{\textbf{Adaptive Algorithm.}}
\label{Algorithm2}
\textbf{Input:} Initial mesh $\mathscr{T}_{0}$ and data $\nu$, $\kappa$, and $\mathbf{f}$;
\\
$\boldsymbol{1}$: Solve the discrete problem \eqref{eq:model_discrete} by using \textbf{Algorithm} \ref{Algorithm1};
\\
$\boldsymbol{2}$: For each $K\in\mathscr{T}_{h}$ compute the local error indicator $\mathsf{E}_K$ defined in \eqref{def:total_indicator};  
\\
$\boldsymbol{3}$: Mark an element $K\in\mathscr{T}_{h}$ for refinement if
\begin{equation*}
\mathsf{E}_K>\tfrac{1}{2}\max_{K'\in \mathscr{T}_{h}} \mathsf{E}_{K'};
\end{equation*}
$\boldsymbol{4}$: From step $\boldsymbol{3}$, construct a new mesh using a longest edge bisection algorithm. Set $i \leftarrow i + 1$ and go to step $\boldsymbol{1}$.
\end{algorithm}

\subsection{Convex domain}
\label{subsec:convex_domain}
Let $\Omega=(0,1)^2$, $\kappa=1$, $\nu(x_1,x_2)=\sin(x_1 x_2)+1.1$, $\mathcal{D} = \{ (0.25,0.25), (0.25,0.75), (0.75,0.25), \GC{(0.75,0.75)}\}$, $\mathbf{f}_0(x_{1},x_{2})=(1,1)^{\intercal}$, and $\mathbf{f}_1(s)=(s,s)^{\intercal}$. In Figure \ref{fig:test_01}, we report the results obtained for Example 1. We observe that optimal experimental rates of convergence are attained for all the values of the parameter $p$ that we considered: $p \in\{1.0,1.2, 1.4, 1.6, 1.8\}$. We also observe that most of the refinement is concentrated around the singular source points.

\begin{figure}[!ht]
\centering
\psfrag{Ndof(-1/2)}{\Large $\text{Ndof}^{-1/2}$}
\psfrag{Est p=1.0}{\Large $p=1.0$}
\psfrag{Est p=1.2}{\Large $p=1.2$}
\psfrag{Est p=1.4}{\Large $p=1.4$}
\psfrag{Est p=1.6}{\Large $p=1.6$}
\psfrag{Est p=1.8}{\Large $p=1.8$}

\begin{minipage}[b]{0.327\textwidth}\centering
\scriptsize{\qquad Estimator $\mathsf{E}_{\T}$}
\includegraphics[trim={0 0 0 0},clip,width=2.9cm,height=3.2cm,scale=0.66]{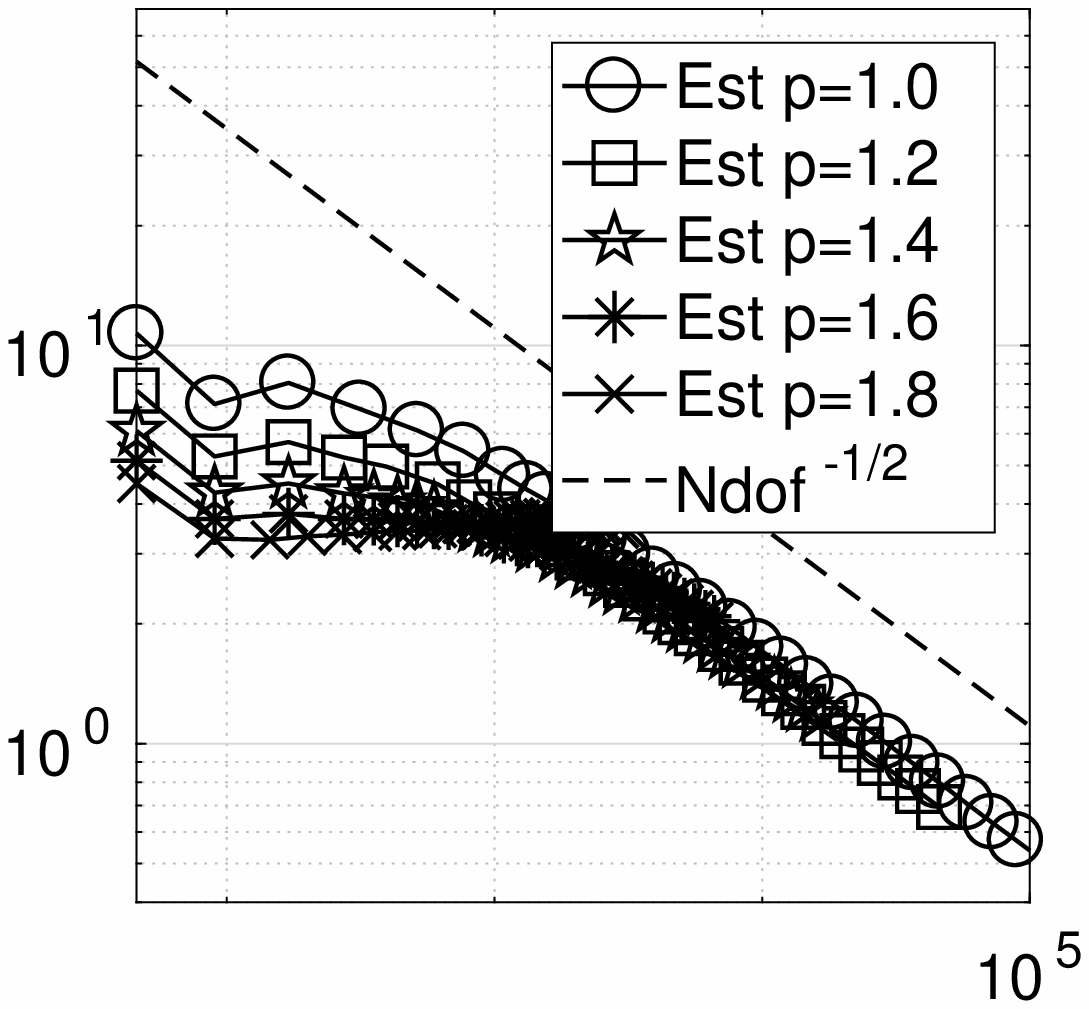} \\
\qquad \tiny{(B.1)}
\end{minipage}
\begin{minipage}[b]{0.327\textwidth}\centering
\includegraphics[width=3.7cm,height=3.5cm,scale=0.66]{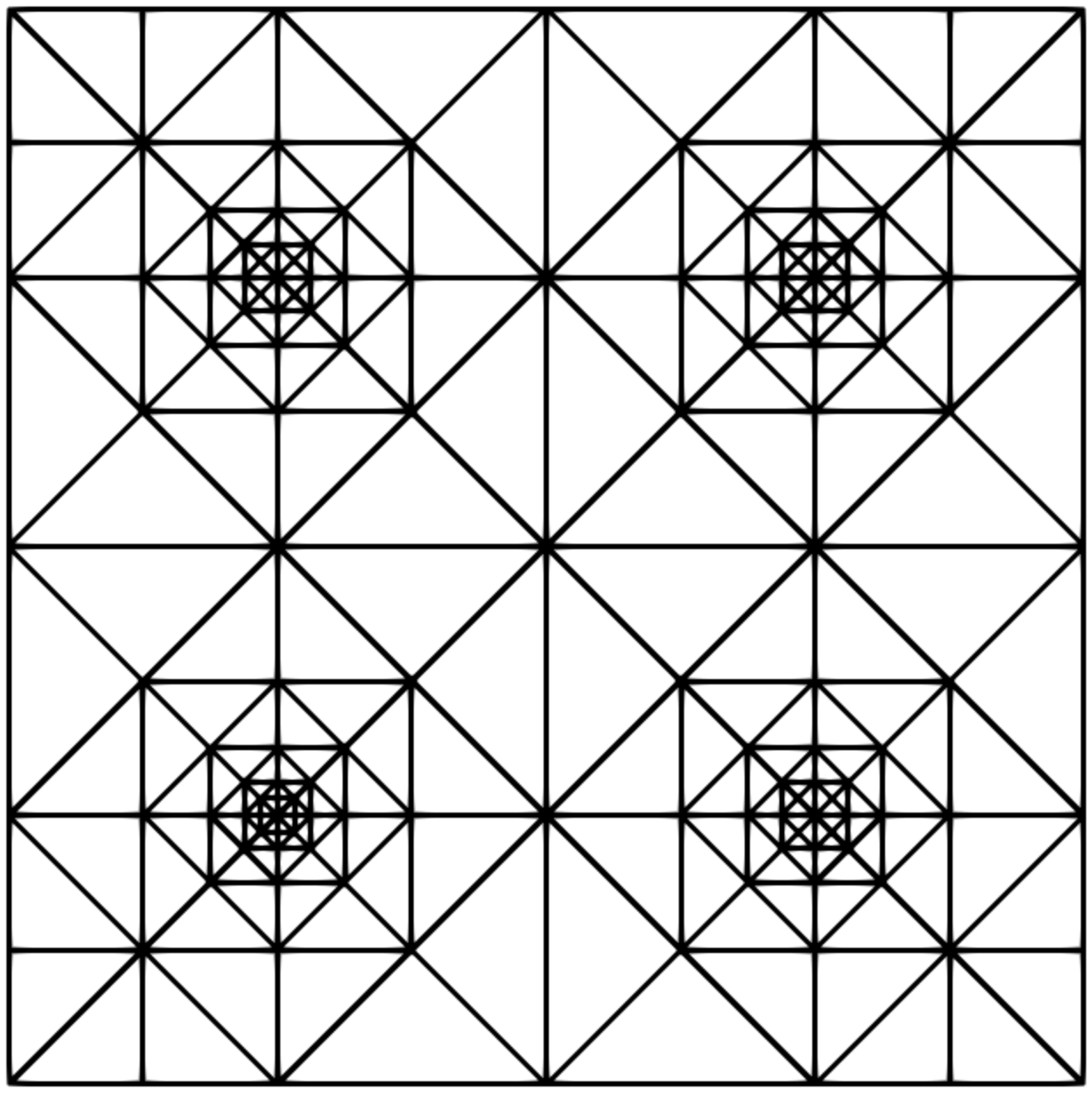} \\
\qquad \tiny{(B.2)}
\end{minipage}
\begin{minipage}[b]{0.327\textwidth}\centering
\includegraphics[width=3.7cm,height=3.5cm,scale=0.66]{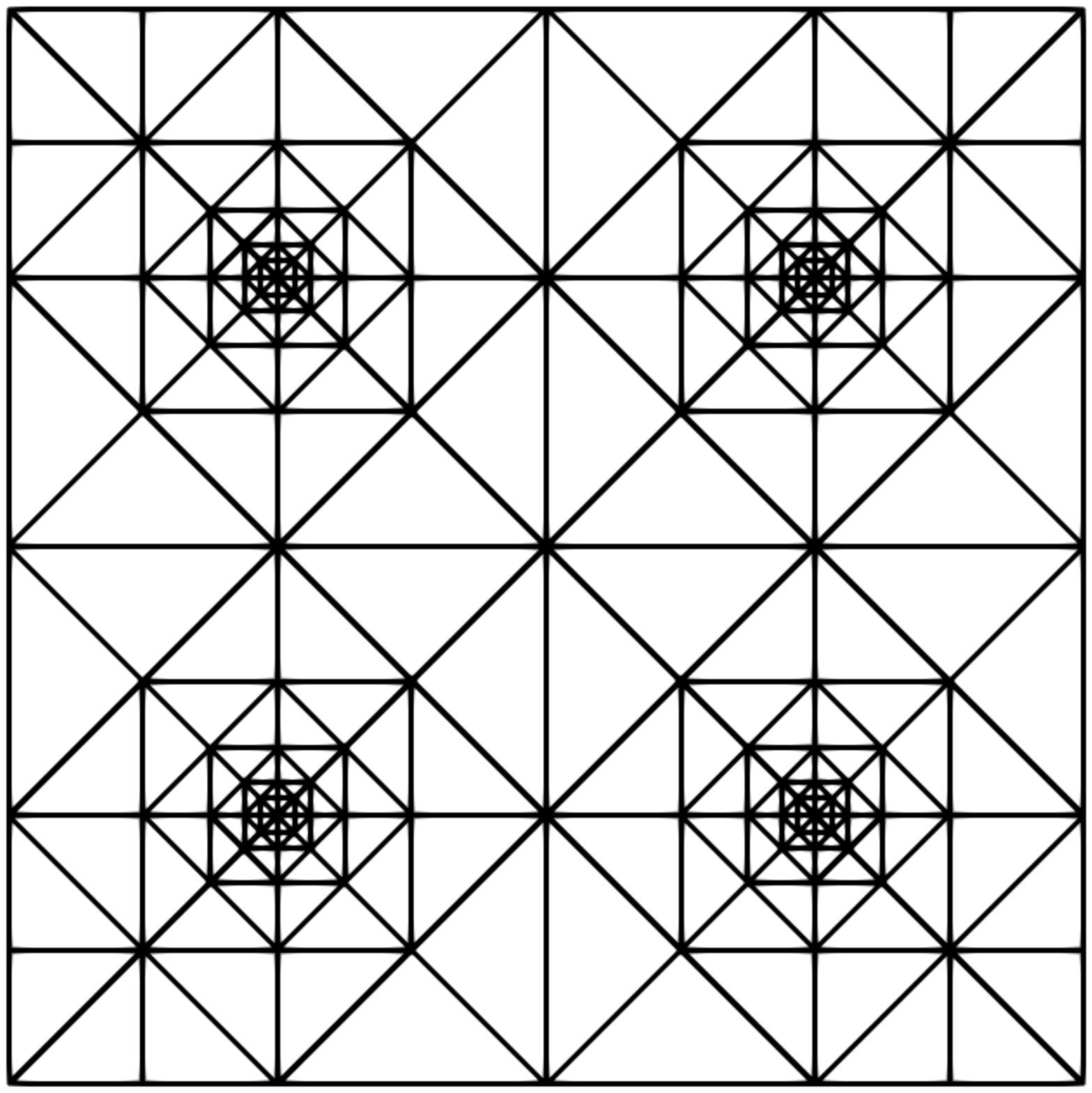} \\
\qquad \tiny{(B.3)}
\end{minipage}
\\~\\
\begin{minipage}[b]{0.327\textwidth}\centering
\includegraphics[width=3.7cm,height=3.5cm,scale=0.66]{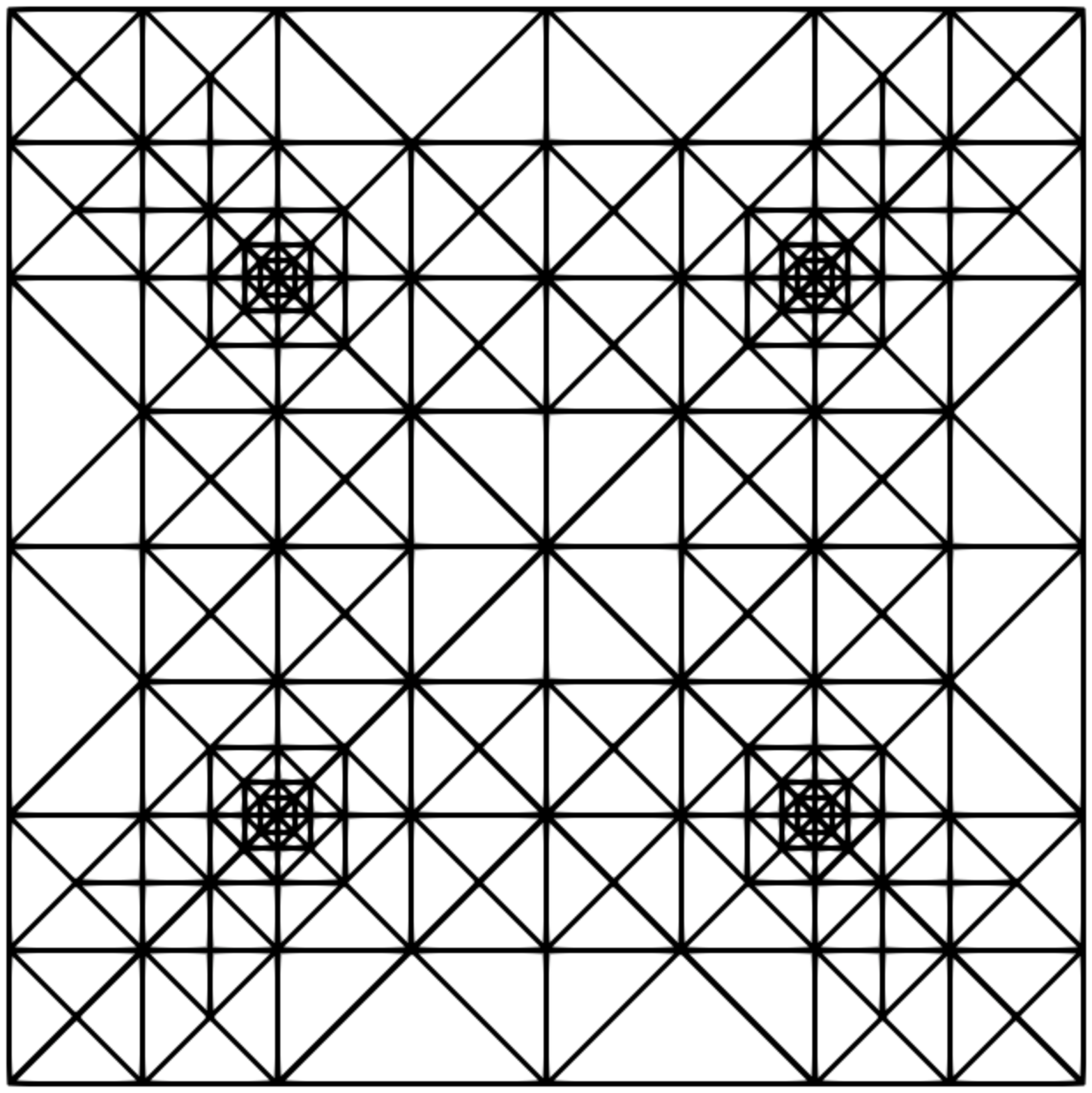} \\
\qquad \tiny{(B.4)}
\end{minipage}
\begin{minipage}[b]{0.327\textwidth}\centering 
\includegraphics[width=3.7cm,height=3.5cm,scale=0.66]{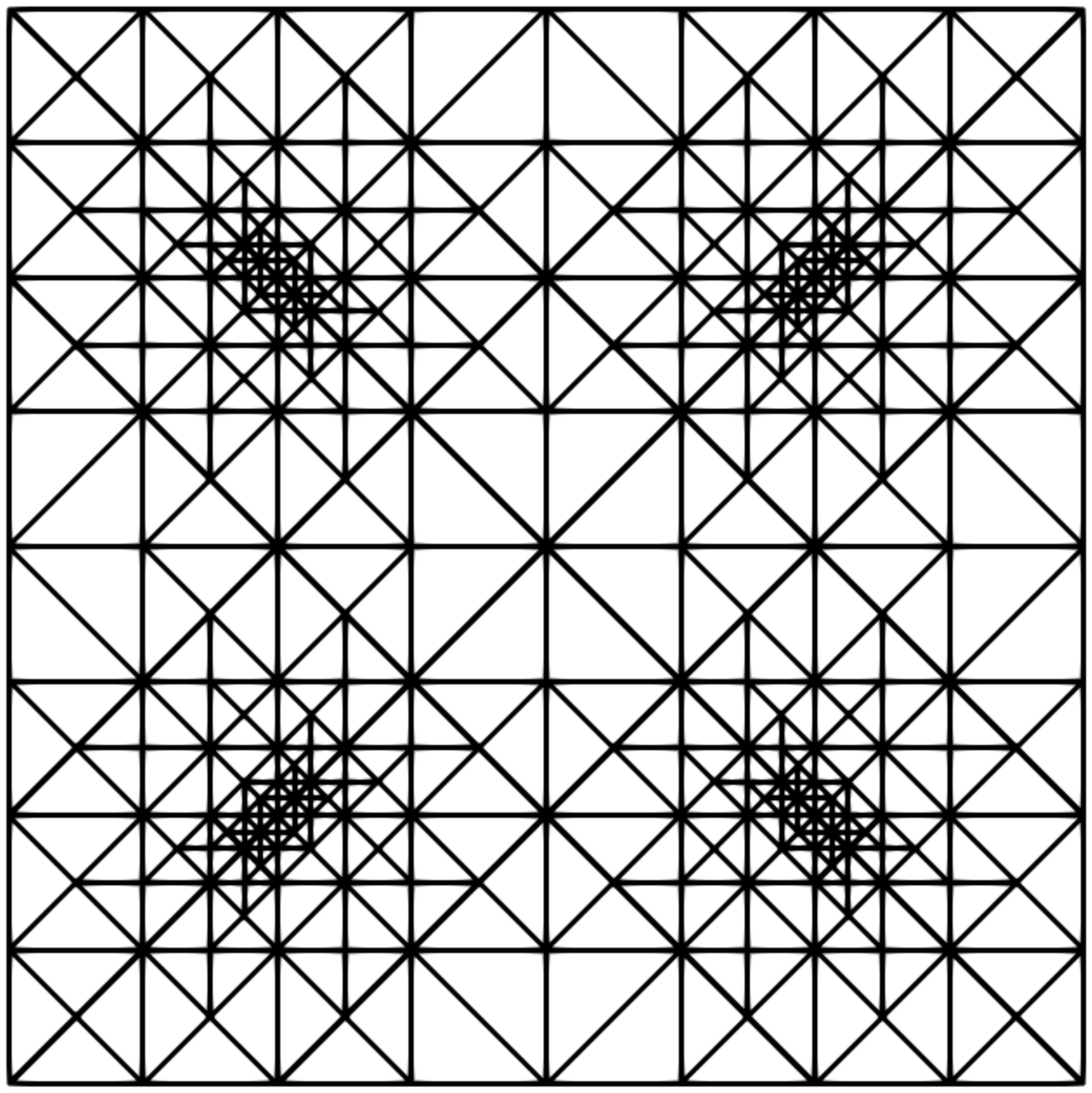} \\
\qquad \tiny{(B.5)}
\end{minipage}
\begin{minipage}[b]{0.327\textwidth}\centering  
\includegraphics[width=3.7cm,height=3.5cm,scale=0.66]{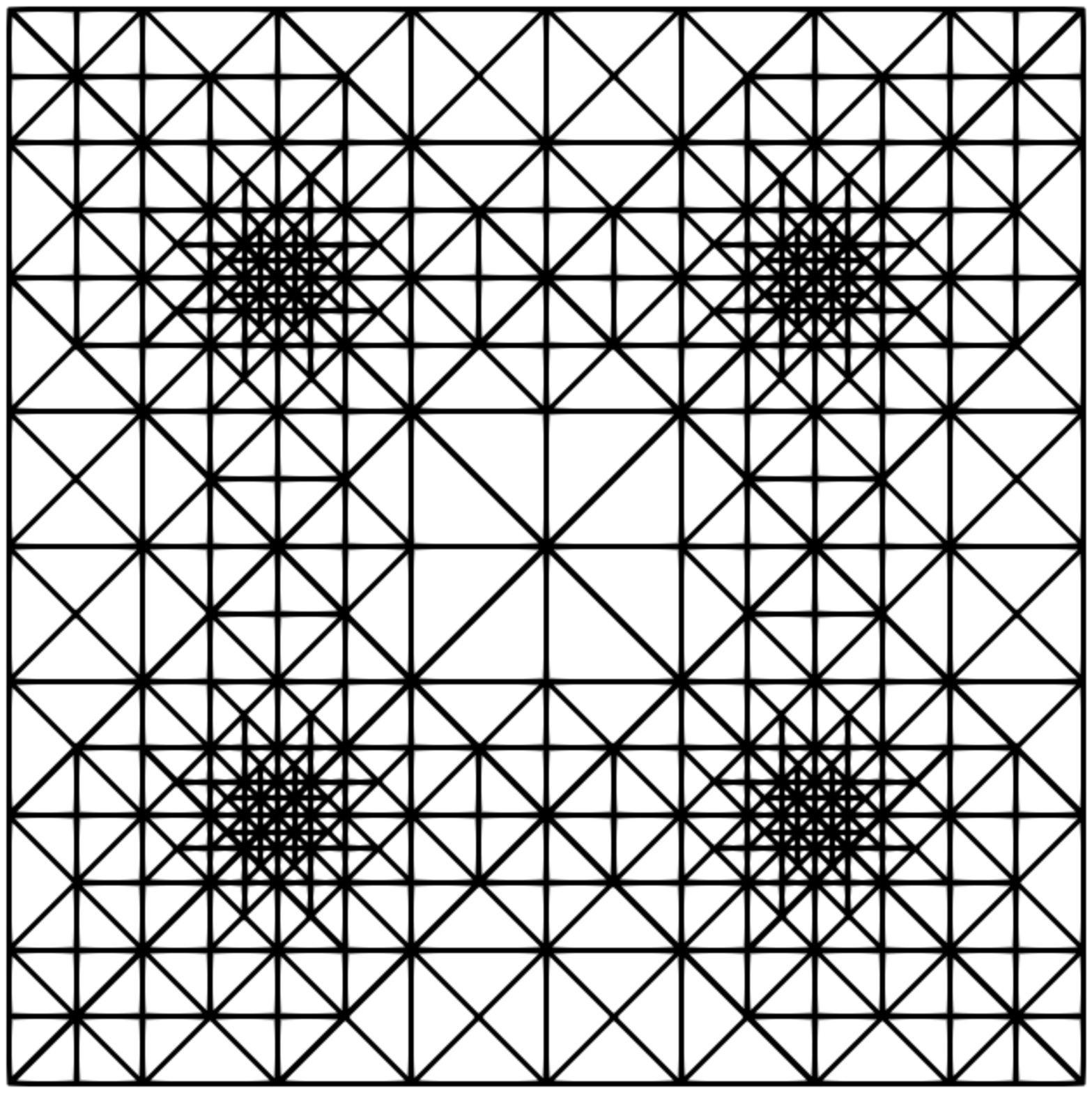} \\
\qquad \tiny{(B.6)}
\end{minipage}
\caption{Example 1: Experimental rates of convergence for the error estimator $\mathsf{E}_{\T}$ considering $p \in \{1.0,1.2, 1.4, 1.6, 1.8\}$ (B.1) and meshes obtained after 10 iterations of our adaptive loop for $p=1.8$ (B.2), $p=1.6$ (B.3), $p=1.4$ (B.4), $p=1.2$ (B.5), and $p=1.0$ (B.6).}
\label{fig:test_01}
\end{figure}

\subsection{Non-convex domain} 
Let $\Omega=(-1,1)^2 \setminus[0,1)\times (-1,0]$, $\kappa=1$, $\nu=1$, $\mathbf{f}_0(x_{1},x_{2})=(0,0)^{\intercal}$, $\mathbf{f}_1(s)=(10s,10s)^{\intercal}$, and $\mathcal{D} = \{(-0.25,0.5)\}$. In Figure \ref{fig:test_03}, we report the results obtained for Example 2. In spite of the involved geometric singularity, our estimator exhibits optimal experimental rates of convergence for $p\in\{1.0,1.6\}$. We also observe that refinement is concentrated around the singular source point and the \GC{reentrant corner.}

\begin{figure}[!ht]
\centering
\psfrag{Ndof(-1/2)}{\Large $\text{Ndof}^{-1/2}$}
\psfrag{Est p=1.0}{\Large $p=1.0$}
\psfrag{Est p=1.6}{\Large $p=1.6$}
\begin{minipage}[b]{0.327\textwidth}\centering
\scriptsize{\qquad Estimator $\mathsf{E}_{\T}$}
\includegraphics[trim={0 0 0 0},clip,width=3.0cm,height=3.3cm,scale=0.66]{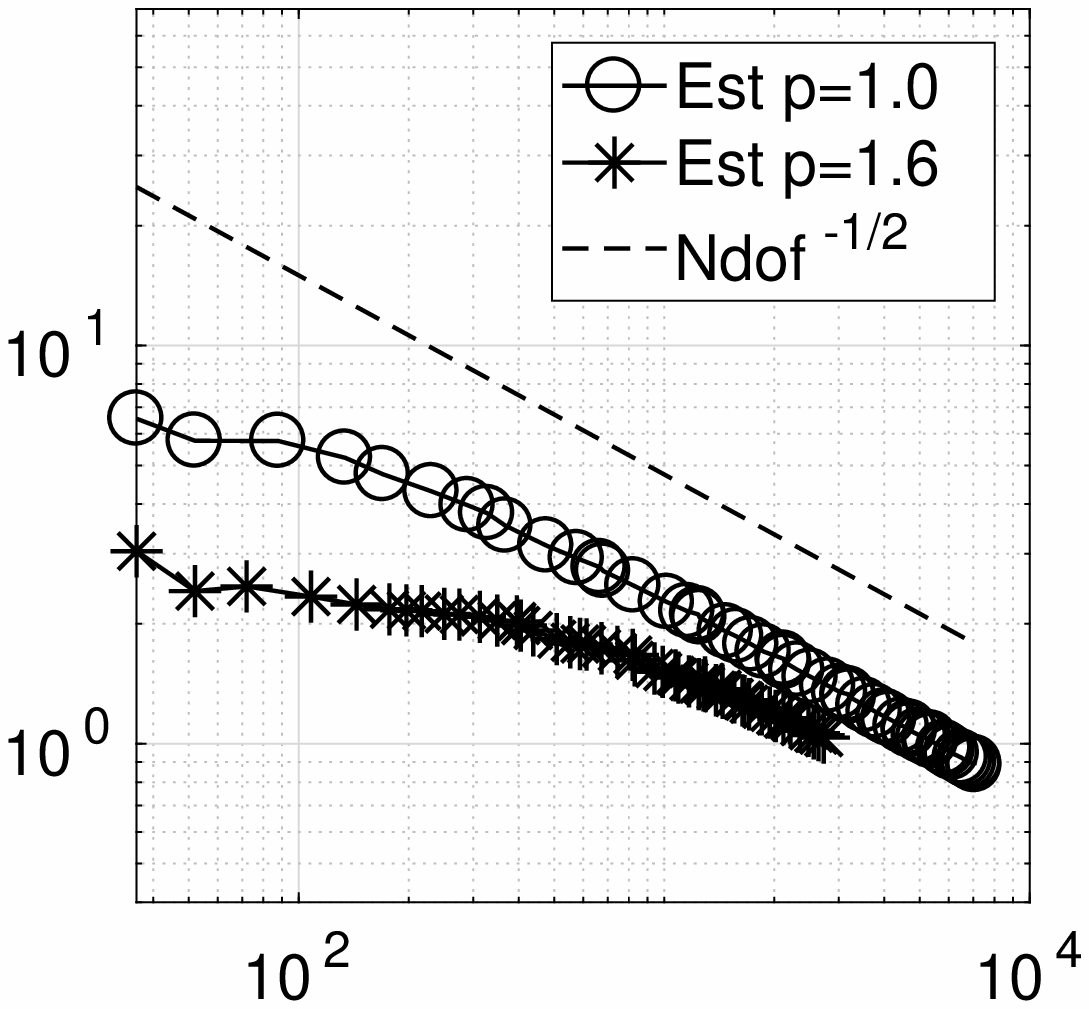} \\
\qquad \tiny{(C.1)}
\end{minipage}
\begin{minipage}[b]{0.327\textwidth}\centering
\includegraphics[width=3.6cm,height=3.6cm,scale=0.66]{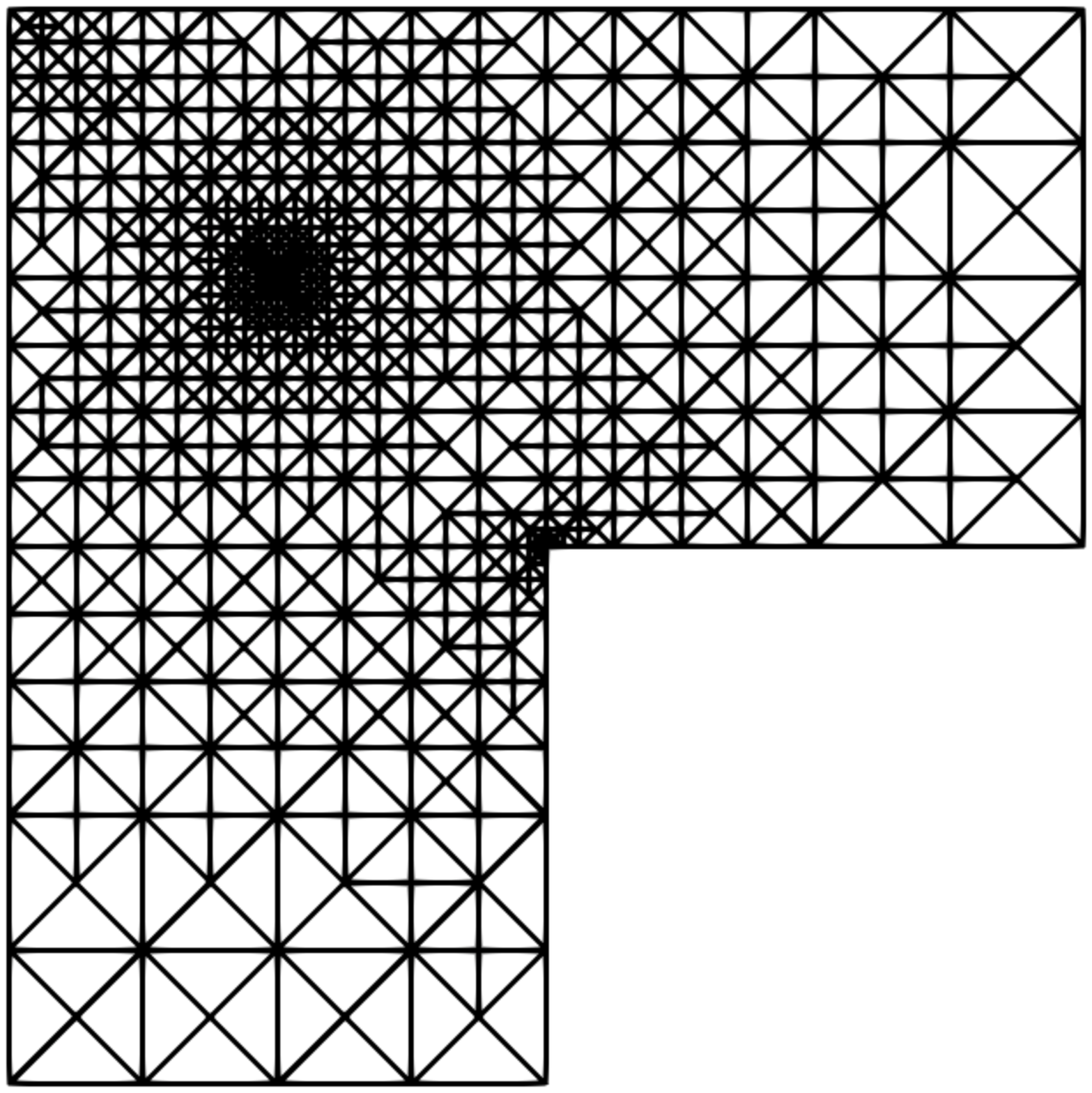} \\
\qquad \tiny{(C.2)}
\end{minipage}
\begin{minipage}[b]{0.327\textwidth}\centering
\includegraphics[width=3.6cm,height=3.6cm,scale=0.66]{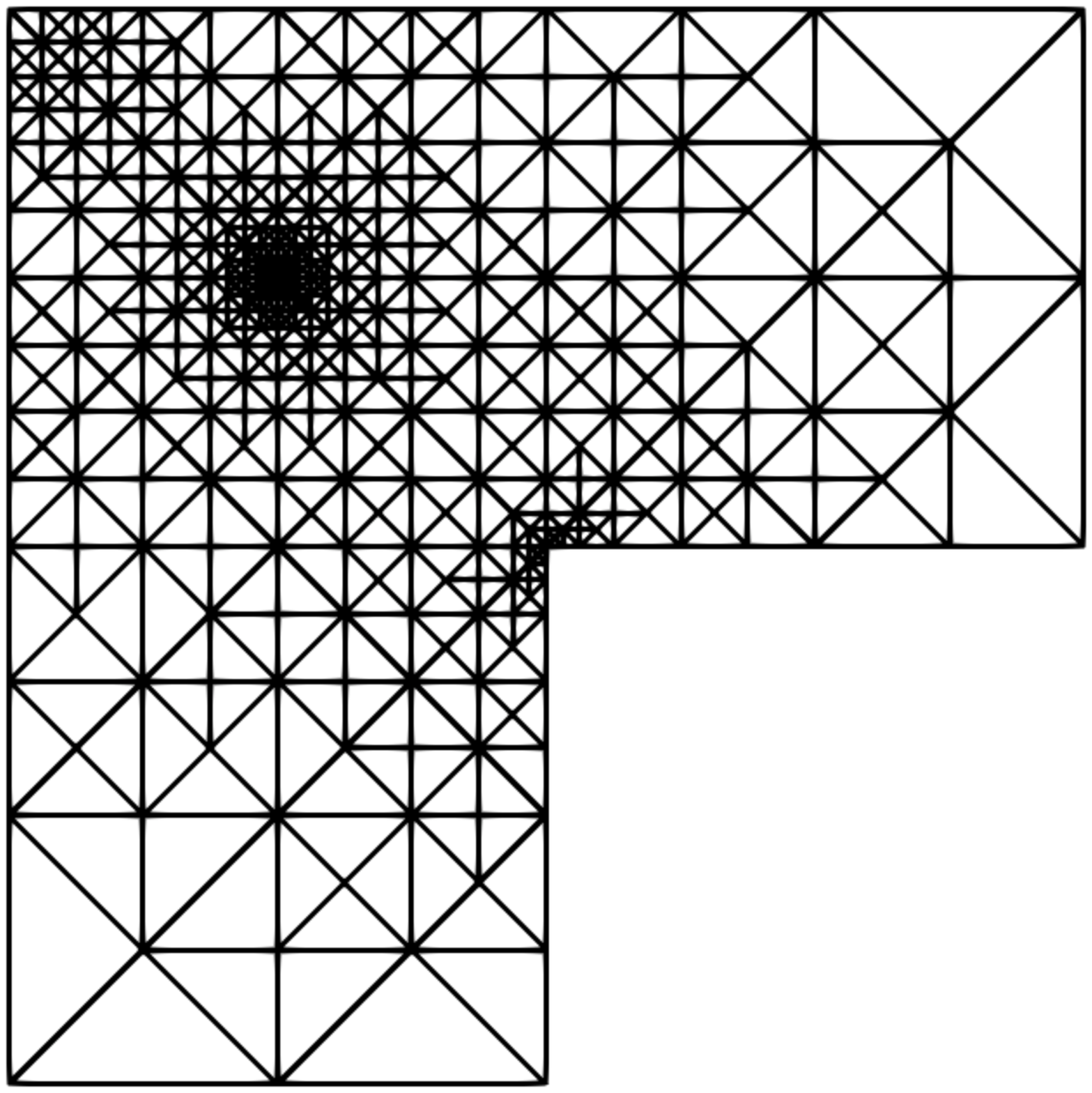} \\
\qquad \tiny{(C.3)}
\end{minipage}
\\~\\
\begin{minipage}[b]{0.327\textwidth}\centering
\includegraphics[width=3.6cm,height=3.6cm,scale=0.66]{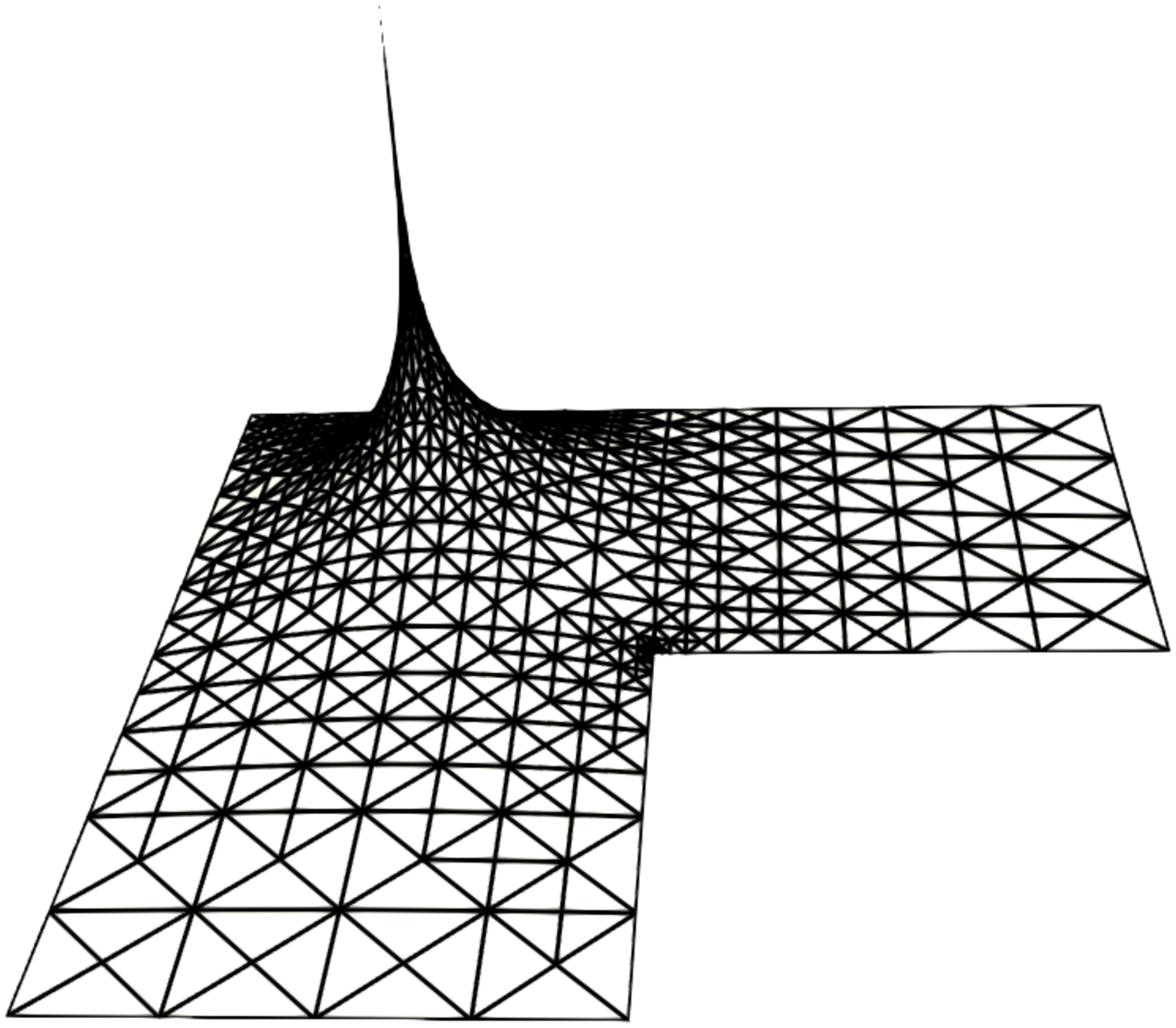} \\
\qquad \tiny{(C.4)}
\end{minipage}
\begin{minipage}[b]{0.327\textwidth}\centering 
\includegraphics[width=3.6cm,height=3.6cm,scale=0.66]{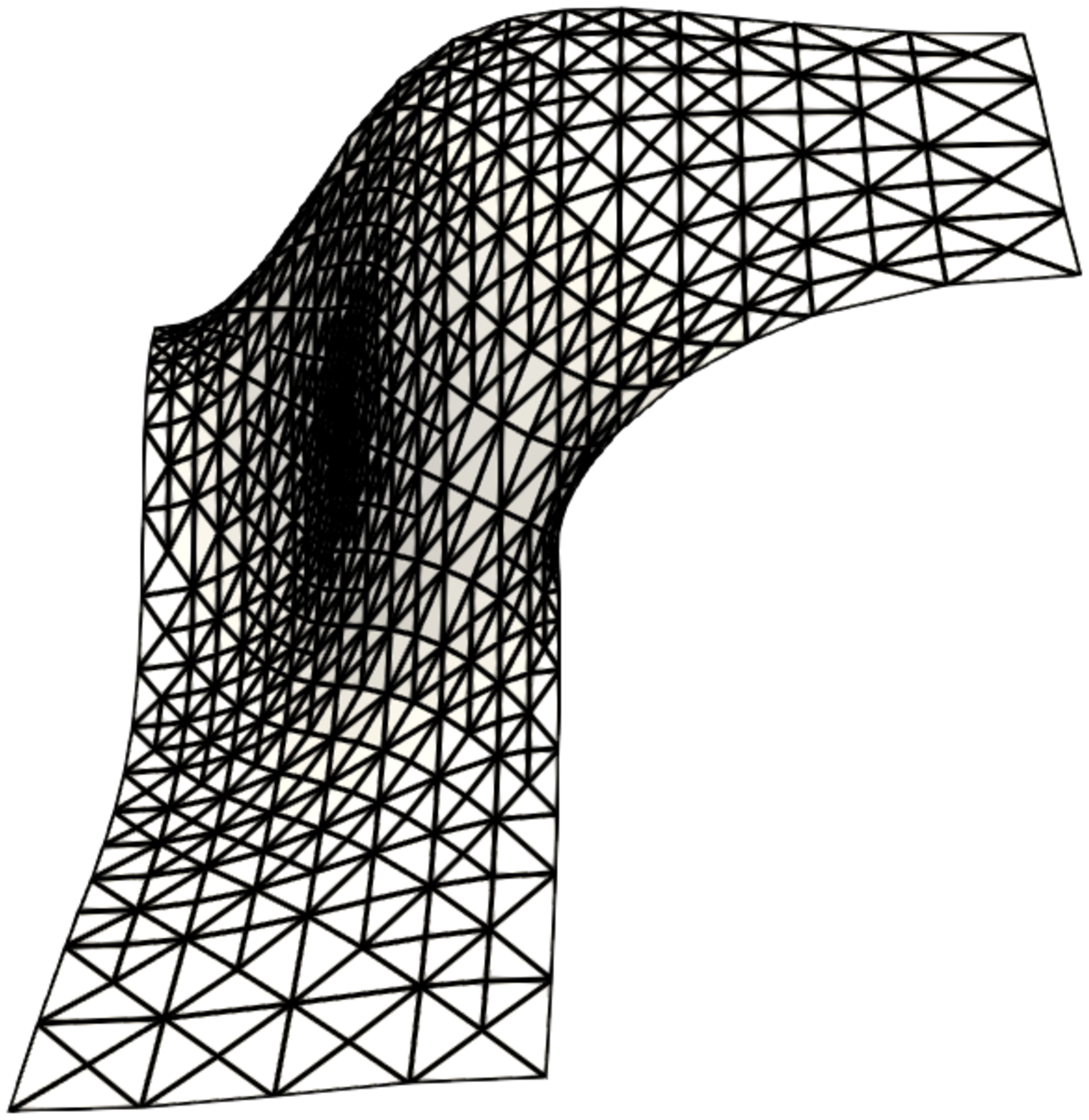} \\
\qquad \tiny{(C.5)}
\end{minipage}
\begin{minipage}[b]{0.327\textwidth}\centering  
\includegraphics[width=3.6cm,height=3.6cm,scale=0.66]{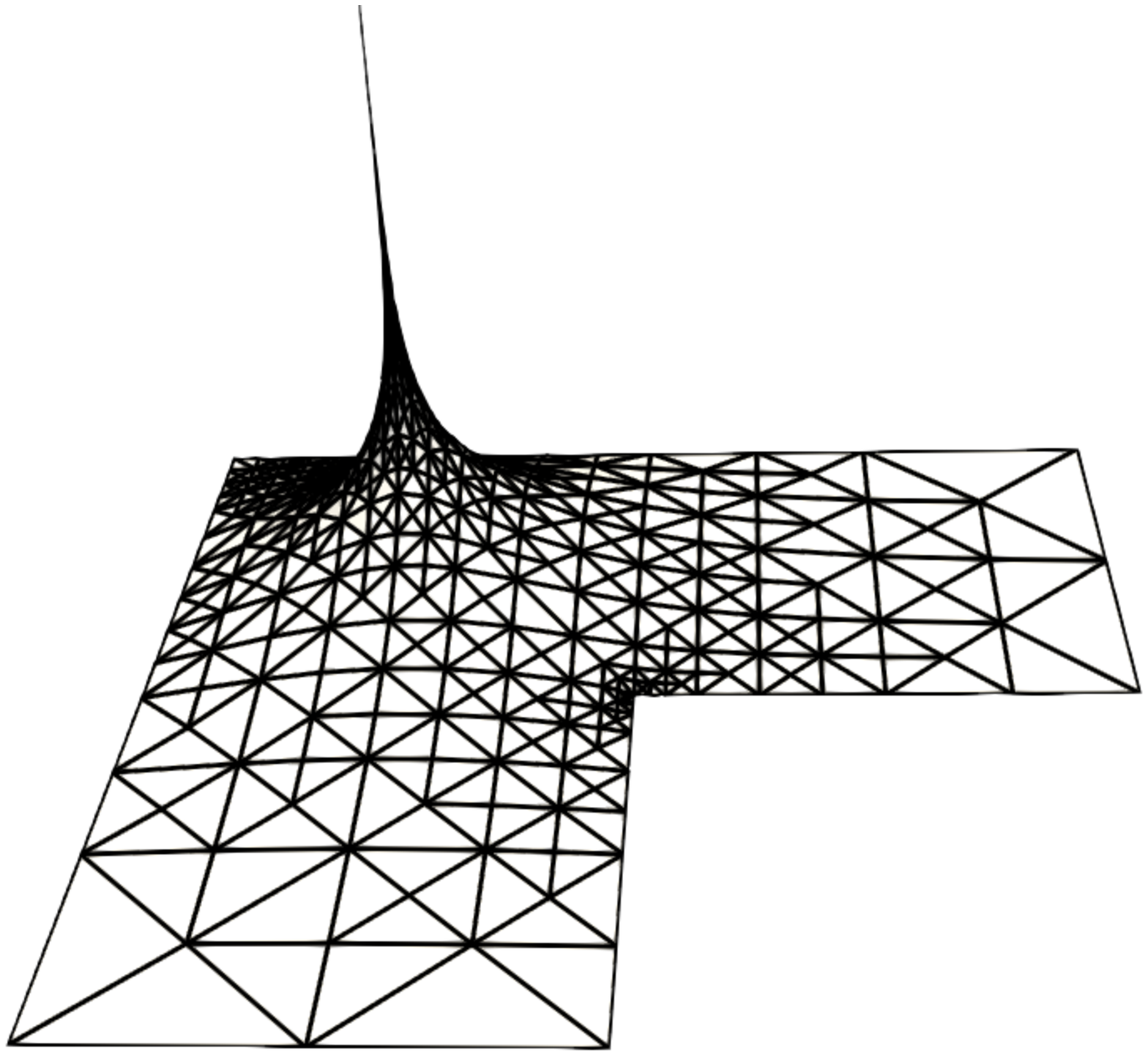} \\
\qquad \tiny{(C.6)}
\end{minipage}
\caption{Example 2: Experimental rates of convergence for the error estimator $\mathsf{E}_{\T}$ considering $p\in\{1.0,1.6\}$ (C.1), the mesh obtained after 30 iterations of our adaptive loop for $p=1.0$ (665 vertices and 1269 elements) (C.2), the mesh obtained after 50 iterations of our adaptive loop for $p=1.6$ (565 vertices and 1082 elements) (C.3), the discrete temperature field obtained within the setting of (C.2) for $p=1.0$ (C.4), the discrete pressure field obtained within the setting of (C.2) for $p=1.0$ (C.5), and the temperature field obtained within the setting of (C.3) for $p=1.6$ (C.6).}
\label{fig:test_03}
\end{figure}

\subsection{The five-spot problem}
Due to its practical importance, the quarter five-spot problem has served as a paradigm to validate \GC{the} stability and accuracy of numerical methods for fluids in porous media \cite{MR1925889}. In what follows, we address this problem within the following setting: Let $\Omega=(0,1)^2$, $\mathbf{f}_{0}:=\left(0,0\right)^{\intercal}$, $\mathbf{f}_{1}(s)=(s,s)^{\intercal}$, $\nu=1$, $\kappa=1$, and $\mathcal{D} = \{ (0.5,0.5)\}$. Instead of modeling injection and production of well by a non-zero source term in the mass conservation equation, we consider a non-homogeneous boundary condition $\mathbf{u}\cdot \mathbf{n}=1$ for all boundary edges that contain the source $(0,0)$ and $\mathbf{u}\cdot \mathbf{n}=-1$ for all boundary edges that contain the sink $(1,1)$.

In Figure \ref{fig:test_02} we report the results obtained for Example 3. We present elevations for the velocity field, pressure, and temperature, streamlines for the velocity field, and pressure and temperature contours on a suitable mesh. We observe an optimal experimental rate of convergence for $p=1.0$.  We also observe that the adaptive refinement is mostly concentrated around the singular source point $(0.5,0.5)$, the source $(0,0)$, and the sink $(1,1)$; see \cite[Section 4.2]{MR1925889} and \cite[Section 5.3]{MR2728970}.

\begin{figure}[!ht]
\centering
\psfrag{Ndof(-1/2)}{\Large $\text{Ndof}^{-1/2}$}
\psfrag{Est p=1.0}{\Large $p=1.0$}
\psfrag{pre}{\LARGE Pressure}
\psfrag{temperature}{\LARGE Temperature}


\begin{minipage}[b]{0.327\textwidth}\centering
\includegraphics[trim={0 0 0 0},clip,width=3.8cm,height=3.6cm,scale=0.66]{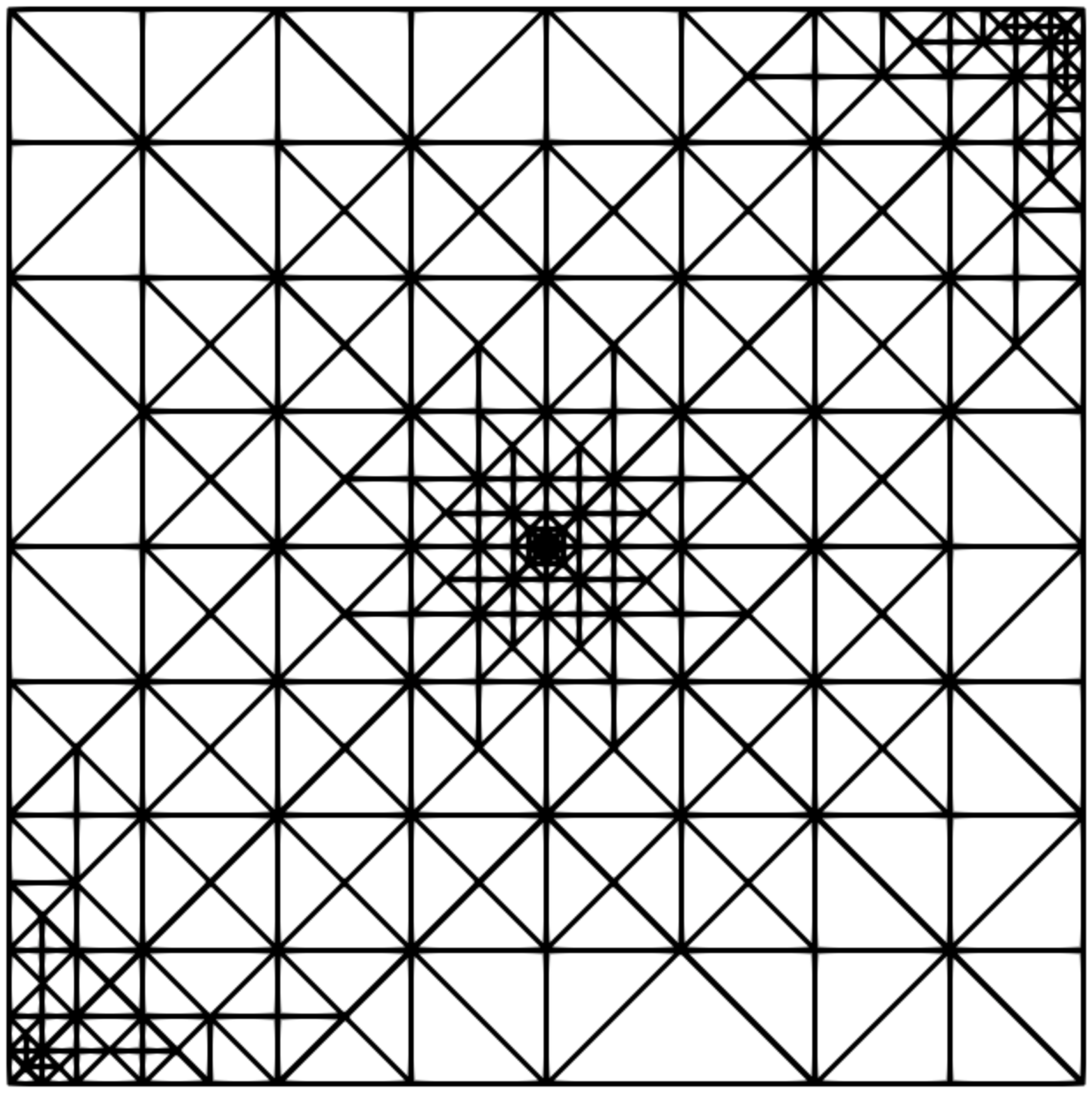} \\
\qquad \tiny{(D.1)}
\end{minipage}
\begin{minipage}[b]{0.327\textwidth}\centering
\includegraphics[width=3.8cm,height=3.6cm,scale=0.66]{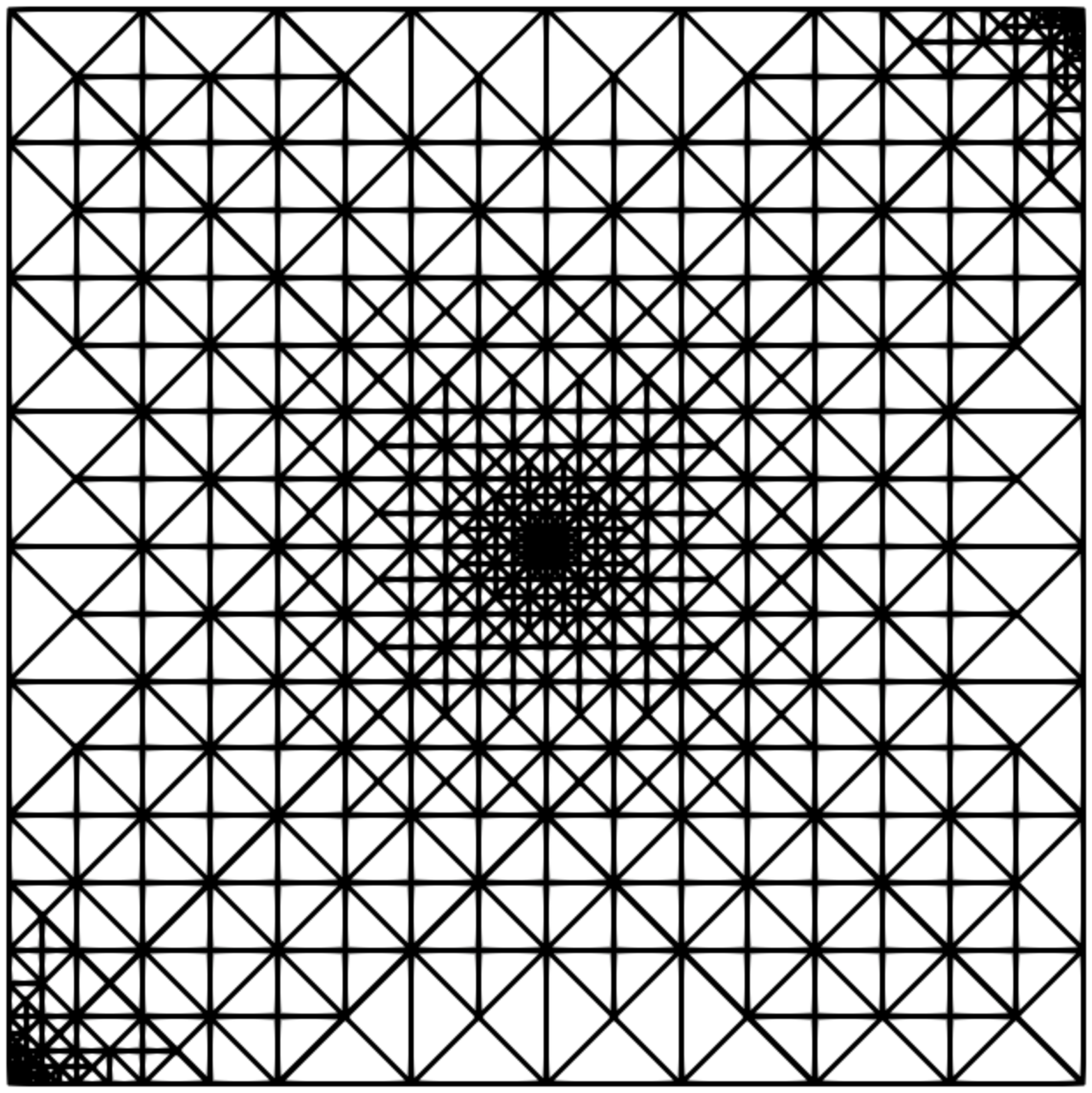} \\
\qquad \tiny{(D.2)}
\end{minipage}
\begin{minipage}[b]{0.327\textwidth}\centering
\scriptsize{\qquad Estimator $\mathsf{E}_K$}
\includegraphics[trim={0 0 0 0},clip,width=3.0cm,height=3.3cm,scale=0.66]{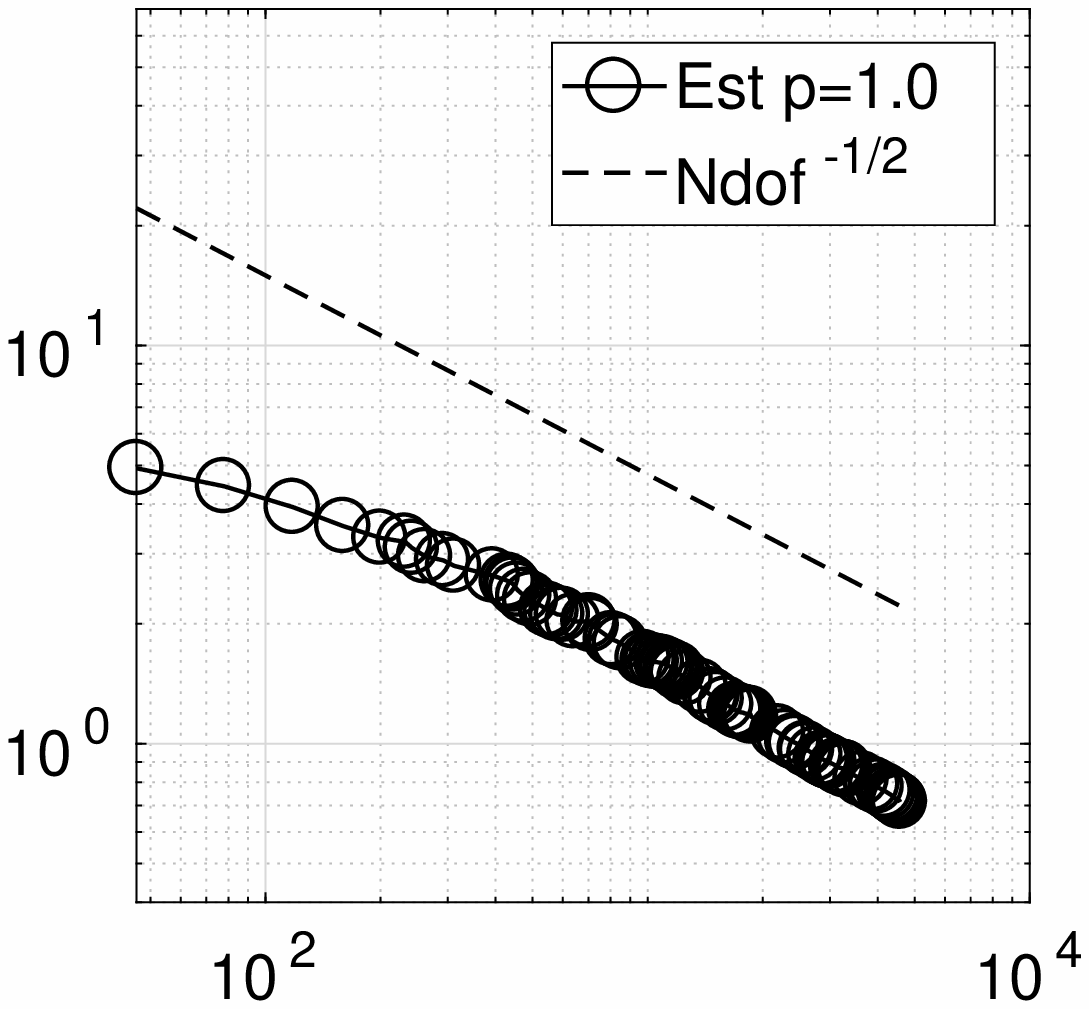} \\
\qquad \tiny{(D.3)}
\end{minipage}
\\~\\
\begin{minipage}[b]{0.327\textwidth}\centering
\includegraphics[width=3.6cm,height=3.6cm,scale=0.66]{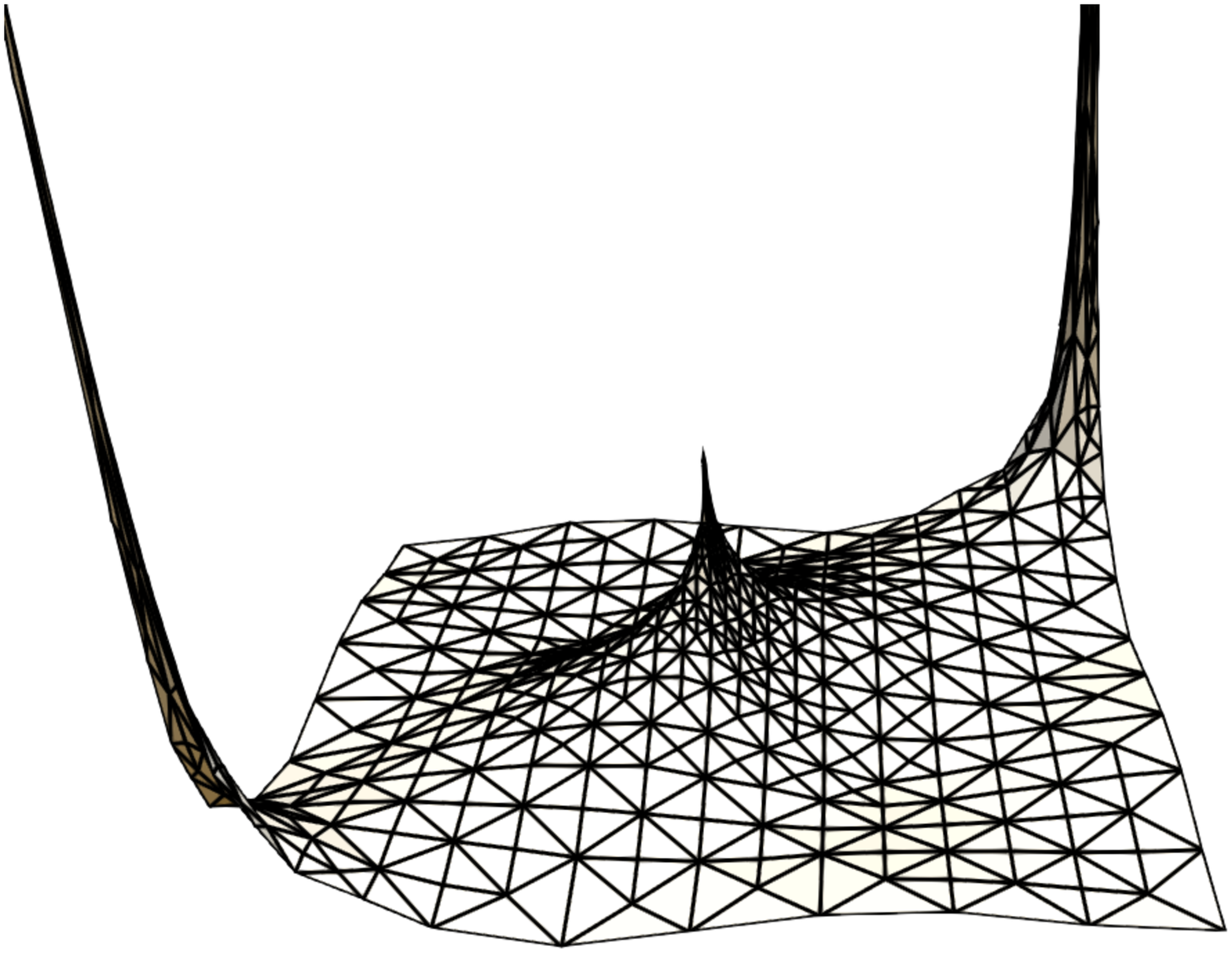} \\
\qquad \tiny{(D.4)}
\end{minipage}
\begin{minipage}[b]{0.327\textwidth}\centering 
\includegraphics[width=3.6cm,height=3.6cm,scale=0.66]{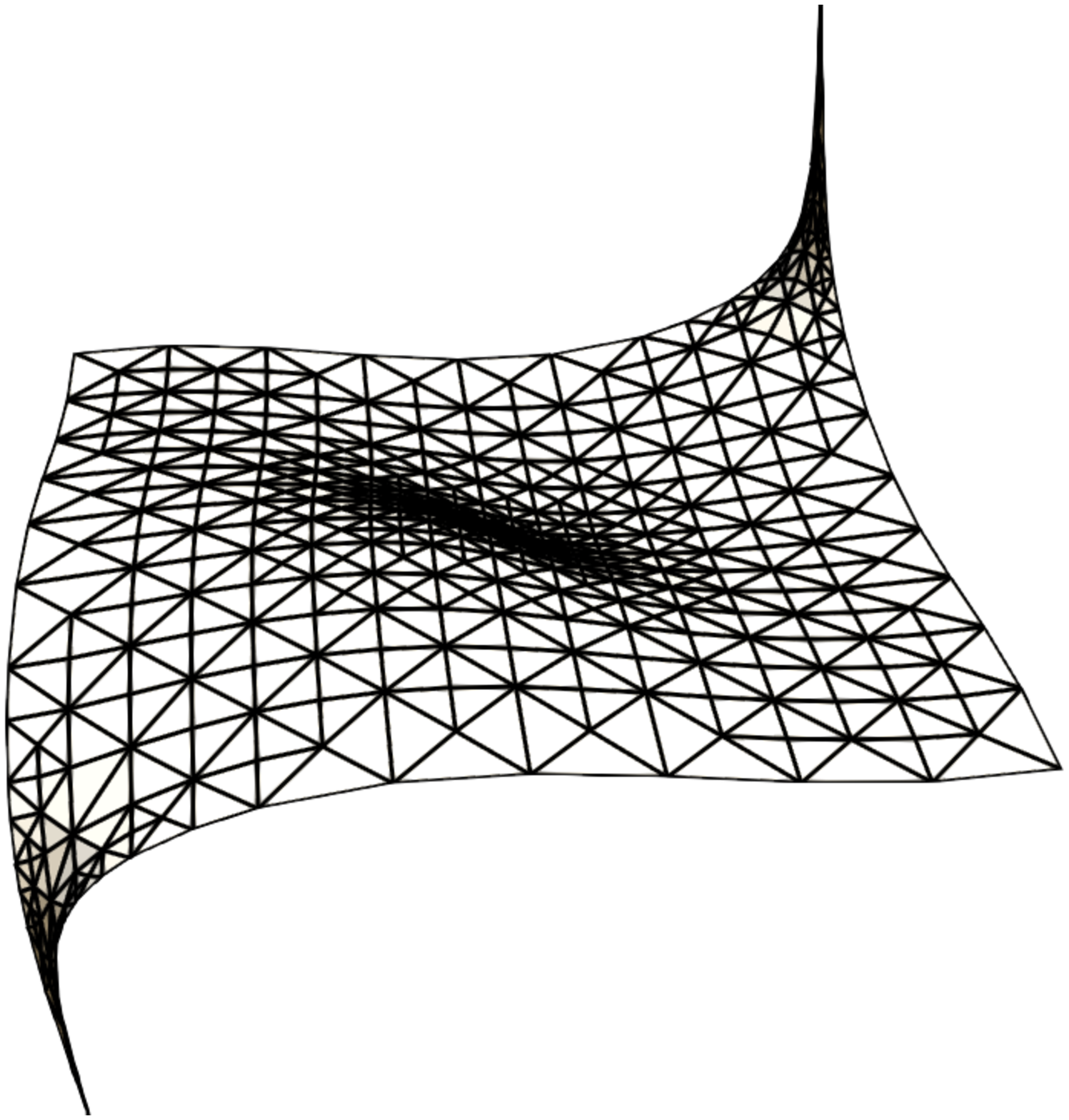} \\
\qquad \tiny{(D.5)}
\end{minipage}
\begin{minipage}[b]{0.327\textwidth}\centering  
\includegraphics[width=3.6cm,height=3.6cm,scale=0.66]{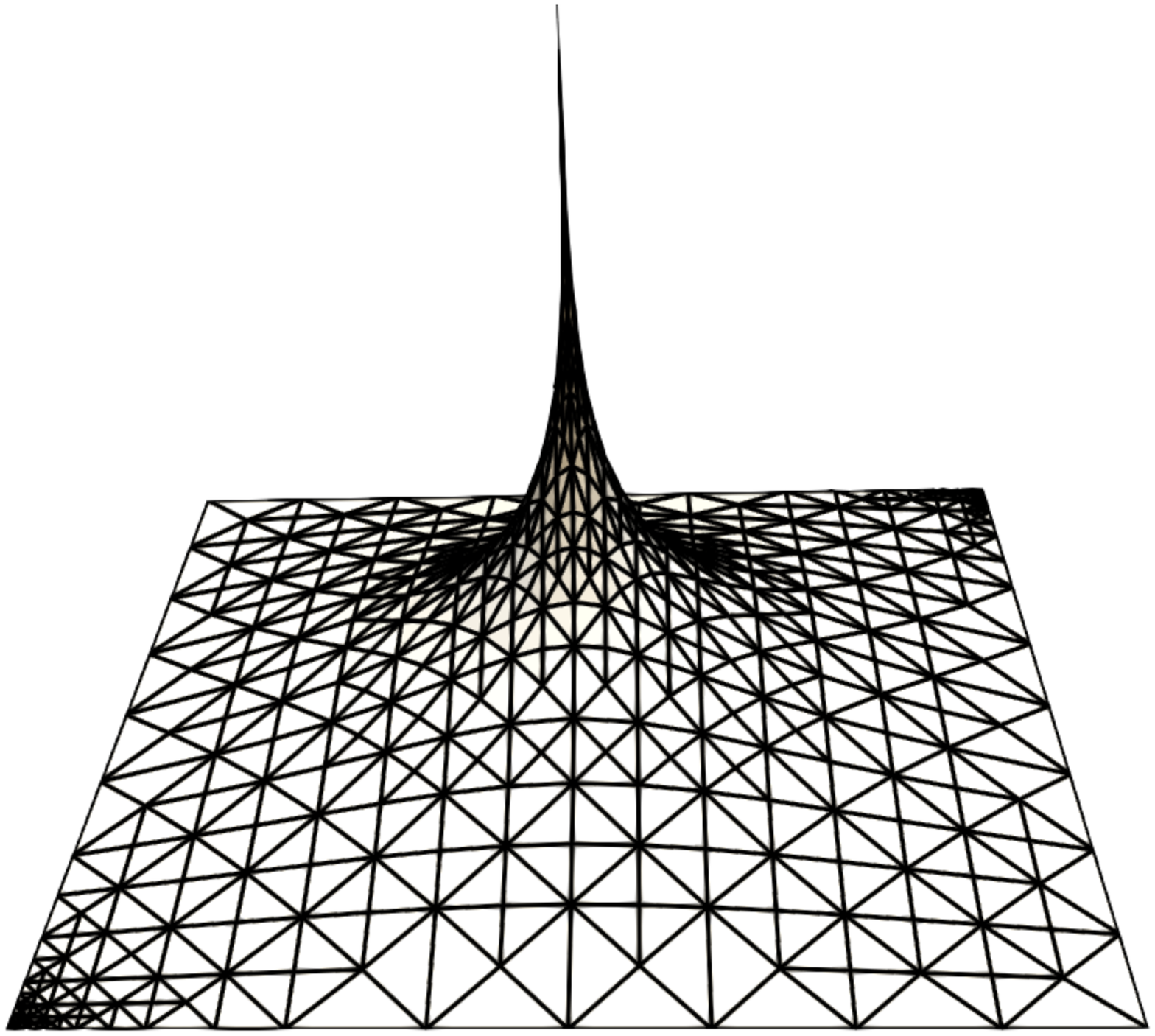} \\
\qquad \tiny{(D.6)}
\end{minipage}
\\~\\
\begin{minipage}[b]{0.327\textwidth}\centering
\includegraphics[width=3.6cm,height=3.6cm,scale=0.65]{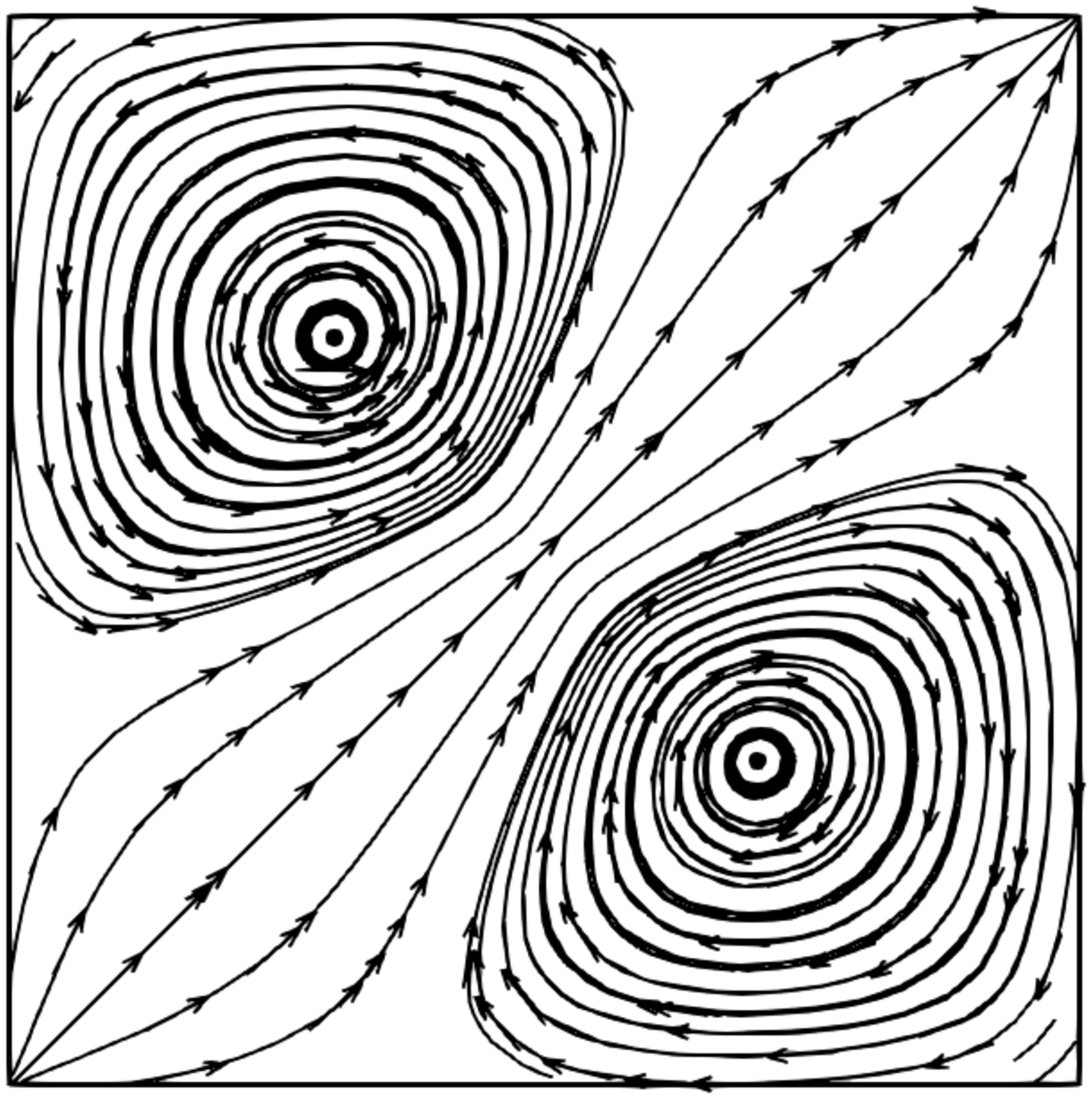} \\
\qquad \tiny{(D.7)}
\end{minipage}
\begin{minipage}[b]{0.327\textwidth}\centering 
\psfrag{Pressure}{\LARGE Pressure}
\includegraphics[width=3.6cm,height=3.85cm,scale=0.76,trim={1cm 2cm 2cm 0},clip]{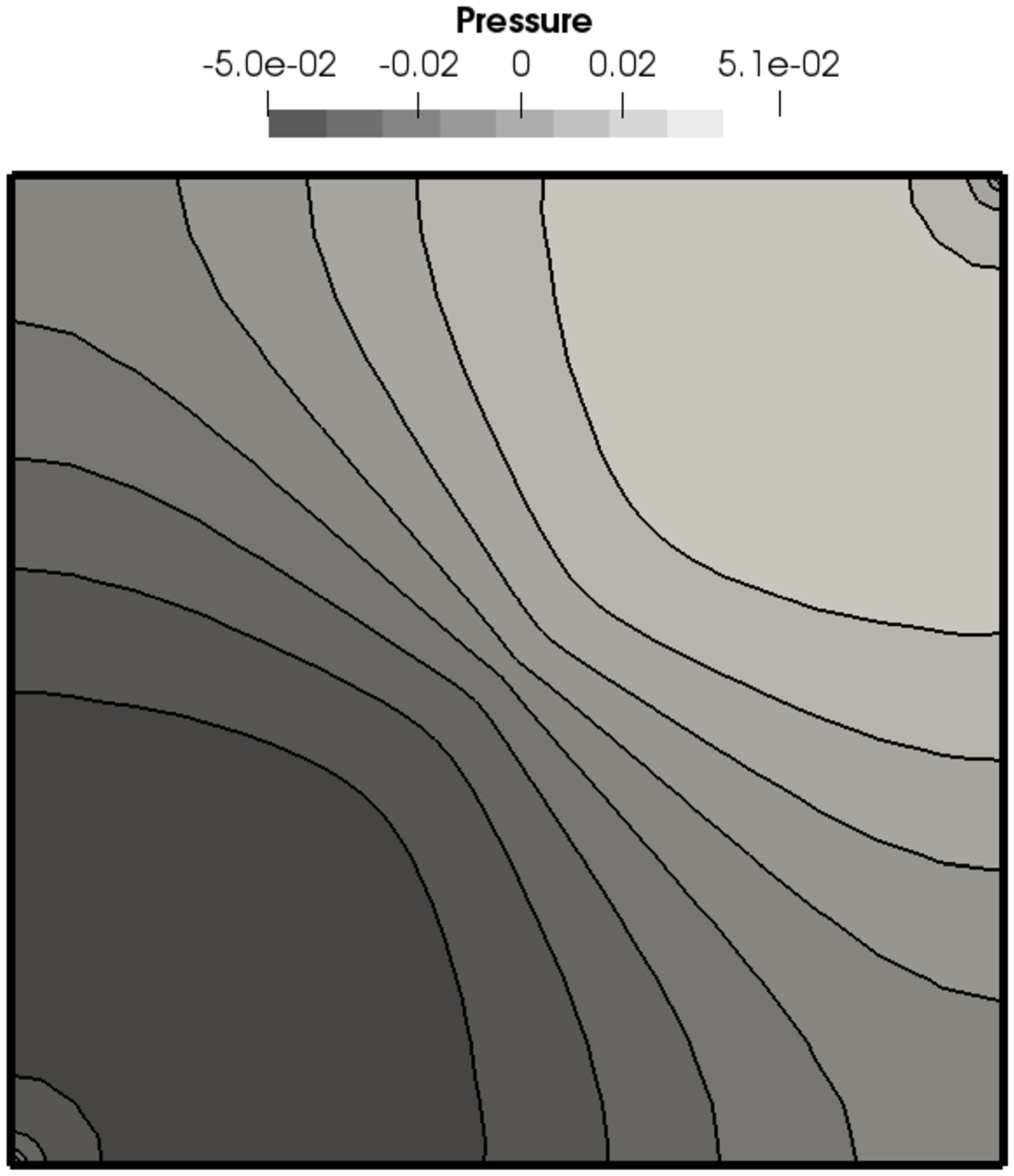} \\
\qquad \tiny{(D.8)}
\end{minipage}
\begin{minipage}[b]{0.327\textwidth}\centering  
\includegraphics[width=3.6cm,height=3.85cm,scale=0.76,trim={1cm 2cm 2cm 0}]{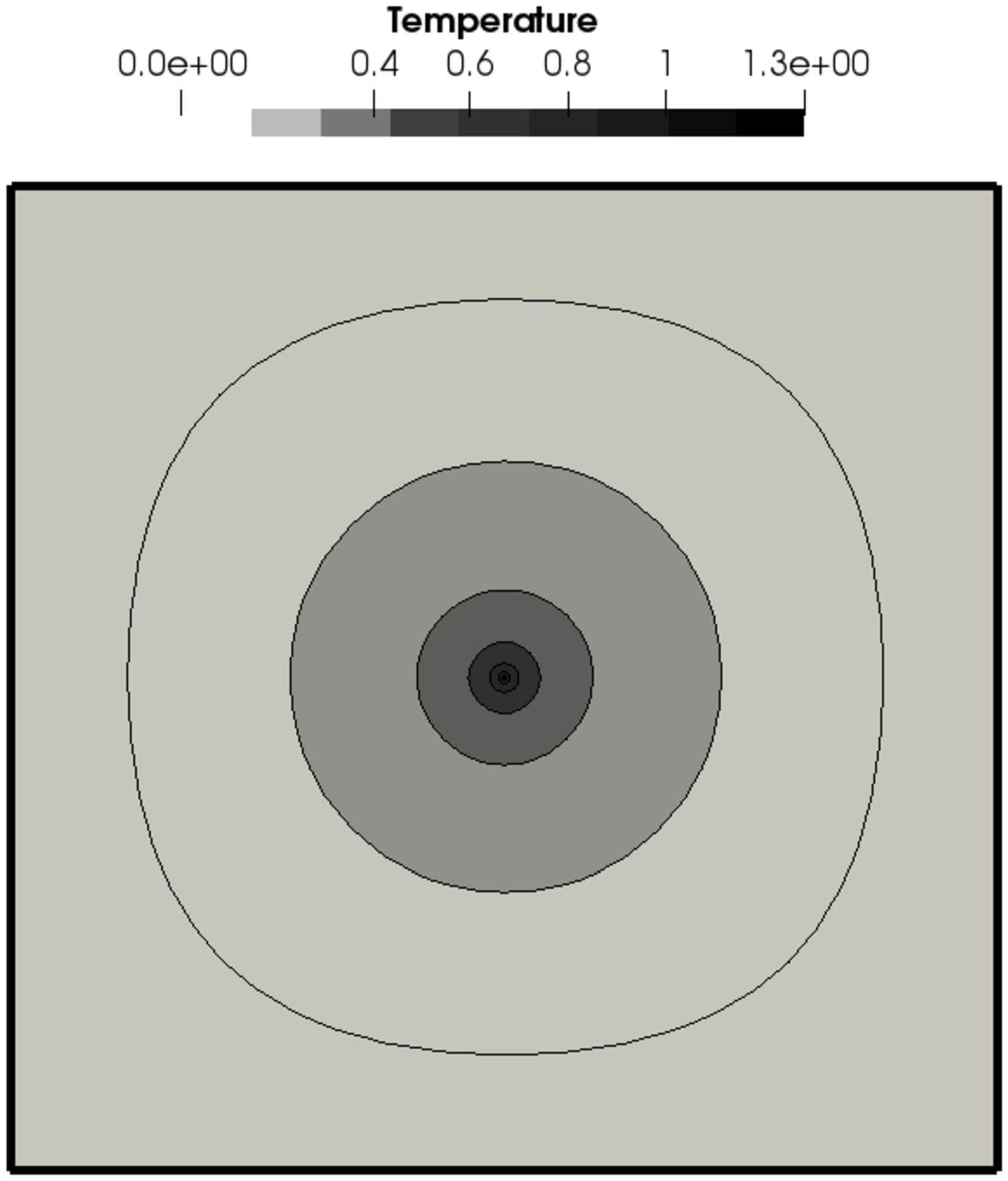} \\
\qquad \tiny{(D.9)}
\end{minipage}

\caption{Example 3: The mesh obtained after 30 iterations of our adaptive loop for $p=1.0$ (230 vertices and 414 elements) (D.1), the mesh obtained after 60 iterations for $p=1.0$ (1979 vertices and 3832 elements) (D.2), experimental rate of convergence for $\mathsf{E}_{\T}$ with $p=1.0$ (D.3), elevation plots for the velocity norm (D.4), pressure field (D.5), and temperature field (D.6), streamlines of the velocity field (D.7), and contour lines of the pressure field (D.8), and temperature field (D.9).} 
\label{fig:test_02}
\end{figure}

\bibliographystyle{siamplain}
\bibliography{biblio}
\end{document}